\documentclass[11pt,a4paper,reqno]{amsart}
\usepackage{a4wide}
\usepackage{amsaddr}
\usepackage{amsmath,amsfonts, amssymb, amsthm}
\usepackage{paralist}
\usepackage{enumitem}
\usepackage{graphics,graphicx,color}
\usepackage[colorlinks=true]{hyperref}
\hypersetup{urlcolor=blue, citecolor=red}

\makeatletter
\renewcommand{\email}[2][]{%
	\ifx\emails\@empty\relax\else{\g@addto@macro\emails{,\space}}\fi%
	\@ifnotempty{#1}{\g@addto@macro\emails{\textrm{(#1)}\space}}%
	\g@addto@macro\emails{#2}%
}
\makeatother

\newtheorem{theorem}{Theorem}[section]
\newtheorem{corollary}[theorem]{Corollary}
\newtheorem{lemma}[theorem]{Lemma}

\theoremstyle{definition}

\numberwithin{equation}{section}

\newcommand{\R}{{\mathbb R}}

\begin{document}
\title[Nonlinear Kirchhoff equation with critical combined powers nonlinearity]
{Asymptotic profiles for a nonlinear Kirchhoff equation with combined powers nonlinearity}
\author{Shiwang Ma}\email{shiwangm@nankai.edu.cn}
\address{School of Mathematical Sciences and LPMC, Nankai University\\ 
	Tianjin 300071, China}
\author{Vitaly Moroz}\email{v.moroz@swansea.ac.uk}
\address{Department of Mathematics, Swansea University\\ 
	Fabian Way, Swansea SA1 8EN, 	Wales, UK}




\date{\today}

\begin{abstract}
We study asymptotic behavior of positive ground state solutions of the nonlinear Kirchhoff equation
\begin{equation*}
-\big(a+b\int_{\mathbb R^N}|\nabla u|^2\big)\Delta u+ \lambda u= u^{q-1}+ u^{p-1} 
\quad {\rm in} \ \mathbb R^N,
 \eqno(P_\lambda)
\end{equation*}
as $\lambda\to 0$ and $\lambda\to +\infty$, where $N=3$ or $N= 4$, $2<q\le p\le 2^*$, $2^*=\frac{2N}{N-2}$ is the Sobolev critical exponent,  $a>0$, $b\ge 0$ are constants and 
$\lambda>0$ is  a parameter.  
In particular, we prove that in the case $2<q<p=2^*$,  as $\lambda\to 0$, after {\em a suitable  rescaling} the ground state solutions of $(P_\lambda)$ converge  to the unique positive solution of the equation 
$-\Delta u+u=u^{q-1}$  and as $\lambda\to +\infty$, after {\em another rescaling} the ground state solutions of $(P_\lambda)$ converge to a particular solution of the critical Emden-Fowler equation $-\Delta u=u^{2^*-1}$. We establish a sharp asymptotic characterisation of such rescalings, which depends in a non-trivial way on the space dimension $N=3$ and  $N= 4$.  We also discuss a connection of our results with a mass constrained problem associated to $(P_\lambda)$ with normalization constraint $\int_{\mathbb R^N}|u|^2=c^2$. As a consequence of the main results, we obtain the existence, non-existence  and asymptotic behavior of positive normalized solutions of such a problem. In particular, we obtain the exact number and their precise asymptotic expressions of normalized 
solutions if  $c>0$ is sufficiently large or sufficiently small. Our results also show that in the space dimension $N=3$, there is a striking difference between the cases $b=0$ and $b\not=0$. More precisely, if $b\not=0$, then both $p_0:= 10/3$ and $p_b:= 14/3$ play a role  in the existence, non-existence, the exact number  and  asymptotic behavior of the normalized solutions of the mass constrained problem, which is completely different from those
for the corresponding nonlinear Schr\"odinger equation and which reveals the special influence of the nonlocal term.

  \end{abstract}

\maketitle

\newpage
\section{Introduction and notations}

We consider the following Kirchhoff equation
\begin{equation*}
-\big(a+b\int_{\mathbb R^N}|\nabla u|^2\big)\Delta u+ \lambda u= |u|^{q-2}u+ |u|^{p-2}u 
\quad\text{in $\R^N$},\eqno{(P_\lambda)}
\end{equation*}
where $N\ge 3$,  $2<q\le p\le 2^*=\frac{2N}{N-2}$, $a>0$, $b\ge 0$ and $\lambda>0$ are parameters. 
For a fixed $\lambda>0$, the corresponding to $(P_\lambda)$ action functional is given by
\begin{equation}\label{e11}
I_{\lambda}(u)=\frac{a}{2}\int_{\mathbb R^N}|\nabla u|^2+\frac{\lambda}{2}\int_{\mathbb R^N}|u|^2+\frac{b}{4}\Big(\int_{\mathbb R^N}|\nabla u|^2\Big)^2-\frac{1}{q}\int_{\mathbb R^N}|u|^q-\frac{1}{2^*}\int_{\mathbb R^N}|u|^{2^*},
\end{equation}
and critical points of $I_\lambda$ in $H^1(\R^N)$ correspond to  solutions of $(P_{\lambda})$.  By a ground state solution of $(P_{\lambda})$ we understand a solution  $u_\lambda\in H^1(\mathbb  R^N)$ such that $I_{\lambda}(u_\lambda)\le I_{\lambda}(u)$ for every nontrivial solution $u$ of $(P_{\lambda})$.
 
In this paper we are interested in the limit asymptotic profile of the ground sates $u_\lambda$ of the  problem $(P_{\lambda})$, and  in the precise  asymptotic behavior of different norms of $u_\lambda$, as $\lambda\to 0$ and $\lambda\to\infty$.
Of particular importance is the $L^2$--mass of the groundstates
\begin{equation}\label{e12}
M(\lambda):=\|u_\lambda\|_2^2,
\end{equation}
which plays a key role in the analysis of stability of the standing wave solution of the time--dependent NLS , cf.~Lewin and Nodari \cite[Section 3.2]{Lewin-1} for a discussion in the context of the local combined power NLS.

For the local prototype of $(P_\lambda)$ with $b=0$ the asymptotic profiles of the ground sates were studied in \cite{Akahori-3, Akahori-2}, where the authors considered the following nonlinear Schr\"odinger equation with focusing combined powers nonlinearity:
\begin{equation}\label{e13}
-\Delta u+\lambda u=|u|^{q-2}u+|u|^{p-2}u, \quad {\rm in} \ \mathbb R^N,
\end{equation}
where $N\ge 3$ and $2<q<p\le 2^*$. When  $p=2^*$ and $q\in (2+\frac{4}{N},2^*)$,   Akahori et al.  in \cite{Akahori-3}   proved that for small $\lambda>0$ the positive ground state of \eqref{e13} is unique and non-degenerate,  and as $\lambda\to 0$ the unique positive ground state $u_\lambda$ converges after an explicit rescaling to the unique positive solution of the limit equation $-\Delta u+u=u^{q-1}$ in $\R^N$. In \cite{Akahori-2}, after a suitable implicit rescaling the authors establish a uniform decay estimate for the positive ground states $u_\lambda$, and then  prove  the uniqueness and nondegeneracy of ground states $u_\lambda$  for $N\ge 5$ and large $\lambda>0$,  and show that for $N\ge 3$, as $\lambda\to\infty$, $u_\lambda$ converges to a particular solution of the critical Emden--Fowler equation.    
Recently, for $p=2^*$,  Coles and Gustafson \cite{Coles} proved that the radial ground 
state $u_\lambda$ is also unique and non-degenerate for all large $\lambda>0$ when $N = 3$ and $q\in (4, 2^* )$. 
In \cite{MM-1}, the authors  studied a related problem and its connection with a mass constrained problem by using a rescaling argument and the concentration--compactness principle. See also \cite{MM-2} for a nonlinear Choquard type equation.

The techniques in this work (as well as in \cite{MM-1,MM-2}) is inspired by \cite{Moroz-1}, where the second author and  C. Muratov studied the asymptotic properties of ground states for a combined powers Schr\"odinger equations with a focusing exponent $p>2$ and a defocusing exponent $q>p$, 
\begin{equation}\label{e14}
-\Delta u+\lambda u=|u|^{p-2}u-|u|^{q-2}u, \quad {\rm in} \ \mathbb R^N,
\end{equation}
and obtained a sharp asymptotic characterisation of the limit profiles of positive ground states $u_\lambda$ of \eqref{e14} as $\lambda\to 0$. Later, in \cite{Lewin-1},  M. Lewin and S. Rota Nodari proved a general result about the uniqueness and non-degeneracy of positive radial solutions of \eqref{e14}. The non-degeneracy of the unique positive solution allowed them to refine the asymptotic results in \cite{Moroz-1} and, amongst other things, to establish the exact asymptotic behavior of $M(\lambda)=\|u_\lambda\|_2^2$. In particular, this implied the uniqueness of normalised energy minimizers at fixed masses in certain regimes. See also \cite{Liu-1}, where Zeng Liu and the second author extended the results in \cite{Moroz-1} to a class of Choquard type equation. 
 
In the present paper, we study the limit asymptotic profiles of the ground sates $u_\lambda$ of the Kirchhoff problem $(P_{\lambda})$ by using a rescaling argument and the concentration--compactness principle,  and obtain an explicit asymptotic expression of different norms of ground states for fixed frequency problem $(P_\lambda)$. To do so, we adapt the technique developed in \cite{MM-1}. However, additional difficulties arise since $(P_\lambda)$ contains five terms and as a consequence, the Poho\v zaev--Nehari algebraic relations can not be resolved, in general.  Fortunately, we succeed to overcome this difficulty in the case $p=2^*$, but the method does not work any more for $p<2^*$. In the latter case, we shall use a suitable scaling to reduce $(P_\lambda)$ to the  local equation \eqref{e13} (cf. \eqref{e52}).  The disadvantage of such a rescaling is that for $b\not=0$, the rescaled family of  ground states  for $(P_\lambda)$ should not necessarily be a ground state for \eqref{e13}.  Besides, to obtain a precise estimate of the least  energy, the scaling should transform a ground state for $(P_\lambda)$  into a ground state of the local equation \eqref{e13}. 
Generally speaking, this is not the case, which prevents us from deriving a precise energy estimate of the ground state $p<2^*$, see Section 5 for a discussion.

 Alternatively to the study of {\em fixed frequency} solutions of $(P_\lambda)$, one can search for solutions to $(P_\lambda)$ with a prescribed mass, that is for a fixed $c>0$ to search for $u\in H^1(\mathbb R^N)$ and $\lambda\in \mathbb R$ that satisfy
\begin{equation}\label{e15}
\left\{\begin{array}{rl}
&-(a+b\int_{\mathbb R^N}|\nabla u|^2)\Delta u+\lambda u=\mu |u|^{q-2}u+|u|^{p-2}u \  \ {\rm in}  \  \mathbb R^N,\medskip\\
&\quad u\in H^1(\mathbb R^N),  \   \    \   \int_{\mathbb R^N}|u|^2=c^2,
\end{array}\right.
\end{equation}
where $\mu>0$ is a new parameter and the frequency $\lambda\in \mathbb R$ becomes a part of the unknown.
 The solutions of \eqref{e15} are usually denoted by a pair $(u,\lambda)$ and referred to as {\em normalized} solutions. Normalised solutions can be obtained by searching critical points of the energy functional 
\begin{equation}\label{e16}
E_\mu(u)=\frac{a}{2}\int_{\mathbb R^N}|\nabla u|^2+\frac{b}{4}\Big(\int_{\mathbb R^N}|\nabla u|^2\Big)^2-\frac{\mu}{q}\int_{\mathbb R^N}|u|^q-\frac{1}{p}\int_{\mathbb R^N}|u|^{p}, 
\end{equation}
subject to the constraint 
$$
S_c:=\{ u\in H^1(\mathbb R^N): \ \int_{\mathbb R^N}|u|^2=c^2\},
$$
where $\lambda\in \mathbb R$ appears a posteriori as a Lagrange multipliers. 
 
In the local case $b=0$,  by rescaling  we also  assume $a=1$. Then equation $(P_\lambda)$  reduces to the classical non linear Schr\"odinger equation
\begin{equation}\label{e17}
-\Delta u +\lambda u =\mu |u|^{q-2}u+|u|^{p-2}u, \ \  {\rm in} \  \mathbb R^N.
\end{equation}
Normalized solutions of \eqref{e17} were studied by T. Cazenave and P.-L. Lions \cite{Cazenave-1}, N. Soave \cite{Soave-1,Soave-2}, L. Jeanjean et al. \cite{Jeanjean-2, Jeanjean-4},  L. Jeanjean and T.  Le \cite{Jeanjean-3}. 
The additional  parameter  $\mu>0$ is often introduced  to control the unknown  Lagrange multipliers $\lambda\in \mathbb R$.   Some of the results on normalized solutions to \eqref{e17} are summarized in \cite{Li-1}. The asymptotic behavior of normalized solution as $\mu$ varies in its range is studied in \cite{Jeanjean-3, Soave-1,  Soave-2, Wei-1, Wei-2}. We mention that in the case $N\ge 4$, $2<q<2+\frac{4}{N}$ and $p=2^*$,  L. Jeanjean and T.  Le \cite{Jeanjean-3} obtained a normalized solution $u_c$ of  mountain pass type for small $c>0$ and proved that 
$\lim_{c\to 0}\|\nabla u_c\|_2^2=S^{N/2}$ and $\lim_{c\to 0}E_\mu(u_c)=\frac{1}{N}S^{N/2}$, where $S$ is the best Sobolev constant.

In  the above results, the number $p_0:=2+\frac{4}{N}$, called the $L^2$-critical exponent,  is crucial for the existence and asymptotic behavior of normalized solutions of \eqref{e17}.   
 However, it is showed in \cite{Ye-1,Ye-2} that when $b\not=0$,
 $p_b:=2+\frac{8}{N}$ is the $L^2$-critical exponent for the minimization problem
\begin{equation}\label{e18}
E_\mu(c):=\inf_{u\in S_c}E_\mu(u).
\end{equation}
 in the sense that for each $c > 0$, $E_\mu(c)>-\infty$ if $2 <p< p_b$ and $E_\mu(c)=-\infty$ if $p_b <
p < 2^*$.  In particular, when $N=3$,  the $L^2$-critical exponent for the problem \eqref{e15} is given by 
 \begin{equation}\label{e19}
 p_b=\left\{\begin{array}{lcl}
 10/3, \   \  {\rm if} \ b=0,\\
 14/3, \   \   {\rm if} \  b\not=0.\end{array}\right.
 \end{equation}
 In \cite{Li-1}, Li, Luo and Yang consider the existence and multiplicity of normalized solutions of \eqref{e15} when $N=3$ and prove that if  $2<q<\frac{10}{3}$ and $\frac{14}{3}<p<6$, then for small $\mu>0$, $E_\mu|_{S_c}$ has a local minimizer at a negative energy level $m(c,\mu)<0$, and has a second critical point of mountain pass type at a positive energy level $\sigma(c,\mu)>0$. If $2<q<\frac{10}{3}<p=6$, then for small  $\mu>0$ a ground state solution is obtained. If $\frac{14}{3}<q<p\le 6$, then for any $\mu>0$, a critical point of  mountain pass type is also obtained. Furthermore,  as the parameter $\mu\to 0^+$, the asymptotic behavior of energy $m(c,\mu)$ and the normalized solution is also investigated.  To our  best  knowledge, the existence of normalized solutions to \eqref{e15} with $\frac{10}{3}<q<p<\frac{14}{3}$ is still unknown.

 When $\mu=0$, then the equation $(P_\lambda)$  reduces to the following Kirchhoff equation with a homogeneous nonlinearity
 \begin{equation}\label{e110}
-\big(a+b\int_{\mathbb R^N}|\nabla u|^2\big)\Delta u +\lambda u =|u|^{p-2}u, \ \  {\rm in} \  \mathbb R^N,
\end{equation}
Amongst other things,  Li and Ye \cite{Ye-2} studied the existence and concentration behavior of minimizers for \eqref{e110} and obtained precise asymptotic behavior  of normalized solutions to  \eqref{e110} with $2<p<2+\frac{8}{N}$ as $c\to +\infty$. For $2<p<2+\frac{4}{N}$, Zeng and Zhang \cite{Zeng-1} proved the existence and uniqueness of the minimizer to the minimization problem
$ E_0(c):=\inf_{u\in S_c}E_0(u)$ for any $c>0$, while  for  $2+\frac{4}{N}\le p<2+\frac{8}{N}$  the authors proved that there exists a threshold mass $c_*>0$ such that for any $c\in (0,c_*)$  there is no  minimizer and for $c>c_*$ there is  a unique minimizer.  Moreover, a precise formula for the minimizer and the threshold value $c_*$ is given according to the mass $c$.  In the case $2<p< 2^*$, Qi and Zou \cite{Qi-1}  obtain the exact number and expressions of the
positive normalized solutions for \eqref{e110} and then answer an open problem
concerning the exact number of positive solutions to the Kirchhoff equation with fixed frequency.  Recently, Qihan He et al.  \cite{He-1} studied  the existence and asymptotic behavior of normalized solutions of \eqref{e110} with  $|u|^{p-2}u$ replaced by a general subcritical nonlinearity $g(u)$ of mass super-critical type. In particular, $g(u)$ contains the nonlinearity in \eqref{e15} with $2+\frac{8}{N}<q\le p<2^*$ as a special case. Under some suitable assumptions,  they obtain the existence of ground state normalized solutions for any given $c >0$. After a detailed analysis via the blow up method, they also described the asymptotic behavior of these solutions as $ c \to  0^+$, as well as $c \to +\infty$.  

If  both $b\not=0$ and $\mu\not=0$,  then $(P_\lambda)$ become a nonlocal equation with a non-homogeneous nonlinearity. It is much more challenging and interesting to investigate the existence and qualitative properties of solutions of \eqref{e15}. Some existence results concerning the normalized solutions of \eqref{e15} have been obtained over  the past few years, see \cite{Li-1, Zhang-1} and reference therein for Kirchhoff equations with combined powers nonlinearity.  
 However, less progress is made on the asymptotic behavior of these solutions whenever the associated parameter varies in a suitable range.  The technique used in \cite{Qi-1,Ye-2, Zeng-1, Zeng-2} for the asymptotic study of equation \eqref{e110} is not applicable any more, and any explicit expression of normalized solution in terms of the mass $c$ is not available for the nonlocal problem $(P_\lambda)$.

Our main purpose in this paper is to study the effect of the nonlocal term in the case $b\neq 0$ on the existence, non-existence, multiplicity and properties of normalized solutions of \eqref{e15}  and to understand the role of the $L^2$-critical exponent $p_b$ in  the existence and asymptotic behavior of normalized  solutions of \eqref{e15}  as the parameter $c$ varies. As a direct consequence of our main results on the fixed frequency problem $(P_\lambda)$ in this paper, we are able to obtain an explicit asymptotic expression of different norms of positive normalized  solutions,  and to give  a complete description on the  existence, multiplicity and precise asymptotic behavior  of positive normalized solutions of \eqref{e15}. In particular, we  prove that both $p_0=\frac{10}{3}$ and $p_b=\frac{14}{3}$ play a key role in the existence,  multiplicity and the asymptotic behavior of normalized solutions of \eqref{e15} if $b\not=0$.
\smallskip

\noindent
\textbf{Organization of the paper.} In Section 2, we state  the main results in this paper. In Section 3, we give a proof of Theorem 2.1 for small $\lambda>0$. Section 4 is  devoted to the proof of Theorems 2.1 for large $\lambda>0$.  In Section 5, we prove Theorem 2.2, and in the last section, as an application of our main results, we present some results concerning the existence, non-existence, and exact number of normalized solutions of the associated mass constrained problem.

\smallskip

\noindent
\textbf{Basic notations.} Throughout this paper, we assume $N\geq
3$. 

\begin{itemize}[leftmargin=2em]
\item $ C_c^{\infty}(\mathbb{R}^N)$ denotes the space of smooth functions with compact support in $\mathbb{R}^N$.
\smallskip

\item $L^p(\mathbb{R}^N)$ with $1\leq p<\infty$ is the Lebesgue space
with the norm $\|u\|_p=\left(\int_{\mathbb{R}^N}|u|^p\right)^{1/p}$.
\smallskip
 
\item $ H^1(\mathbb{R}^N)$ is the usual Sobolev space with the norm
$\|u\|_{H^1(\mathbb{R}^N)}=\left(\int_{\mathbb{R}^N}|\nabla u|^2+|u|^2\right)^{1/2}$. 
\smallskip

\item $H_r^1(\mathbb{R}^N)=\{u\in H^1(\mathbb{R}^N):
u\  \mathrm{is\ radially \ symmetric}\}$. 
\smallskip

\item $ D^{1,2}(\mathbb{R}^N)=\{u\in
L^{2^*}(\mathbb{R}^N): |\nabla u|\in
L^2(\mathbb{R}^N)\}$. 
\smallskip

\item
For any $q\in (2,2^*)$ where $2^*=\frac {2N}{N-2}$, we define 
\begin{equation}\label{e111}
S:=\inf_{u\in D^{1,2}(\mathbb R^N)\setminus \{0\}}\frac{\int_{\mathbb R^N}|\nabla u|^2}{\left(\int_{\mathbb R^N}|u|^{2^*}\right)^{\frac{2}{2^*}}},\quad 
S_q=\inf_{u\in H^1(\mathbb R^N)\setminus \{0\}}\frac{\int_{\mathbb R^N}|\nabla u|^2+|u|^2}{(\int_{\mathbb R^N}|u|^q)^{\frac{2}{q}}}.
\end{equation}

\item
$B_r$ denotes the ball in $\mathbb R^N$ with radius $r>0$ and centered at the origin,  $|B_r|$ and $B_r^c$ denote its Lebesgue measure  and its complement in $\mathbb R^N$, respectively.  

\item
As usual, $C$, $c$, etc., denote generic positive constants.
\end{itemize}

\noindent
\textbf{Asymptotic notations.}
For $\lambda>0$ and nonnegative functions $f(\lambda)$ and  $g(\lambda)$, we write:
\smallskip

(1)  $f(\lambda)\lesssim g(\lambda)$ or $g(\lambda)\gtrsim f(\lambda)$ if there exists a positive constant $C$ independedent of $\lambda$ such that $f(\lambda)\le Cg(\lambda)$.
\smallskip

(2) $f(\lambda)\sim g(\lambda)$ if $f(\lambda)\lesssim g(\lambda)$ and $f(\lambda)\gtrsim g(\lambda)$.
\smallskip

\noindent
If $|f(\lambda)|\lesssim |g(\lambda)|$, we write $f(\lambda)=O((g(\lambda))$. We also denote by $\Theta=\Theta(\lambda)$ a  generic  positive function satisfying
$
C_1\lambda \le \Theta(\lambda)\le C_2\lambda
$
for some positive numbers $C_1,C_2>0$, which are independent of $\lambda$. Finally, if $\lim f(\lambda)/g(\lambda)=1$ as $\lambda\to \lambda_0$, then we write $f(\lambda)\simeq g(\lambda)$ as $\lambda\to \lambda_0$.

\section{Main results}

The existence of ground state solutions established in  \cite{Lu-1} for Kirchhoff equations with a general nonlinearity and in \cite[Theorem 2.1]{Xie-1} for Kirchhoff equations with critical nonlinearites applies  to $(P_\lambda)$  directly or after a suitable scaling. 
 In this paper, we are interested in the asymptotic behavior of ground state solutions of $(P_\lambda)$.  Our main results are the following two theorems. In the first one, we consider $(P_\lambda)$ with $p=2^*$ in dimensions $N=3,4$. In the second one, we consider $(P_\lambda)$ in dimension $N=3$.

\begin{theorem} \label{t1}
 Let $p=2^*$ and $\{u_\lambda\}$ be a family ground states of $(P_\lambda)$. 
If $N=3, 4$, and $q\in(2,2^*)$, then for small $\lambda>0$, $u_\lambda$ satisfies
$$
u_\lambda(0)= \lambda^{\frac{1}{q-2}}(V_0(0)+o(1)),
$$
$$
\|\nabla u_\lambda\|^2_2=\lambda^{\frac{2N-q(N-2)}{2(q-2)}}\left\{\frac{N(q-2)}{2q}a^{\frac{N-2}{2}}S_q^{\frac{q}{q-2}}+O(\lambda^{\frac{2N-q(N-2)}{2(q-2)}})\right\},
$$
$$
\| u_\lambda\|^2_2=\lambda^{\frac{4-N(q-2)}{2(q-2)}}\left\{\frac{2N-q(N-2)}{2q}a^{\frac{N}{2}}S_q^{\frac{q}{q-2}}+O(\lambda^{\frac{2N-q(N-2)}{2(q-2)}})\right\},
$$
$$
\|u_\lambda\|_q^q=\lambda^{\frac{2N-q(N-2)}{2(q-2)}}\left\{a^{\frac{N}{2}}S_q^{\frac{q}{q-2}}+o(1)\right\}.
$$
$$
\|u_\lambda\|_{2^*}^{2^*}=\lambda^{\frac{N[2N-q(N-2)]}{2(N-2)(q-2)}}\left\{a^{\frac{N}{2}}\|V_0\|_{2^*}^{2^*}+o(1)\right\}.
$$
and the least energy $m_\lambda$ of the ground state satisfies
\begin{equation}\label{e21}
m_\lambda=I_\lambda(u_\lambda)=\lambda^{\frac{2N-q(N-2)}{2(q-2)}}\left\{\frac{q-2}{2q}a^{\frac{N}{2}}S_q^{\frac{q}{q-2}}+O(\lambda^{\frac{2N-q(N-2)}{2(q-2)}})\right\}.
\end{equation}
Moreover, for small $\lambda>0$, the rescaled family of ground states 
\begin{equation}\label{e22}
v_\lambda(x)=\lambda^{-\frac{1}{q-2}}u_\lambda\big(\lambda^{-\frac{1}{2}}x\big)
\end{equation}
satisfies 
$$
\|\nabla v_\lambda\|^2_2\sim \|v_\lambda\|_{2^*}^{2^*} \sim\|v_\lambda\|^q_q\sim \|v_\lambda\|_2^2\sim 1,
$$
and as $\lambda\to 0$, $v_\lambda$ converges in $H^1(\mathbb R^N)$ to $v_0$, where $v_0(x)=V_0(\frac{x}{\sqrt a})$ and $V_0$ is the unique positive solution of the equation
$$
-\Delta V +V=V^{q-1}\quad\text{in $\R^N$}.
$$
If $N=4$, $q\in (2,4)$ and $bS^2<1$ or $N=3$ and $q\in (4,6)$, then for large  $\lambda>0$, $u_\lambda$ satisfies 
$$
u_\lambda(0)\sim\left\{
\begin{array}{lcl}
(\lambda\ln\lambda)^{\frac{2}{q-2}},
&\text{if}&N=4,\smallskip\\
\lambda^{\frac{1}{2(q-4)}}, &\text{if}&N=3,
\end{array}\right.
$$
$$
\frac{aS^2}{1-bS^2}-\|\nabla u_\lambda\|^2_2\sim \left(\lambda\ln\lambda\right)^{-\frac{4-q}{q-2}},   \   \  \text{if} \  \   N=4,\smallskip
$$
$$
\frac{bS^{3}+S^{\frac{3}{2}}\sqrt{b^2S^3+4a}}{2}-\|\nabla u_\lambda\|^2_2\sim \lambda^{-\frac{6-q}{2(q-4)}}, \  \   \   \text{if}  \  \  N=3,
$$
$$
 \|u_\lambda\|_{2^*}^{2^*}=\left\{
\begin{array}{lcl}
\frac{a^2S^2}{(1-bS^2)^2}+O(\left(\lambda\ln\lambda\right)^{-\frac{4-q}{q-2}}), &\text{if}&N=4,\smallskip\\
 \frac{1}{8}(bS^{2}+S^{\frac{1}{2}}\sqrt{b^2S^3+4a})^3+O(\lambda^{-\frac{6-q}{2(q-4)}}), &\text{if}&N=3,
 \end{array}
 \right.
$$
$$
\|u_\lambda\|_2^2\sim\left\{
\begin{array}{lcl}
\lambda^{-\frac{2}{q-2}}(\ln\lambda)^{-\frac{4-q}{q-2}}, &\text{if}&N=4,\smallskip\\
\lambda^{-\frac{q-2}{2(q-4)}}, &\text{if}&N=3,
\end{array}\right.
$$
$$
\|u_\lambda\|_q^q\sim\left\{
\begin{array}{lcl}
(\lambda\ln\lambda)^{-\frac{4-q}{q-2}}, &\text{if}&N=4,\smallskip\\
\lambda^{-\frac{6-q}{2(q-4)}}, &\text{if}&N=3.
\end{array}\right.
$$
Moreover,  there exists $\zeta_\lambda\in (0,+\infty)$ verifying  
$$
\zeta_\lambda\sim\left\{
\begin{array}{lcl}
(\lambda\ln\lambda)^{-\frac{1}{q-2}}, &\text{if}&N=4,\smallskip\\
\lambda^{-\frac{1}{q-4}},&\text{if}&N=3,
\end{array}\right.
$$
such that for large $\lambda>0$, the rescaled family of ground states
\begin{equation}\label{e23}
w_\lambda(x)=\zeta_\lambda^{\frac{N-2}{2}}u_\lambda(\zeta_\lambda x)
\end{equation}
satisfies 
$$
\|\nabla w_\lambda\|^2_2\sim \|w_\lambda\|_{2^*}^{2^*}\sim \|w_\lambda\|_q^q\sim 1, \quad 
\  \|w_\lambda\|_2^2\sim \left\{
\begin{array}{lcl}
\ln\lambda, &\text{if}&N=4,\smallskip\\
\lambda^{\frac{6-q}{2(q-4)}},&\text{if}&N=3,
\end{array}\right.
$$
and as $\lambda\to \infty$, $w_\lambda$ converges in $D^{1,2}(\mathbb R^N)$ and  $L^q(\mathbb R^N)$  to $w_\infty$, where 
$w_\infty(x)=U_1(\gamma_Nx)$ with $\gamma_N$ being given by
\begin{equation}\label{e24}
\gamma_N=\left\{
\begin{array}{lcl}
\sqrt{\frac{1-bS^2}{a}}, &\text{if}&N=4,\smallskip\\
\frac{2}{bS^{\frac{3}{2}}+\sqrt{b^2S^3+4a}}, &\text{if}&N=3.
\end{array}\right.
\end{equation}
Moreover,  the least energy $m_\lambda$ of the ground state satisfies
\begin{equation}\label{e25}
m_\infty-m_\lambda\sim \left\{
\begin{array}{lcl}
 \left(\lambda\ln\lambda\right)^{-\frac{4-q}{q-2}}, &\text{if}&N=4,\smallskip\\
\lambda^{-\frac{6-q}{2(q-4)}}, &\text{if}&N=3,
 \end{array}
 \right.
\end{equation}
where $m_\infty$ is given by
\begin{equation}\label{e26}
m_\infty=\left\{\begin{array}{lcl}
\frac{a^2S^2}{4(1-bS^2)}, &\text{if}&N=4,\smallskip\\
\frac{1}{6}a(bS^3+S^{\frac{3}{2}}\sqrt{b^2S^3+4a})&\text{if}&N=3.\\
\mbox{} +\frac{1}{48}b(bS^3+S^{\frac{3}{2}}\sqrt{b^2S^3+4a})^2,
\end{array}\right.
\end{equation}
\end{theorem}

\begin{theorem} \label{t2}
 Let $N=3$, $b>0$, $2<q\le p<2^*$ and $\{u_\lambda\}$ be a family ground states of $(P_\lambda)$. 
If $p=q$, then for any $\lambda>0$, $u_\lambda$ is the unique positive solution of $(P_\lambda)$ and satisfies
$$
u_\lambda(0)= \left(\lambda/2\right)^{\frac{1}{p-2}}W_0(0),
$$
\begin{equation}\label{e27}
\begin{array}{lcl}
\|\nabla u_\lambda\|^2_2&=&\lambda^{\frac{6-p}{2(p-2)}}\frac{3(p-2)}{p}\sqrt{\varpi_\lambda}\left(S_p/2\right)^{\frac{p}{p-2}}\\
&=&\left\{ \begin{array}{lcl}
\lambda^{\frac{6-p}{p-2}}\left(\frac{9(p-2)^2}{2p^2}(S_p/2)^{\frac{p}{p-2}}+\Theta(\lambda^{-\frac{6-p}{p-2}})\right),\  &as& \lambda\to\infty,\\
\lambda^{\frac{6-p}{2(p-2)}}\left(a^{\frac{1}{2}}(S_p/2)^{\frac{p}{p-2}}+\Theta(\lambda^{\frac{6-p}{2(p-2)}})\right), \  &as&  \lambda\to 0,\end{array}\right.
\end{array}
\end{equation}
\begin{equation}\label{e28}
\begin{array}{lcl}
\| u_\lambda\|^2_2&=&\lambda^{\frac{10-3p}{2(p-2)}}\frac{6-p}{p}(\sqrt{\varpi_\lambda})^3\left(S_p/2\right)^{\frac{p}{p-2}}\\
&=&\left\{\begin{array}{lcl}
\lambda^{\frac{14-3p}{p-2}}\left(\frac{27b^3(p-2)^3(6-p)}{8p^4}(S_p/2)^{\frac{p}{p-2}}+\Theta(\lambda^{-\frac{6-p}{p-2}}\right), \  &as&  \lambda\to\infty,\\
\lambda^{\frac{10-3p}{2(p-2)}}\left(\frac{6-p}{p}a^{\frac{3}{2}}(S_p/2)^{\frac{p}{p-2}}+\Theta(\lambda^{\frac{6-p}{2(p-2)}})\right), \  &as&  \lambda\to 0,\end{array}
\right.
\end{array}
\end{equation}
\begin{equation}\label{e29}
\begin{array}{lcl}
\|u_\lambda\|_p^p&=&\lambda^{\frac{6-p}{2(p-2)}}(\sqrt{\varpi_\lambda})^3\left(S_p/2\right)^{\frac{p}{p-2}}\\
&=&\left\{\begin{array}{lcl}
\lambda^{\frac{14-3p}{p-2}}\left(\frac{27b^3(p-2)^3}{8p^3}(S_p/2)^{\frac{p}{p-2}}+\Theta(\lambda^{-\frac{6-p}{p-2}})\right), \  &as&  \lambda\to \infty,\\
\lambda^{\frac{6-p}{2(p-2)}}\left(a^{\frac{3}{2}}(S_p/2)^{\frac{p}{p-2}}+\Theta(\lambda^{\frac{6-p}{2(p-2)}})\right),\  &as& \lambda\to 0.\end{array}
\right.
\end{array}
\end{equation}
Moreover, for any $\lambda>0$, there holds
\begin{equation}\label{e223}
u_\lambda(x)=\left(\lambda/2\right)^{\frac{1}{p-2}}W_0(\lambda^{\frac{1}{2}}\varpi_\lambda^{-\frac{1}{2}}x), 
\end{equation}
where 
\begin{equation}\label{e210}
\sqrt{\varpi_\lambda}=
\frac{3b(p-2)}{4p}\lambda^{\frac{6-p}{2(p-2)}}(S_p/2)^{\frac{p}{p-2}}+\sqrt{\frac{9b^2(p-2)^2}{16p^2}\lambda^{\frac{6-p}{p-2}}(S_p/2)^{\frac{2p}{p-2}}+a},
\end{equation}
and $S_p=\|
W_0\|_p^{p-2}, \  W_0\in H^1(\mathbb R^3)$ is the unique positive solution of the equation
$$
-\Delta W +W=W^{p-1} \quad\text{in $\R^N$}.
$$
If $q<p$ and $\lambda>0$ is sufficiently small, then $u_\lambda$ is the unique positive solution of $(P_\lambda)$ and satisfies
$$
u_\lambda(0)= \lambda^{\frac{1}{q-2}}(V_0(0)+o(1)),
$$
\begin{equation}\label{e211}
\|\nabla u_\lambda\|_2^2=\left\{\begin{array}{lcl}
\lambda^{\frac{6-q}{2(q-2)}}\left(\frac{3(q-2)}{2q}a^{\frac{1}{2}}S_q^{\frac{q}{q-2}}-\Theta(\lambda^{\frac{p-q}{q-2}})\right),   \   \ &{\rm if}&  \  q>2p-6,\\
\lambda^{\frac{6-q}{2(q-2)}}\left(\frac{3(q-2)}{2q}a^{\frac{1}{2}}S_q^{\frac{q}{q-2}}+O(\lambda^{\frac{6-q}{2(q-2)}})\right),   \   \ &{\rm if}&  \  q\le 2p-6, \end{array} \right.
\end{equation}
\begin{equation}\label{e212}
\|u_\lambda\|^2_2
=\left\{\begin{array}{lcl}
\lambda^{\frac{10-3q}{2(q-2)}}\left(\frac{6-q}{2q}a^{\frac{3}{2}}S_q^{\frac{q}{q-2}}-\Theta(\lambda^{\frac{p-q}{q-2}})\right),   \   \ &{\rm if}&  \  q>2p-6,\\
\lambda^{\frac{10-3q}{2(q-2)}}\left(\frac{6-q}{2q}a^{\frac{3}{2}}S_q^{\frac{q}{q-2}}+O(\lambda^{\frac{6-q}{2(q-2)}})\right),   \   \ &{\rm if}&  \  q\le 2p-6, \end{array} \right.
\end{equation}
\begin{equation}\label{e213}
\|u_\lambda\|_q^q=\left\{\begin{array}{lcl}
\lambda^{\frac{6-q}{2(q-2)}}\left(a^{\frac{3}{2}}S_q^{\frac{q}{q-2}}-\Theta(\lambda^{\frac{p-q}{q-2}})\right),   \   \ &{\rm if}&  \  q>2p-6,\\
\lambda^{\frac{6-q}{2(q-2)}}\left(a^{\frac{3}{2}}S_q^{\frac{q}{q-2}}+O(\lambda^{\frac{6-q}{2(q-2)}})\right),   \   \ &{\rm if}&  \  q\le 2p-6. \end{array} \right.
\end{equation}
Moreover, as $\lambda\to 0$,  the rescaled family of groundstates 
\begin{equation}\label{e224}
v_\lambda(x)=\lambda^{-\frac{1}{q-2}}u_\lambda(\lambda^{-1/2}\sqrt{\varpi_\lambda}x), \  \
\varpi_\lambda=a+b\int_{\R^N}|\nabla u_\lambda|^2,
\end{equation}
converge in $H^1(\mathbb R^3)$ to  the unique positive solution $V_0$ of the equation
$$
-\Delta V+V=V^{q-1}\quad\text{in $\R^N$}.
$$
If $q<p$ and   $\lambda>0$ is sufficiently large, then $u_\lambda$ is the unique positive solution of $(P_\lambda)$ and satisfies 
$$
u_\lambda(0)= \lambda^{\frac{1}{p-2}}(W_0(0)+o(1)),
$$
\begin{equation}\label{e214}
\|\nabla u_\lambda\|_2^2
=\left\{\begin{array}{lcl}
\lambda^{\frac{6-p}{p-2}}\left(\frac{9b(p-2)^2}{4p^2}S_p^{\frac{2p}{p-2}}+O(\lambda^{-\frac{p-q}{p-2}})\right),   \   \ &{\rm if}&  \  q>2p-6,\\
\lambda^{\frac{6-p}{p-2}}\left(\frac{9b(p-2)^2}{4p^2}S_p^{\frac{2p}{p-2}}+O(\lambda^{-\frac{6-p}{p-2}})\right),   \   \ &{\rm if}&  \  q\le 2p-6. \end{array} \right.
\end{equation}
\begin{equation}\label{e215}
\| u_\lambda\|^2_2
=\left\{\begin{array}{lcl}
\lambda^{\frac{14-3p}{p-2}}\left(\frac{27b^3(p-2)^3(6-p)}{16p^4}S_p^{\frac{4p}{p-2}}+O(\lambda^{-\frac{p-q}{p-2}})\right),   \   \ &{\rm if}&  \  q>2p-6,\\
\lambda^{\frac{14-3p}{p-2}}\left(\frac{27b^3(p-2)^3(6-p)}{16p^4}S_p^{\frac{4p}{p-2}}+O(\lambda^{-\frac{6-p}{p-2}})\right),   \   \ &{\rm if}&  \  q\le 2p-6, \end{array} \right.
\end{equation}
\begin{equation}\label{e216}
\| u_\lambda\|^p_p
=\left\{\begin{array}{lcl}
\lambda^{\frac{2(6-p)}{p-2}}\left(\frac{27b^3(p-2)^3}{8p^3}S_p^{\frac{4p}{p-2}}+O(\lambda^{-\frac{p-q}{p-2}})\right),   \   \ &{\rm if}&  \  q>2p-6,\\
\lambda^{\frac{2(6-p)}{p-2}}\left(\frac{27b^3(p-2)^3}{8p^3}S_p^{\frac{4p}{p-2}}+O(\lambda^{-\frac{6-p}{p-2}})\right),   \   \ &{\rm if}&  \  q\le 2p-6. \end{array} \right.
\end{equation}
Moreover, as $\lambda\to \infty$,  the rescaled family of groundstates 
\begin{equation}\label{e225}
w_\lambda(x)=\lambda^{-\frac{1}{p-2}}u_\lambda(\lambda^{-1/2}\sqrt{\varpi_\lambda}x), \  \
\varpi_\lambda=a+b\int_{\R^N}|\nabla u_\lambda|^2,
\end{equation}
converge  in $H^1(\mathbb R^3)$ to  the unique positive solution $W_0$  of the equation
$$
-\Delta W+W=W^{p-1}\quad\text{in $\R^N$}.
$$
\end{theorem}

Assume $M\in C^1((0,\infty), \mathbb R)$, we denote $M(0):=\lim_{\lambda\to 0}M(\lambda)$ and $M(\infty):=\lim_{\lambda\to \infty}M(\lambda)$. The following lemma is proved in \cite[Lemma 1.1]{MM-2}.

\smallskip
\noindent
{\bf Lemma 2.1.} {\it Let $\eta>0$ is a constant.  Then the following statements hold true:

\smallskip
(1)  If $M(\lambda)\sim \lambda^\eta$ as $\lambda\to 0$, then there is $\lambda_0>0$ such that  $M'(\lambda)>0$ 
for $\lambda\in (0,\lambda_0)$.

\smallskip
(2)  If $M(\lambda)\sim \lambda^{-\eta}$ as $\lambda\to 0$, then there is $\lambda_0>0$ such that  $M'(\lambda)<0$ 
for $\lambda\in (0,\lambda_0)$. 

\smallskip
(3)  If $M(\lambda)\sim \lambda^{-\eta}$ as $\lambda\to\infty$, then there is $\lambda_\infty>0$ such that  $M'(\lambda)<0$ 
for $\lambda\in (\lambda_\infty, \infty)$.

\smallskip
(4)  If $M(\lambda)\sim \lambda^{\eta}$ as $\lambda\to\infty$, then there is  $\lambda_\infty>0$ such that  $M'(\lambda)>0$ 
for $\lambda\in (\lambda_\infty, \infty)$.}

\smallskip
\noindent
The following corollary is a direct consequence of Theorems 2.1--2.2 and Lemma 2.1.

\noindent
{\bf Corollary 2.1.} {\em Let  $2<q\le p\le 6$, then
\begin{equation}\label{e217}
M(0)= \left\{\begin{array}{rcl} 
0,  \quad \qquad \quad  &{\rm if}& \   \ q<10/3  \    \  {\rm and}  \   \  q\le p< 6,\\
\frac{6-q}{2q}a^{\frac{3}{2}}S_q^{\frac{q}{q-2}},  \  &{\rm if}& \   \ q=10/3  \  \  {\rm and} \   \ 10/3< p<6,\\
\infty,  \qquad  \  \quad \  &{\rm if}& \   \ q>10/3 \    \  {\rm and}  \   \  q\le p<6,\\
 \end{array}
\right.
\end{equation}
and 
\begin{equation}\label{e218}
M(\infty)= \left\{\begin{array}{rcl} 
\infty, \quad \quad \qquad  \qquad  &{\rm if}& \   \ p<14/3 \    \  {\rm and}  \   \  q\le p,\\
\frac{27b^3(p-2)^3(6-p)}{16p^4}S_p^{\frac{4p}{p-2}},\  &{\rm if}& \   \ p=14/3 \  \  {\rm and} \   q<p,\\
0, \quad  \quad  \qquad  \qquad   \   &{\rm if}& \   \ p>14/3 \    \  {\rm and}  \   \  q\le p.
 \end{array}
\right.
\end{equation}
 Moreover,  there exists a small $\lambda_0>0$  such that
for any $\lambda\in (0,\lambda_0)$,
\begin{equation}\label{e219}
 \left\{\begin{array}{rcl} 
M'(\lambda)>0, \quad   &{\rm if}& \   \ q<10/3 \    \  {\rm and}  \   \  q\le p<6,\\
M'(\lambda)<0,  \quad   &{\rm if}& \  \ q=10/3  \    \  {\rm and} \   \  10/3<p<14/3,\\
M'(\lambda)<0,  \quad    &{\rm if}& \  \ q>10/3  \    \  {\rm and} \   \  q\le p,
 \end{array}\right.
 \end{equation}
and there exists a large $\lambda_\infty>0$ such that for any $\lambda\in (\lambda_\infty, +\infty)$, 
\begin{equation}\label{e220}
 \left\{\begin{array}{rcl} 
M'(\lambda)<0, \quad   &{\rm if}& \   \  p<14/3 \   \   {\rm and } \    \  q\le p,\\
M'(\lambda)>0, \quad   &{\rm if}& \   \ p>14/3 \    \  {\rm and}  \   \   q\le p.
 \end{array}
\right.
\end{equation} }

\smallskip
\noindent
{\bf Remark 2.1.}  In the Sobolev critical case $p=2^*$, a similar result as above also holds and in particular, we have
$M(0)=\infty$ and $M(\infty)=0$ if  $N=4$ and $q\in (3,4)$, or $N=3$ and  $q\in (4,6)$; 
$M(0)=M(\infty)=0$ if $N=4$ and $q\in (2,3)$; $M(0)=\infty$ if  $N=3$ and $q\in (10/3,4]$, and $M(0)=0$ if $N=3$ and $q\in (2,10/3)$.
 The sign of $M'(\lambda)$   can also be determined for small $\lambda>0$ or for large $\lambda>0$, and we omitted the details. 
 We mention that in the subcritical case, the sign of $M'(\lambda)$
 is still open in the following cases:
 
(1)   $q=\frac{10}{3}$, $p\in (\frac{14}{3}, 6)$ and $\lambda>0$ small,  \quad   (2)  $p=\frac{14}{3}$, $q\in (2, \frac{14}{3})$ and $\lambda>0$ large.

According to Corollary 2.1 and results concerning the case $b=0$ given in Section 5,  we draw the following figures which  reveal the variations of $M(\lambda)$ for small $\lambda>0$ and large $\lambda>0$ when $(p,q)$ belongs to different regions in the $(p,q)$ plane. Plainly,  there exists a  dramatic change between $b=0$ and $b\not=0$.  This observation results in a striking different  feature in the existence, non-existence and exact number of normalized solutions of \eqref{e15}.

\begin{figure}[h]
	\centering
	\includegraphics[width=1.0\linewidth]{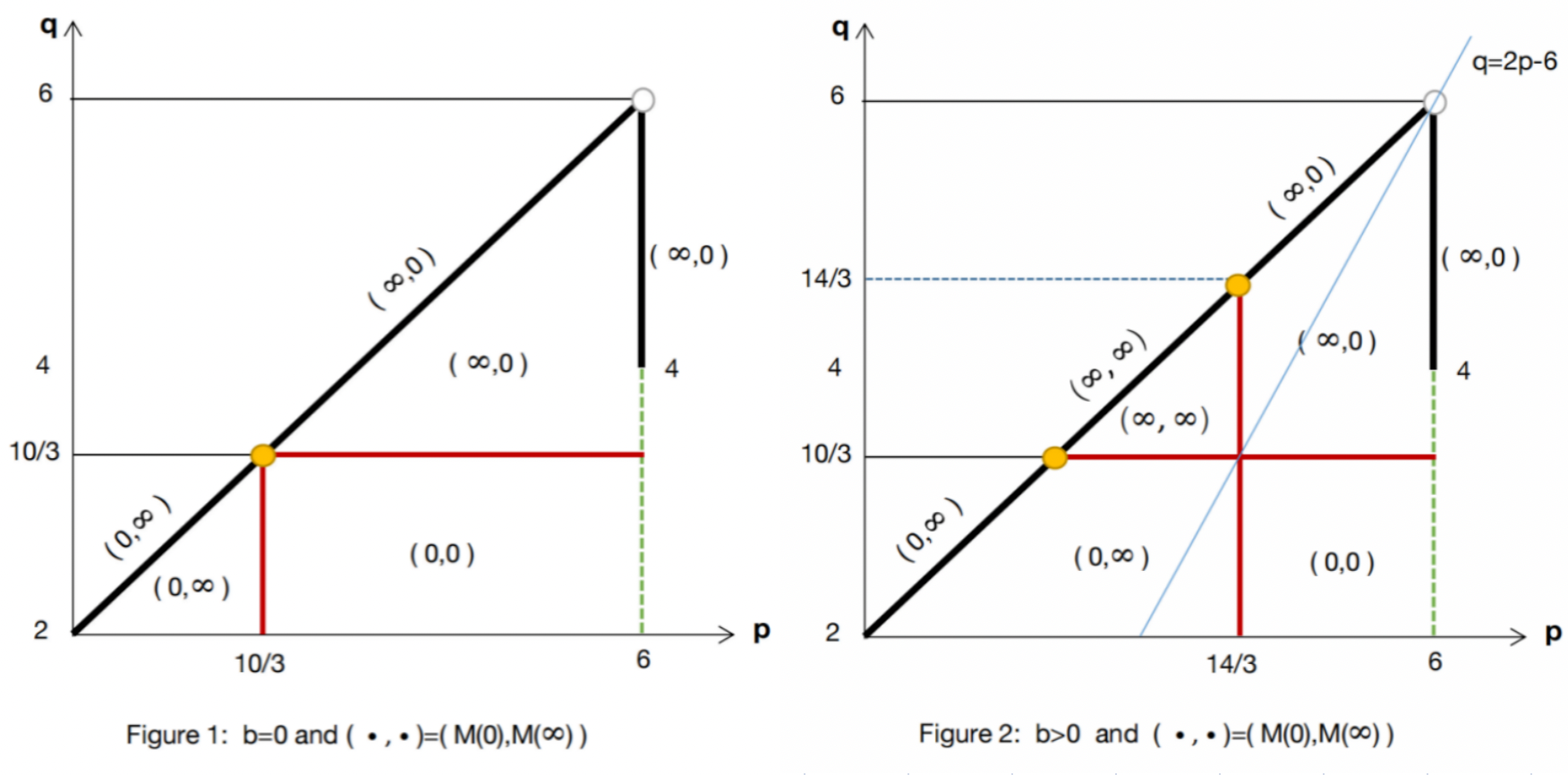}
	\label{fig:figure1-2}
\end{figure}

\section{ Proof of Theorem 2.1 as $\lambda\to 0$}\label{s3}

Let 
\begin{equation}\label{e31}
v(x)=\lambda^{-\frac{1}{q-2}}u(\lambda^{-\frac{1}{2}}x).
\end{equation}
Then the equation $(P_\lambda)$ reduces to 
$$
-a\Delta v+v-b\lambda^{\frac{2N-q(N-2)}{2(q-2)}}\int_{\mathbb R^N}|\nabla v|^2dx \Delta v=v^{q-1}+\lambda^{\frac{2^*-q}{q-2}}v^{2^*-1}.
\eqno{(Q_\lambda)}
$$
The energy functional for $(Q_\lambda)$ is defined by 
\begin{equation}\label{e32}
\begin{array}{lcl}
J_\lambda (v)&=&\frac{a}{2}\int_{\mathbb R^N}|\nabla v|^2+\frac{1}{2}\int_{\mathbb R^N}|v|^2+\frac{b}{4}\lambda^{\frac{2N-q(N-2)}{2(q-2)}}\Big(\int_{\mathbb R^N}|\nabla v|^2\Big)^2\\
&\mbox{}&\qquad 
-\frac{1}{q}\int_{\mathbb R^N}|v|^q-\frac{1}{2^*}\lambda^{\frac{2^*-q}{q-2}}\int_{\mathbb R^N}|v|^{2^*}.
\end{array}
\end{equation}
The formal limit equation for $(Q_\lambda)$ as $\lambda\to 0$ is given by
$$
-a\Delta v+v=v^{q-1} \quad \text{in $\R^N$.}
\eqno(Q_0)
$$
The energy functional for $(Q_0)$ is given by 
$$
J_0(v)=\frac{a}{2}\int_{\mathbb R^N}|\nabla v|^2+\frac{1}{2}\int_{\mathbb R^N}|v|^2
-\frac{1}{q}\int_{\mathbb R^N}|v|^q.
$$

\begin{lemma}\label{l22}
Let $\lambda>0$, $u\in H^1(\mathbb R^N)$ and $v$ is the rescaling \eqref{e31} of $u$.  Then: 

\begin{enumerate}
\item[$(a)$] $ \|\nabla u\|_2^2=\lambda^{\frac{2N-q(N-2)}{2(q-2)}}\|\nabla v\|_2^2,$ \quad $\|u\|_q^q=\lambda^{\frac{2N-q(N-2)}{2(q-2)}}\|v\|_q^q$,\medskip

\item[$(b)$] $\|u\|_2^2=\lambda^{\frac{4-N(q-2)}{2(q-2)}}\|v\|_2^2$,  \quad $\|u\|_{2^*}^{2^*}=\lambda^{\frac{N[2N-q(N-2)]}{2(N-2)(q-2)}}\|v\|_{2^*}^{2^*}$,
\medskip

\item[$(c)$] $I_\lambda(u)=\lambda^{\frac{2N-q(N-2)}{2(q-2)}}J_\lambda(v)$.
\end{enumerate}
\end{lemma}

The above lemma is easily proved and the details will be omitted. In particular, it follows from Lemma 3.1 (c) that  the rescaling $v$ of the ground state $u$ of $(P_\lambda)$ corresponds to a ground state of $(Q_\lambda)$.

\begin{lemma}\label{l24-1}
The rescaled family of ground-sates $\{v_\lambda\}$ is bounded in $H^1(\mathbb R^N)$ for small $\lambda>0$.
\end{lemma}

\begin{proof} It is standard to see that ground-sates of $(Q_\lambda)$ satisfy the Nehari identity 
$$
a\int_{\mathbb R^N}|\nabla v_\lambda|^2+\int_{\mathbb R^N}|v_\lambda|^2+b\lambda^{\frac{2N-q(N-2)}{2(q-2)}}(\int_{\mathbb R^N}|\nabla v_\lambda|^2)^2=\int_{\mathbb R^N}|v_\lambda|^q+\lambda^{\frac{2^*-q}{q-2}}\int_{\mathbb R^N}|v_\lambda|^{2^*}
$$
and the Poho\v zaev identity
$$
\frac{a}{2^*}\int_{\mathbb R^N}|\nabla v_\lambda|^2+\frac{1}{2}\int_{\mathbb R^N}|v_\lambda|^2+\frac{b}{2^*}\lambda^{\frac{2N-q(N-2)}{2(q-2)}}(\int_{\mathbb R^N}|\nabla v_\lambda|^2)^2=\frac{1}{q}\int_{\mathbb R^N}|v_\lambda|^q+\frac{1}{2^*}\lambda^{\frac{2^*-q}{q-2}}\int_{\mathbb R^N}|v_\lambda|^{2^*}.
$$
Therefore, it follows that
\begin{equation}\label{e33}
\left(\frac{1}{2}-\frac{1}{2^*}\right)\int_{\mathbb R^N}|v_\lambda|^2=\left(\frac{1}{q}-\frac
{1}{2^*}\right)\int_{\mathbb R^N}|v_\lambda|^q,
\end{equation}
and hence
$$
\frac{1}{N}\int_{\mathbb R^N}|v_\lambda|^2=\frac{2^*-q}{2^*q}\int_{\mathbb R^N}|v_\lambda|^q\le \frac{2^*-q}{2^*q}(\int_{\mathbb R^N}|v_\lambda|^2)^{\frac{2^*-q}{2^*-2}}
(\frac{1}{S}\int_{\mathbb R^N}|\nabla v_\lambda|^2)^{\frac{2^*(q-2)}{2(2^*-2)}},
$$
which implies that
\begin{equation}\label{e34}
\frac{1}{N}(\int_{\mathbb R^N}|v_\lambda|^2)^{\frac{q-2}{2^*-2}}\le \frac{2^*-q}{2^*q}(\frac{1}{S}\int_{\mathbb R^N}|\nabla v_\lambda|^2)^{\frac{2^*(q-2)}{2(2^*-2)}}.
\end{equation}
 
 To prove the boundedness of $\{v_\lambda\}$ in $H^1(\mathbb R^N)$, it suffices to show that $\{v_\lambda\}$ is bounded in $D^{1,2}(\mathbb R^N)$.

Let $v_0$ be the unique positive solution of the equation $(Q_0)$, then by the  Poho\v zaev's identity, we have 
\begin{equation}\label{e35}
\frac{1}{2}\int_{\mathbb R^N}|v_0|^2<\frac{a}{2^*}\int_{\mathbb R^N}|\nabla v_0|^2+\frac{1}{2}\int_{\mathbb R^N}|v_0|^2=\frac{1}{q}\int_{\mathbb R^N}|v_0|^q,
\end{equation}
which implies that 
for small $\lambda>0$, there is a unique $t_\lambda>0$ such that
\begin{equation}\label{e36}
\begin{array}{cl}
&\frac{a}{2^*t^2}\int_{\R^N}|\nabla v_0|^2+\frac{1}{2}\int_{\R^N}|v_0|^2+\frac{b}{2^*t^{4-N}}\lambda^{\frac{2N-q(N-2)}{2(q-2)}}(\int_{\R^N}|\nabla v_0|^2)^2\\
&\quad =\frac{1}{q}\int_{\R^N}|v_0|^q+\frac{1}{2^*}\lambda^{\frac{2^*-q}{q-2}}\int_{\R^N}|v_0|^{2^*}.
\end{array}
\end{equation}
If $N=3$ and $t_\lambda>1$, then 
$$
\begin{array}{lcl}
t_\lambda&\le&\frac{ \frac{1}{2^*}a\int_{\R^N}|\nabla v_0|^2+\frac{1}{2^*}b\lambda^{\frac{2N-q(N-2)}{2(q-2)}}(\int_{\R^N}|\nabla v_0|^2)^2}
{\frac{1}{q}\int_{\R^N}|v_0|^q-\frac{1}{2}\int_{\R^N}|v_0|^2+\frac{1}{2^*}\lambda^{\frac{2^*-q}{q-2}}\int_{\R^N}|v_0|^{2^*}}\\
&=&\frac{ a\int_{\R^N}|\nabla v_0|^2+b\lambda^{\frac{2N-q(N-2)}{2(q-2)}}(\int_{\R^N}|\nabla v_0|^2)^2}
{a\int_{\R^N}|\nabla v_0|^2+\lambda^{\frac{2^*-q}{q-2}}\int_{\R^N}|v_0|^{2^*}}\\
&<&C<\infty.
\end{array}
$$
If $N=4$ and $t_\lambda>1$, then
$$
\begin{array}{lcl}
t_\lambda^2&=&\frac{\frac{1}{2^*}a\int_{\R^N}|\nabla v_0|^2}{\frac{1}{q}\int_{\R^N}|v_0|^q-\frac{1}{2}\int_{\R^N}|v_0|^2+\frac{1}{2^*}\lambda^{\frac{2^*-q}{q-2}}\int_{\R^N}|v_0|^{2^*}
-\frac{1}{2^*}b\lambda^{\frac{2N-q(N-2)}{2(q-2)}}(\int_{\R^N}|\nabla v_0|^2)^2}\\
&=&\frac{a\int_{\R^N}|\nabla v_0|^2}{a\int_{\R^N}|\nabla v_0|^2+\lambda^{\frac{2^*-q}{q-2}}\int_{\R^N}|v_0|^{2^*}
-b\lambda^{\frac{2N-q(N-2)}{2(q-2)}}(\int_{\R^N}|\nabla v_0|^2)^2}\\
&<&C<\infty.
\end{array}
$$
Therefore, we obtain 
\begin{equation}\label{e37}
\begin{array}{lcl}
m_\lambda&\le &\sup_{t>0}J_\lambda((v_0)_t)\\
&=&\sup_{t>0}\frac{a}{2}t^{N-2}\int_{\R^N}|\nabla v_0|^2+\frac{1}{2}t^N\int_{\R^N}|v_0|^2+\frac{b}{4}t^{2(N-2)}\lambda^{\frac{2N-q(N-2)}{2(q-2)}}(\int_{\R^N}|\nabla v_0|^2)^2\\
&\mbox{}& \  -\frac{1}{q}t^N\int_{\R^N}|v_0|^q-\frac{1}{2^*}t^N\lambda^{\frac{2^*-q}{q-2}}\int_{\R^N}|v_0|^{2^*}\\
 &\le & \sup_{t>0}J_0((v_0)_t)++\frac{b}{4}t_\lambda^{2(N-2)}\lambda^{\frac{2N-q(N-2)}{2(q-2)}}(\int_{\R^N}|\nabla v_0|^2)^2
-\frac{1}{2^*}t_\lambda^N\lambda^{\frac{2^*-q}{q-2}}\int_{\R^N}|v_0|^{2^*}\\ 
&\le &m_0+C\lambda^{\frac{2N-q(N-2)}{2(q-2)}}.
\end{array}
\end{equation}
This prove that $m_\lambda\le C<+\infty$ for all small $\lambda>0$.

On the other hand, we have 
 $$
\begin{array}{rcl}
m_\lambda:=J_\lambda (v_\lambda)&=&\frac{a}{2}\int_{\mathbb R^N}|\nabla v_\lambda|^2+\frac{1}{2}\int_{\mathbb R^N}|v_\lambda|^2+\frac{b}{4}\lambda^{\frac{2N-q(N-2)}{2(q-2)}}(\int_{\mathbb R^N}|\nabla v_\lambda|^2)^2\\
&\mbox{}& \quad 
-\frac{1}{q}\int_{\mathbb R^N}|v_\lambda|^q-\frac{1}{2^*}\lambda^{\frac{2^*-q}{q-2}}\int_{\mathbb R^N}|v_\lambda|^{2^*}\\
&=&(\frac{1}{2}-\frac{1}{2^*})a\int_{\mathbb R^N}|\nabla v_\lambda|^2+(\frac{1}{4}-\frac{1}{2^*})b\lambda^{\frac{2N-q(N-2)}{2(q-2)}}(\int_{\mathbb R^N}|\nabla v_\lambda|^2)^2\\
&\ge & \frac{1}{N}a\int_{\mathbb R^N}|\nabla v_\lambda|^2.
\end{array}
$$
This yields the boundedness of $\|\nabla v_\lambda\|_2$  for small $\lambda>0$  and completes the proof.
\end{proof}

\begin{lemma}\label{l23}
Set
$$
v_t(x)=\left\{\begin{array}{ccl} v(\frac{x}{t}) &\text{if}&  t>0,\smallskip\\
0 &\text{if}& t=0.
\end{array}\right.
$$
Then for small $\lambda>0$,  there holds
\begin{equation}\label{e38}
m_\lambda=\inf_{v\in H^1(\mathbb R^N)\setminus\{0\}}\sup_{t\ge 0}J_\lambda(tv)=\inf_{v\in H^1(\mathbb R^N)\setminus\{0\}}\sup_{t\ge 0}J_\lambda(v_t).
\end{equation}
In particular, we have $m_\lambda=J_\lambda(v_\lambda)=\sup_{t>0}J_\lambda(tv_\lambda)=\sup_{t>0}J_\lambda((v_\lambda)_t)$. 
\end{lemma}

The proof of Lemma 3.3  is similar to that of  \cite[Lemma 3.2]{MM-2} and is omitted. Next we obtain an estimation of the least energy.

\begin{lemma}\label{l24-2}
Let $N=3$ or $4$, and $u_\lambda$ is a ground state solution of $(P_\lambda)$, then 
\begin{equation}\label{e39}
m_\lambda-m_0=O(\lambda^{\frac{2N-q(N-2)}{2(q-2)}}),
\end{equation}
as $\lambda\to 0$, where $m_0:=\inf_{v\in H^1(\mathbb R^N)\setminus\{0\}}\sup_{t\ge 0}J_0(tv)$ is the ground state energy for $(Q_0)$.
\end{lemma}

\begin{proof}  
By \eqref{e37}, we have 
$$
m_\lambda\le m_0+C\lambda^{\frac{2N-q(N-2)}{2(q-2)}}.
$$
On the other hand, by Lemma 3.3, we have 
$$
\begin{array}{lcl}
m_0&\le& \sup_{t\ge 0}J_0(tv_\lambda)=J_0(t_\lambda v_\lambda)\\
&\le& \sup_{t\ge 0}J_\lambda(tv_\lambda)-\frac{b}{4}t_\lambda^4\lambda^{\frac{2N-q(N-2)}{2(q-2)}}(\int_{\mathbb R^N}|\nabla v_\lambda|^2)^2+\frac{1}{2^*}t_\lambda^{2^*}\lambda^{\frac{2^*-q}{q-2}}\int_{\mathbb R^N}|v_\lambda|^{2^*}
\end{array}
$$
where 
$$
\begin{array}{lcl}
t_\lambda&=&\frac{a\int_{\mathbb R^N}|\nabla v_\lambda|^2+\int_{\mathbb R^N}|v_\lambda|^2}{\int_{\mathbb R^N}|v_\lambda|^q}\\
&=&\frac{\|v_\lambda\|^2}{\|v_\lambda\|^2+b\lambda^{\frac{2N-q(N-2)}{2(q-2)}}(\int_{\mathbb R^N}|\nabla v_\lambda|^2)^2-\lambda^{\frac{2^*-q}{q-2}}\int_{\mathbb R^N}|v_\lambda|^{2^*}}\\
&\le &\frac{1}{1-C\lambda^{\frac{2^*-q}{q-2}}\|v_\lambda\|^{2^*-2}}\\
&\le &1+C\lambda^{\frac{2^*-q}{q-2}}.
\end{array}
$$
Therefore, we get
$$
m_0\le m_\lambda+C\lambda^{\frac{2N-q(N-2)}{2(q-2)}}.
$$
The proof is complete.
\end{proof}

\begin{lemma}\label{l25}
Let $N=3$ or $4$, and $u_\lambda$ is a ground state solution of $(P_\lambda)$, then 
\begin{equation}\label{e29-2}
\int_{\mathbb R^N}|v_\lambda|^q=\int_{\mathbb R^N}|v_0|^q+O(\lambda^{\frac{2N-q(N-2)}{2(q-2)}}),
\end{equation}
\begin{equation}\label{e310}
\int_{\mathbb R^N}|v_\lambda|^2=\int_{\mathbb R^N}|v_0|^2+O(\lambda^{\frac{2N-q(N-2)}{2(q-2)}}),
\end{equation}
\begin{equation}\label{e311}
\int_{\mathbb R^N}|\nabla v_\lambda|^2=\int_{\mathbb R^N}|\nabla v_0|^2+O(\lambda^{\frac{2N-q(N-2)}{2(q-2)}}),
\end{equation}
as $\lambda\to 0$.
\end{lemma}

\begin{proof}  Let $u_\lambda$ be  a ground state solution of $(P_\lambda)$ and 
\begin{equation}\label{e312}
A_\lambda=\int_{\mathbb R^N}|\nabla v_\lambda|^2, \quad B_\lambda=\int_{\mathbb R^N}|v_\lambda|^2, \quad C_\lambda=\int_{\mathbb R^N}|v_\lambda|^q, \quad D_\lambda=\int_{\mathbb R^N}|v_\lambda|^{2^*}.
\end{equation}
Then the Nehari and Poho\v zaev identities imply that
$$
\left\{ \begin{array}{rcl}
aA_\lambda +B_\lambda+b\lambda^{\frac{2N-q(N-2)}{2(q-2)}}A_\lambda^2 &=& C_\lambda+\lambda^{\frac{2^*-q}{q-2}}D_\lambda,\\
\frac{a}{2^*}A_\lambda +\frac{1}{2}B_\lambda+\frac{b}{2^*}\lambda^{\frac{2N-q(N-2)}{2(q-2)}}A_\lambda^2 &=& \frac{1}{q}C_\lambda+\frac{1}{2^*}\lambda^{\frac{2^*-q}{q-2}}D_\lambda.
\end{array}
\right.
$$
From which, we conclude that
\begin{equation}\label{e313}
B_\lambda=\frac{N(2^*-q)}{2^*q}C_\lambda, \quad aA_\lambda=\frac{N(q-2)}{2q}C_\lambda +O(\lambda^{\frac{2N-q(N-2)}{2(q-2)}}).
\end{equation}
Therefore, we obtain 
$$
\begin{array}{lcl}
m_\lambda&=&\frac{1}{2}aA_\lambda+\frac{1}{2}B_\lambda-\frac{1}{q}C_\lambda+O(\lambda^{\frac{2N-q(N-2)}{2(q-2)}})\\
&=&\left(\frac{N(q-2)}{4q}+\frac{N(2^*-q)}{22^*q}-\frac{1}{q}\right)C_\lambda+O(\lambda^{\frac{2N-q(N-2)}{2(q-2)}})\\
&=&\frac{q-2}{2q}C_\lambda+O(\lambda^{\frac{2N-q(N-2)}{2(q-2)}}).
\end{array}
$$
In a similar way, we can show that 
$$
m_0=\frac{q-2}{2q}C_0.
$$
Thus, we obtain
$$
\frac{q-2}{2q}(C_\lambda-C_0)=m_\lambda-m_0+O(\lambda^{\frac{2N-q(N-2)}{2(q-2)}}),
$$
which together with Lemma 3.4 implies that
$$
C_\lambda-C_0=\int_{\mathbb R^N}|v_\lambda|^q-\int_{\mathbb R^N}|v_0|^q=O(\lambda^{\frac{2N-q(N-2)}{2(q-2)}}).
$$
Since 
\begin{equation}\label{e314}
B_0=\frac{N(2^*-q)}{2^*q}C_0, \quad aA_0=\frac{N(q-2)}{2q}C_0.
\end{equation}
it follows from \eqref{e313} and \eqref{e314}
that
$$
A_\lambda-A_0=\frac{N(q-2)}{2aq}(C_\lambda-C_0)+O(\lambda^{\frac{2N-q(N-2)}{2(q-2)}})    =O(\lambda^{\frac{2N-q(N-2)}{2(q-2)}}),
$$
$$
B_\lambda-B_0=\frac{N(2^*-q)}{2^*q}(C_\lambda-C_0)=O(\lambda^{\frac{2N-q(N-2)}{2(q-2)}}),
$$
from which \eqref{e310} and \eqref{e311} follows. The  proof is complete.
\end{proof}

\begin{proof}[Proof of Theorem 2.1 for small $\lambda$]   Observe that $v_\lambda\to v_0$ in $H^1(\mathbb R^N)$ with $v_0$ being the unique ground state solution of $(Q_0)$.  For small $\lambda>0$, Theorem 2.1 follows from Lemmas 3.1-3.5 and the details will be omitted.
\end{proof}

\section{Proof of Theorem 2.1 as $\lambda\to\infty$}\label{s4}

\subsection{Rescalings} 

For $\lambda>0$, define the rescaling
\begin{equation}\label{e41}
v(x)=\lambda^{-\frac{1}{q-2}} u\big(\lambda^{-\frac{2^*-2}{2(q-2)}}x\big). 
\end{equation}
Rescaling \eqref{e41} transforms $(P_\lambda)$ into the equivalent equaition 
$$
-\big(a+b\int_{\mathbb R^N}|\nabla v|^2\big)\Delta v+\lambda^{-\frac{2^*-q}{q-2}} v=\lambda^{-\frac{2^*-q}{q-2}}v^{q-1} +v^{2^*-1}
\quad\text{in} \  \ \R^N.
 \eqno(R_\lambda)
$$
The corresponding energy functional is given by 
\begin{equation}\label{e42}
J_\lambda(v)=\frac{a}{2}\int_{\mathbb R^N}|\nabla v|^2+\frac{\lambda^{-\sigma}}{2}\int_{\R^N}|v|^2+\frac{b}{4}\left(\int_{\mathbb R^N}|\nabla v|^2\right)^2-\frac{\lambda^{-\sigma}}{q}\int_{\mathbb R^N}|v|^q-\frac{1}{2^*}\int_{\mathbb R^N}|v|^{2^*},
\end{equation}
here and in what follows, we set
$$
\sigma:=\frac{2^*-q}{q-2}.
$$
The formal limit equation for $(R_\lambda)$ as $\lambda\to \infty$ is given by the  equation
$$
-\big(a+b\int_{\mathbb R^N}|\nabla v|^2\big)\Delta v=v^{2^*-1} \quad \text{in $\R^N$.}
\eqno(R_\infty)
$$
The corresponding functional is given by
\begin{equation}\label{e43}
J_\infty(v)=\frac{a}{2}\int_{\mathbb R^N}|\nabla v|^2+\frac{b}{4}\left(\int_{\mathbb R^N}|\nabla v|^2\right)^2-\frac{1}{2^*}\int_{\mathbb R^N}|v|^{2^*}.
\end{equation}
We denote their corresponding Nehari manifolds as follows:
$$
\mathcal N_\lambda:=\left\{ v\in H^1(\mathbb R^N)\setminus\{0\}  \ \left | \ \int_{\mathbb R^N}a|\nabla v|^2+\lambda^{-\sigma} |v|^2+b(\int_{\mathbb R^N}|\nabla v|^2)^2=\int_{\mathbb R^N}|v|^{2^*}+\lambda^{-\sigma} |v|^q \right. \right\}.
$$
$$
\mathcal{N}_\infty:=
\left\{v\in D^{1,2}(\mathbb R^N)\setminus\{0\} \ \left | \ \int_{\mathbb R^N}a|\nabla v|^2+b(\int_{\mathbb R^N}|\nabla v|^2)^2=\int_{\mathbb R^N}|v|^{2^*}\  \right. \right\}. 
$$
Then
$$
m_\lambda:=\inf_{v\in \mathcal {N}_\lambda}J_\lambda(v), \qquad  m_\infty:=\inf_{v\in \mathcal {N}_\infty}J_\infty(v) 
$$
are well-defined and positive.

Let $\varpi:=\varpi(v)=a+b\int_{\mathbb R^N}|\nabla v|^2$, then  $\varpi(v)$ is invariant respect to the rescaling
\begin{equation}\label{e44}
\tilde v(x)=\eta^{\frac{N-2}{2}}v(\eta x).
\end{equation}
That is, $\varpi(\tilde v)=\varpi(v)$.

Set
\begin{equation}\label{e45}
w(x)=v(\sqrt \varpi x),
\end{equation}
then the equation $(R_\infty)$ reduces to
\begin{equation}\label{e46}
-\Delta w=w^{2^*-1}.
\end{equation}
Moreover, $\varpi$ satisfies the equation
$$
\varpi=a+\varpi^{\frac{N-2}{2}}b\int_{\mathbb R^N}|\nabla w|^2.
$$
If $N=3$, then
$$
\sqrt \varpi=\frac{b\int_{\mathbb R^N}|\nabla w|^2+\sqrt{b^2(\int_{\mathbb R^N}|\nabla w|^2)^2+4a}}{2}=\frac{bS^{\frac{3}{2}}+\sqrt{b^2S^3+4a}}{2}.
$$
If $N=4$ and $bS^2<1$, then 
$$
\varpi=\frac{a}{1-b\int_{\mathbb R^N}|\nabla w|^2}=\frac{a}{1-bS^2}.
$$

It is easy to see  that  $m_\infty$ is attained on $\mathcal N_\infty$ by $v_1(x):=W_1(\frac{x}{\sqrt\varpi})$ and the family of its rescalings 
\begin{equation}\label{e47}
v_\rho(x):=\rho^{-\frac{N-2}{2}}v_1(x/\rho),  \quad \rho>0,
\end{equation}
where $W_1$ is the the Talenti function
$$
W_1(x):=[N(N-2)]^{\frac{N-2}{4}}\left(\frac{1}{1+|x|^2}\right)^{\frac{N-2}{2}}.
$$
Furthermore, a direct computation shows that
\begin{equation}\label{e48}
\frac{\int_{\mathbb R^N}|\nabla v_\rho|^2}{\left(\int_{\mathbb R^N}|v_\rho|^{2^*}|\right)^{\frac{2}{2^*}}}=\frac{\int_{\mathbb R^N}|\nabla W_1|^2}{\left(\int_{\mathbb R^N}|W_1|^{2^*}|\right)^{\frac{2}{2^*}}}=S.
\end{equation}

\begin{lemma}\label{l310}
Let $\lambda>0$, $u\in H^1(\mathbb R^N)$ and $v$ is the rescaling \eqref{e41} of $u$.  Then: 

\begin{enumerate}
\item[$(a)$] $ \|\nabla u\|_2^2=\|\nabla v\|_2^2,$    \   $\|u\|_{2^*}^{2^*}=\|v\|_{2^*}^{2^*}$,
\medskip

\item[$(b)$] $\lambda^{1+\sigma}\|u\|_2^2=\|v\|_2^2$,  \    $\lambda^\sigma\|u\|_q^q=\|v\|_q^q$, \medskip

\item[$(c)$] $I_\lambda(u)=J_\lambda(v)$.
\end{enumerate}
\end{lemma}

In particular, if $v_\lambda$ is the rescaling \eqref{e41} of the ground state $u_\lambda$, then 
$$
J_\lambda(v_\lambda)=I_\lambda(u_\lambda)
$$ 
and hence  $v_\lambda$ is the ground state of $(R_\lambda)$. 
Moreover, $v_\lambda$ satisfies the Poho\v zaev's identity \cite{Berestycki-1}:
\begin{equation}\label{e49}
\frac{a}{2^*}\int_{\mathbb R^N}|\nabla v_\lambda|^2+\frac{\lambda^{-\sigma}}{2}\int_{\mathbb R^N}|v_\lambda|^2+\frac{b}{2^*}\left(\int_{\mathbb R^N}|\nabla v_\lambda|^2\right)^2=\frac{\lambda^{-\sigma}}{q}\int_{\mathbb R^N}|v_\lambda|^q+\frac{1}{2^*}\int_{\mathbb R^N}|v_\lambda|^{2^*}.
\end{equation}
Define the Poho\v zaev manifold 
$$
\mathcal P_\lambda:=\{v\in H^1(\mathbb R^N)\setminus\{0\} \   | \  P_\lambda(v)=0 \},
$$
where 
\begin{equation}\label{e410}
P_\lambda(v):=\frac{a}{2^*}\int_{\R^N}|\nabla v|^2+\frac{\lambda^{-\sigma}}{2}\int_{\mathbb R^N}|v|^2+\frac{b}{2^*}\left(\int_{\R^N}|\nabla v|^2\right)^2-\frac{\lambda^{-\sigma}}{q}\int_{\mathbb R^N}|v|^q-\frac{1}{2^*}\int_{\mathbb R^N}|v|^{2^*}.
\end{equation}
Clearly, $v_\lambda\in \mathcal P_\lambda$. Moreover,  we have the following minimax characterizations for the least energy level $m_\lambda$:
\begin{equation}\label{e411}
m_\lambda=\inf_{v\in H^1(\mathbb R^N)\setminus\{0\}}\sup_{t\ge 0}J_\lambda(tv)=\inf_{v\in H^1(\mathbb R^N)\setminus\{0\}}\sup_{t\ge 0}J_\lambda(v_t).
\end{equation}
In particular, we have $m_\lambda=J_\lambda(v_\lambda)=\sup_{t>0}J_\lambda(tv_\lambda)=\sup_{t>0}J_\lambda((v_\lambda)_t)$. 

\begin{lemma}\label{l24-3}
The rescaled family of ground-sates $\{v_\lambda\}$ is bounded in $H^1(\mathbb R^N)$ for large $\lambda>0$.
\end{lemma}

\begin{proof} The Nehari identity 
$$
a\int_{\mathbb R^N}|\nabla v_\lambda|^2+\lambda^{-\sigma}\int_{\mathbb R^N}|v_\lambda|^2+b(\int_{\mathbb R^N}|\nabla v_\lambda|^2)^2=\lambda^{-\sigma}\int_{\mathbb R^N}|v_\lambda|^q+\int_{\mathbb R^N}|v_\lambda|^{2^*}
$$
and the Poho\v zaev identity
$$
\frac{a}{2^*}\int_{\mathbb R^N}|\nabla v_\lambda|^2+\frac{1}{2}\lambda^{-\sigma}\int_{\mathbb R^N}|v_\lambda|^2+\frac{b}{2^*}(\int_{\mathbb R^N}|\nabla v_\lambda|^2)^2=
\frac{1}{q}\lambda^{-\sigma}\int_{\mathbb R^N}|v_\lambda|^q+\frac{1}{2^*}\int_{\mathbb R^N}|v_\lambda|^{2^*}.
$$
imply that
\begin{equation}\label{e412}
\left(\frac{1}{2}-\frac{1}{2^*}\right)\int_{\mathbb R^N}|v_\lambda|^2=\left(\frac{1}{q}-\frac
{1}{2^*}\right)\int_{\mathbb R^N}|v_\lambda|^q.
\end{equation}
Hence
$$
\frac{1}{N}\int_{\mathbb R^N}|v_\lambda|^2=\frac{2^*-q}{2^*q}\int_{\mathbb R^N}|v_\lambda|^q\le \frac{2^*-q}{2^*q}(\int_{\mathbb R^N}|v_\lambda|^2)^{\frac{2^*-q}{2^*-2}}
(\frac{1}{S}\int_{\mathbb R^N}|\nabla v_\lambda|^2)^{\frac{2^*(q-2)}{2(2^*-2)}},
$$
which implies that
$$
\frac{1}{N}(\int_{\mathbb R^N}|v_\lambda|^2)^{\frac{q-2}{2^*-2}}\le \frac{2^*-q}{2^*q}(\frac{1}{S}\int_{\mathbb R^N}|\nabla v_\lambda|^2)^{\frac{2^*(q-2)}{2(2^*-2)}}.
$$
 So,  to prove the boundedness of $\{v_\lambda\}$ in $H^1(\mathbb R^N)$, it suffices to show that $\{v_\lambda\}$ is bounded in $D^{1,2}(\mathbb R^N)$. It is easy to see that
$$
\begin{array}{rcl}
m_\lambda:=J_\lambda (v_\lambda)&=&\frac{a}{2}\int_{\mathbb R^N}|\nabla v_\lambda|^2+\frac{1}{2}\lambda^{-\sigma}\int_{\mathbb R^N}|v_\lambda|^2+\frac{b}{4}(\int_{\mathbb R^N}|\nabla v\lambda|^2)^2\\
&\mbox{}& \quad 
-\frac{1}{q}\lambda^{-\sigma}\int_{\mathbb R^N}|v_\lambda|^q-\frac{1}{2^*}\int_{\mathbb R^N}|v_\lambda|^{2^*}\\
&=&(\frac{1}{2}-\frac{1}{2^*})a\int_{\mathbb R^N}|v_\lambda|^2+(\frac{1}{4}-\frac{1}{2^*})b(\int_{\mathbb R^N}|\nabla v_\lambda|^2)^2\\
&\ge & \frac{1}{N}a\int_{\mathbb R^N}|\nabla v_\lambda|^2.
\end{array}
$$
On the other hand, it follows from Lemma  \ref{l28} below  that $m_\lambda\le C<+\infty$ for large $\lambda>0$, and hence $\{v_\lambda\}$ is bounded in $D^{1,2}(\mathbb R^N)$ and $H^1(\mathbb R^N)$. 
\end{proof}

Next we obtain an estimation of the least energy.

\begin{lemma}\label{l28}
There exists a constant $C=C(q)>0$ such that  for all large  $\lambda>0$, 
\begin{equation}\label{e413}
m_\lambda\le\left\{\begin{array}{lclcl} m_\infty-C(\lambda\ln\lambda)^{-\frac{4-q}{q-2}}  &\text{if}&N=4,\smallskip\\
m_\infty-C\lambda^{-\frac{6-q}{2(q-4)}} &\text{if}&N=3 \  \text{and} \  q>4.\end{array}\right.
\end{equation}
\end{lemma}

\begin{proof}
Let  $\rho>0$, $R\gg 1$ be a large parameter and $\eta_R\in C_0^\infty(\mathbb R)$ is a cut-off function such that $\eta_R(r)=1$ for $|r|<R$, $0<\eta_R(r)<1$ for $R<|r|<2R$, $\eta_R(r)=0$ for $|r|>2R$ and $|\eta'_R(r)\le 2/R$.  

For $\ell\gg 1$, a straightforward computation shows that 
\begin{equation}\label{e414}
\int_{\mathbb R^N}|\nabla (\eta_\ell W_1)|^2=S^{\frac{N}{2}}+O(\ell^{-(N-2)})=
\left\{ 
\begin{array}{ccl} S^{\frac{N}{2}}+O(\ell^{-2}) &\text{if}& N=4,\\
S^{\frac{N}{2}}+O(\ell^{-1})   &\text{if}& N=3.
\end{array}
\right.
\end{equation}
\begin{equation}\label{e415}
\int_{\mathbb R^N} |\eta_\ell W_1|^{2^*}=S^{\frac{N}{2}}+O(\ell^{-N}),
\end{equation}
\begin{equation}\label{e416}
\int_{\mathbb R^N}|\eta_\ell W_1|^2=\left\{ 
\begin{array}{ccl} \ln \ell(1+o(1)) &\text{if}& N=4,\\
\ell(1+o(1)) &\text{if}& N=3.
\end{array}\right.
\end{equation}
By Lemma \ref{l24-2}, we find
\begin{align}\label{e417}
m_\lambda&\le \sup_{t\ge 0}J_\lambda((\eta_R\tilde W_\rho)_t)=J_\lambda((\eta_R\tilde W_\rho)_{t_\lambda})\\
&\le \sup_{t\ge 0}\left(\frac{at^{N-2}}{2}\int_{\mathbb R^N}|\nabla(\eta_R\tilde W_\rho)|^2+\frac{bt^{2(N-2)}}{4}(\int_{\R^N}|\nabla(\eta_R\tilde W_\rho)|^2)^2-\frac{t^{N}}{2^*}\int_{\mathbb R^N}|\eta_R\tilde W_\rho|^{2^*}\right)\nonumber\\
&\quad  -\lambda^{-\sigma}
t_\lambda^N\left[\int_{\mathbb R^N}\frac{1}{q}|\eta_R\tilde W_\rho|^q-\frac{1}{2}|\eta_R\tilde W_\rho|^2\right]\nonumber\\
&=  (I)-\lambda^{-\sigma} (II),\nonumber
\end{align}
where $t_\lambda>0$ is the unique solution of the following equation
$$
\begin{array}{rl}
&\frac{a}{2^*t_\lambda^2}\int_{\R^N}|\nabla(\eta_R\tilde W_\rho)|^2+\frac{b}{2^*t_\lambda^{4-N}}(\int_{\R^N}|\nabla(\eta_R\tilde W_\rho)|^2)^2\\
&=\frac{1}{2^*}\int_{\R^N}|\eta_R\tilde W_\rho|^{2^*}
+\lambda^{-\frac{2^*-q}{q-2}}\left[\frac{1}{q}\int_{\R^N}|\eta_R\tilde W_\rho|^q-\frac{1}{2}\int_{\R^N}|\eta_R\tilde W_\rho|^2\right].
\end{array}
$$
Since 
$$
\int_{\mathbb R^N}|\eta_R\tilde W_\rho|^q\le \left(\int_{\mathbb R^N}|\eta_R\tilde W_\rho|^2\right)^{\frac{2^*-q}{2^*-2}}\left(\int_{\mathbb R^N}|\eta_R\tilde W_\rho|^{2^*}\right)^{\frac{q-2}{2^*-2}},
$$
it follows that 
$$
\frac{\frac{1}{q}\int_{\mathbb R^N}|\eta_R\tilde W_\rho|^q-\frac{1}{2}\int_{\mathbb R^N}|\eta_R\tilde W_\rho|^2}{\int_{\mathbb R^N}|\eta_R\tilde W_\rho|^{2^*}}\le \sup_{x\ge 0}\left[\frac{1}{q}x^{\frac{2^*-q}{2^*-2}}-\frac{1}{2}x\right]
=\frac{q-2}{2(2^*-q)}\left(\frac{2(2^*-q)}{q(2^*-2)}\right)^{\frac{2^*-2}{q-2}}.
$$

Set $\ell=R/\rho$ and $\tilde W_1(x)=W_1(\gamma_Nx)$,  then
$$
\begin{array}{lcl}
\int_{\mathbb R^N}|\nabla(\eta_\ell\tilde W_1)|^2&=&\gamma^{-(N-2)}\int_{\mathbb R^N}|\nabla(\eta_{\ell\gamma}W_1)|^2\\
&=&\gamma^{-(N-2)}S^{\frac{N}{2}}+O(\ell^{-(N-2)})\\
&=&\int_{\mathbb R^N}|\nabla \tilde W_1|^2+O(\ell^{-(N-2)}).
\end{array}
$$
We also deduce
$$
\int_{\mathbb R^N}|\eta_\ell\tilde W_1|^{2^*}=\gamma^{-N}\int_{\mathbb R^N}|\eta_{\ell\gamma}W_1|^{2^*}=\gamma^{-N}S^{\frac{N}{2}}+O(\ell^{-N})=\int_{\mathbb R^N}|\tilde W_1|^{2^*}+O(\ell^{-N}),
$$
$$
\int_{\mathbb R^N}|\eta_\ell \tilde W_1|^q=\gamma^{-N}\int_{\mathbb R^N}|\eta_{\ell\gamma}W_1|^q=\gamma^{-N}\int_{\mathbb R^N}|W_1|^q+o(1)=\int_{\mathbb R^N}|\tilde W_1|^q+o(1),
$$
$$
\int_{\mathbb R^N}|\eta_\ell\tilde W_1|^2=\gamma^{-N}\int_{\mathbb R^N}\eta_{\ell\gamma}W_1|^2=\left\{ 
\begin{array}{ccl} \ln (\ell\gamma)(\gamma^{-4}+o(1)) &\text{if}& N=4,\\
\ell(\gamma^{-2}+o(1)) &\text{if}& N=3.
\end{array}\right.
$$ 
Therefore
$$
\begin{array}{lcl}
(I)&=&\sup_{t>0}\left(\frac{at^{N-2}}{2}\int_{\mathbb R^N}|\nabla(\eta_\ell\tilde W_1)|^2+\frac{bt^{2(N-2)}}{4}(\int_{\mathbb R^N}|\nabla(\eta_\ell \tilde W_1)|^2)^2-\frac{t^{N}}{2^*}\int_{\mathbb R^N}|\eta_\ell\tilde W_1|^{2^*}\right)\\
&=&a(\frac{1}{2}-\frac{1}{2^*})t_\ell^{N-2}\int_{\mathbb R^N}|\nabla(\eta_\ell\tilde W_1)|^2+b(\frac{1}{4}-\frac{1}{2^*})t_\ell^{2(N-2)}(\int_{\mathbb R^N}|\nabla(\eta_\ell \tilde W_1)|^2)^2,
\end{array}
$$
where $t_\ell>0$ is given by 
$$
at_\ell^{N-2}\int_{\mathbb R^N}|\nabla (\eta_\ell \tilde W_1)|^2+bt_\ell^{2(N-2)}(\int_{\mathbb R^N}|\nabla(\eta_\ell \tilde W_1)|^2)^2=t_\ell^N\int_{\mathbb R^N}|\eta_\ell\tilde W_1|^{2^*}.
$$

If $N=4$, then $t=t_\ell$ satisfies
$$
at^2\int_{\mathbb R^N}|\nabla(\eta_\ell\tilde W_1)|^2+bt^4(\int_{\mathbb R^N}|\nabla(\eta_\ell\tilde W_1)|^2)^2=t^4\int_{\mathbb R^N}|\eta_\ell\tilde W_1|^4.
$$
Hence, we have 
$$
\begin{array}{lcl}
t_\ell^2&=&\frac{a\int_{\mathbb R^N}|\nabla(\eta_\ell\tilde W_1)|^2}{\int_{\mathbb R^N}|\eta_\ell\tilde W_1|^4-b(\int_{\mathbb R^N}|\nabla(\eta_\ell\tilde W_1)|^2)^2}\\
&=&\frac{a\int_{\mathbb R^N}|\nabla\tilde W_1|^2}{\int_{\mathbb R^N}|\tilde W_1|^4-b(\int_{\mathbb R^N}|\nabla\tilde W_1|^2)^2}+O(\ell^{-2})\\
&=&t_\infty^2+O(\ell^{-2}).
\end{array}
$$
$$
\begin{array}{lcl}
\int_{\mathbb R^N}|\eta_\ell \tilde W_1|^4-b(\int_{\mathbb R^N}|\nabla(\eta_\ell\tilde W_1)|^2)^2&=&\gamma^{-4}S^2+O(\ell^{-4})-b(\gamma^{-2}S^2+O(\ell^{-2}))^2\\
&=&\gamma^{-4}S^2(1-bS^2)+O(\ell^{-2}).
\end{array}
$$
Therefore, we obtain
$$
\begin{array}{lcl}
(I)&=&\frac{a}{N}t_\ell^2\int_{\mathbb R^N}|\nabla(\eta_\ell\tilde W_1)|^2
=\frac{a}{4}\frac{a(\int_{\mathbb R^N}|\nabla(\eta_\ell\tilde W)|^2)^2}{\int_{\mathbb R^N}|\eta_\ell\tilde W_1|^4-b(\int_{\mathbb R^N}|\nabla(\eta_\ell \tilde W_1)|^2)^2}\\
&=&\frac{a^2}{4}\frac{(\gamma^{-2}S^{2}+O(\ell^{-2}))^2}{\gamma^{-4}S^2(1-bS^2)+O(\ell^{-2})}\\
&=&\frac{a^2S^{2}}{4(1-bS^2)}+O(\ell^{-2}).\\
&=&=\frac{a}{4}\frac{a(\int_{\mathbb R^N}|\nabla\tilde W|^2)^2}{\int_{\mathbb R^N}|\tilde W_1|^4-b(\int_{\mathbb R^N}|\nabla\tilde W_1|^2)^2}+O(\ell^{-2})\\
&=&m_\infty+O(\ell^{-2}).
\end{array}
$$

If $N=3$, then $t=t_\ell$ satisfies 
$$
t^2\int_{\mathbb R^N}|\eta_\ell \tilde W_1|^6-bt(\int_{\mathbb R^N}|\nabla(\eta_\ell \tilde W_1)|^2)^2-a\int_{\mathbb R^N}|\nabla(\eta_\ell\tilde W_1)|^2=0.
$$
Therefore,  we get 
$$
\begin{array}{lcl}
t_\ell&=&\frac{b(\int_{\mathbb R^N}|\nabla(\eta_\ell \tilde W_1)|^2)^2+\sqrt{b^2(\int_{\mathbb R^N}|\nabla(\eta_\ell \tilde W_1)|^2)^4+4a\int_{\mathbb R^N}|\nabla(\eta_\ell \tilde W_1)|^2\int_{\mathbb R^N}|\eta_\ell\tilde W_1|^6}}{2\int_{\mathbb R^N}|\eta_\ell\tilde W_1|^6}\\
&=&\frac{b(\int_{\mathbb R^N}|\nabla\tilde W_1|^2)^2+\sqrt{b^2(\int_{\mathbb R^N}|\nabla\tilde W_1|^2)^4+4a\int_{\mathbb R^N}|\nabla \tilde W_1|^2\int_{\mathbb R^N}|\tilde W_1|^6}}{2\int_{\mathbb R^N}|\tilde W_1|^6}+O(\ell^{-1})\\
&=&t_\infty+O(\ell^{-1}).
\end{array}
$$
Thus, we obtain 
$$
\begin{array}{lcl}
(I)&=&\sup_{t\ge 0}\frac{a}{2}t\int_{\R^N}|\nabla(\eta_\ell\tilde W_1)|^2+\frac{b}{4}t^2(\int_{\R^N}|\nabla(\eta_\ell\tilde W_1)|^2)^2-\frac{1}{2^*}t^3\int_{\R^N}|\eta_\ell\tilde W_1|^{2^*}\\
&=&\frac{a}{2}t_\ell\int_{\R^N}|\nabla(\eta_\ell\tilde W_1)|^2+\frac{b}{4}t_\ell^2(\int_{\R^N}|\nabla(\eta_\ell\tilde W_1)|^2)^2-\frac{1}{2^*}t_\ell^3\int_{\R^N}|\eta_\ell\tilde W_1|^{2^*}\\
&=&\frac{a}{2}t_\infty\int_{\R^N}|\nabla\tilde W_1|^2+\frac{b}{4}t_\infty^2(\int_{\R^N}|\nabla\tilde W_1|^2)^2-\frac{1}{2^*}t_\infty^3\int_{\R^N}|\tilde W_1|^{2^*}+O(\ell^{-1})\\
&=&\sup_{t\ge 0}J_0((\tilde W_1)_t)+O(\ell^{-1})\\
&=&m_\infty+O(\ell^{-1}).
\end{array}
$$
On the other hand, we have 
$$
\begin{array}{lcl}
(II)&=&t_\lambda^N\left[\frac{1}{q}\int_{\R^N}|\eta_R\tilde W_\rho|^q-\frac{1}{2}\int_{\R^N}|\eta_R\tilde W_\rho|^2\right]\\
&=&t_\lambda^N\left[\frac{1}{q}\rho^{\frac{2N-q(N-2)}{2}}\int_{\R^N}|\eta_\ell \tilde W_1|^q-\frac{1}{2}\rho^2\int_{\R^N}|\eta_\ell \tilde W_1|^2\right].
\end{array}
$$
Let $h(\rho):=\frac{1}{q}\rho^{\frac{2N-q(N-2)}{2}}\int_{\R^N}|\eta_\ell \tilde W_1|^q-\frac{1}{2}\rho^2\int_{\R^N}|\eta_\ell \tilde W_1|^2$. Then $h(\rho)$ take its maximum value $\varphi(\rho_\ell)$ at the unique point $\rho_\ell>0$, and 
$$
\begin{array}{lcl}
\sup_{\rho>0}h(\rho)&=&h(\rho_\ell)\\
&=&\frac{(q-2)(N-2)}{4q}\left[\frac{2N-q(N-2)}{2q}\right]^{\frac{2N-q(N-2)}{(q-2)(N-2)}}
\left(\frac{\|\eta_\ell \tilde W_1\|_q^{q(2^*-2)}}{\|\eta_\ell \tilde W_1\|_2^{2(2^*-q)}}\right)^{\frac{1}{q-2}}.
\end{array}
$$
Then we obtain
$$
(II)=t_\lambda^Nh(\rho_\ell)\ge 2C_q\|\eta_\ell \tilde W_1\|_2^{-\frac{2(2^*-q)}{q-2}}\ge \left\{ 
\begin{array}{ccl} C_q(\ln \ell)^{-\frac{4-q}{q-2}} &\text{if}& N=4,\\
C_q\ell^{-\frac{6-q}{q-2}} &\text{if}& N=3.
\end{array}\right.
$$ 

For the rest of the proof we consider separately the cases $N=4$ and $N=3$.

\smallskip
\paragraph{\sc Case $N=4$.}  In this case, we have 
$$
m_\lambda\le m_\infty+O(\ell^{-2})-C_q\lambda^{-\frac{4-q}{q-2}}(\ln\ell)^{-\frac{4-q}{q-2}}.
$$
Take $\ell=\lambda^M$. Then 
$$
m_\lambda\le m_\infty+O(\lambda^{-2M})-C_qM^{-\frac{4-q}{q-2}}(\lambda\ln\lambda)^{-\frac{4-q}{q-2}}.
$$
If $2M>\frac{4-q}{q-2}$, then for large $\lambda>0$, we have
\begin{equation}\label{e418}
m_\lambda\le m_\infty-\frac{1}{2}C_qM^{-\frac{4-q}{q-2}}(\lambda\ln\lambda)^{-\frac{4-q}{q-2}}.
\end{equation}
Thus, if $N=4$, the result of Lemma \ref{l28} is proved by choosing 
$$
\varpi=\frac{1}{2}C_qM^{-\frac{4-q}{q-2}}.
$$

\paragraph{\sc Case $N=3$.} In this case, we have 
$$
m_\lambda\le m_\infty+O(\ell^{-1})-C_q\lambda^{-\frac{6-q}{q-2}}\ell^{-\frac{6-q}{q-2}}.
$$
Take $\ell=\delta^{-1}\lambda^\tau$. Then 
$$
m_\lambda\le m_\infty+\delta O(\lambda^{-\tau})-C_q\delta^{\frac{6-q}{q-2}}\lambda^{-(1+\tau)\frac{6-q}{q-2}}.
$$
Let $\tau>0$ be such that $\tau=(1+\tau)\frac{6-q}{q-2}$, that is, $\tau=\frac{6-q}{2(q-4)}$. 

Since $q>4$, we have $\frac{6-q}{q-2}<1$, we can choose a small $\delta>0$ such that 
$$
m_\lambda\le m_\infty-\frac{1}{2}C_q\delta^{\frac{6-q}{q-2}}\lambda^{-\frac{6-q}{2(q-4)}}.
$$
and take
$$
\varpi=\frac{1}{2}C_q\delta^{\frac{6-q}{q-2}},
$$
which finished the proof in the case $N=3$.
\end{proof}

\begin{corollary}\label{c29}
Let $\delta_\lambda:=m_\infty-m_\lambda$, then 
$$
\lambda^{-\frac{2^*-q}{q-2}}\gtrsim \delta_\lambda\gtrsim 
\left\{\begin{array}{lcl} 
 (\lambda\ln\lambda)^{-\frac{4-q}{q-2}}&\text{if}&N=4,\\
 \lambda^{-\frac{6-q}{2(q-4)}} &\text{if}& N=3 \ \text{and} \ q\in (4,6).
 \end{array}\right.
 $$
\end{corollary}
\begin{proof}  Arguing as in the proof of Lemma 4.3, it is easy to show that 
$$
\delta_\lambda\lesssim \lambda^{-\frac{2^*-q}{q-2}},
$$
which together with Lemma 4.3 yields the desired conclusion.
\end{proof}

\subsection{Proof of Theorem 2.1 for large  $\lambda$}\label{s42}

We recall the P.-L.~Lions' concentration--compact\-ness lemma, which is at the core of our proof of  Theorem 2.1 as $\lambda\to \infty$.

\begin{lemma}[P.-L.~Lions \cite{Lions-1}]\label{l31}
Let $r>0$ and $2\leq q\leq 2^{*}$. If $(u_{n})$ is bounded in $H^{1}(\mathbb{R}^N)$ and if
$$\sup_{y\in\mathbb{R}^N}\int_{B_{r}(y)}|u_{n}|^{q}\to0\quad\textrm{as}\ n\to\infty,$$
then $u_{n}\to0$ in $L^{p}(\mathbb{R}^N)$ for $2<p<2^*$. Moreover, if $q=2^*$, then $u_{n}\to0$ in $L^{2^{*}}(\mathbb{R}^N)$.
\end{lemma}

Using Lemma \ref{l31}, we establish the following.
 
\begin{lemma}\label{l34}  If $N=4$ or $N=3$ then there exists $\xi_\lambda\in (0,+\infty)$ such that $\xi_\lambda\to 0$ and
$$
v_\lambda-\xi_\lambda^{-\frac{N-2}{2}}v_1(\xi^{-1}_\lambda\cdot)\to 0
$$
in $D^{1,2}(\mathbb R^N)$ and $L^{2^*}(\mathbb R^N)$ as $\lambda\to \infty$, where $v_1$ is given by \eqref{e47}.
\end{lemma}

\begin{proof}
Note that $v_\lambda$ is a positive radially symmetric function, and by Lemma \ref{l24-3}, $\{v_\lambda\}$ is bounded in $H^1(\mathbb R^N)$. Then there exists
coanstant $A\in \R$ and  $0\le v_\infty\in H^1(\mathbb R^N)$ such that as $\lambda\to \infty$, up to a subsequence, we have 
$$
\int_{\R^N}|\nabla v_\lambda|^2\to A^2,
$$ 
\begin{equation}\label{e419}
v_\lambda\rightharpoonup v_\infty   \quad {\rm weakly \ in} \  H^1(\mathbb R^N), \quad v_\lambda\to v_\infty \quad {\rm in} \ L^p(\mathbb R^N) \quad {\rm for \ any} \ p\in (2,2^*),
\end{equation}
and 
\begin{equation}\label{e420}
v_\lambda(x)\to v_\infty(x) \quad \text{a.e.~on $\R^N$},  \qquad v_\lambda\to v_\infty \quad\text{in}  \  L^2_{loc}(\mathbb R^N).
\end{equation}
Moreover, $v_\infty$ verifies the equation
$$
-(a+bA^2)\Delta v=v^{2^*-1}.
$$

Observe that
$$
J_\infty(v_\lambda)=J_\lambda(v_\lambda)+\frac{\lambda^{-\sigma}}{q}\int_{\mathbb R^N}|v_\lambda|^q-\frac{\lambda^{-\sigma}}{2}\int_{\mathbb R^N}|v_\lambda|^2=m_\lambda+o(1)=m_\infty+o(1),
$$
and 
$$
J'_\infty(v_\lambda)v=J'_\lambda(v_\lambda)v+\lambda^{-\sigma}\int_{\mathbb R^N}|v_\lambda|^{q-2}v_\lambda v-\lambda^{-\sigma}\int_{\mathbb R^N}v_\lambda v=o(1).
$$
Therefore, $\{v_\lambda\}$ is a $(PS)$ sequence for $J_\infty$ at level $m_\infty$.

By Lemma \ref{l31} and an argument similar to that in \cite{Xie},  it is standard to show that there exists $\zeta^{(j)}_\lambda\in (0,+\infty)$, $v^{(j)}\in D^{1,2}(\mathbb R^N)$ with $j=1,2,\dots, k$ where $k$  is a non-negative integer, such that
\begin{equation}\label{e421}
v_\lambda=v_\infty+\sum_{j=1}^k(\zeta^{(j)}_\lambda)^{-\frac{N-2}{2}}v^{(j)}((\zeta^{(j)}_\lambda)^{-1} x)+\tilde v_\lambda,
\end{equation}
where $\tilde v_\lambda\to 0$ in $D^{1,2}(\mathbb R^N)$, $\zeta_\lambda^{(j)}\to 0$ as $\lambda\to \infty$, and  $v^{(j)}$ are nontrivial solutions of the equation 
$$
-(a+bA^2)\Delta v=v^{2^*-1}.
$$
Moreover, we have 
\begin{equation}\label{e422}
A^2= \|v_\infty\|^2_{D^{1,2}(\mathbb R^N)}+\sum_{j=1}^k\|v^{(j)}\|^2_{D^{1,2}(\mathbb R^N)},
\end{equation}
and 
\begin{equation}\label{e423}
m_\infty=J^A_\infty(v_\infty)+\sum_{j=1}^kJ^A_\infty(v^{(j)}),
\end{equation}
where 
$$
J_\infty^A(v)=(\frac{a}{2}+\frac{bA^2}{4})\int_{\R^N}|\nabla v|^2-\frac{1}{2^*}\int_{\R^N}|v|^{2^*}.
$$
For any solution $v$ of the equation $-(a+bA^2)\Delta v=v^{2^*-1}$, we have 
$$
(a+bA^2)\int_{\R^N}|\nabla v|^2=\int_{\R^N}|v|^{2^*}.
$$
Therefore, we obtain
$$
\begin{array}{lcl}
J^A_\infty(v)&=&(\frac{a}{2}+\frac{bA^2}{4})\int_{\R^N}|\nabla v|^2-\frac{1}{2^*}\int_{\R^N}|v|^{2^*}\\
&=&(\frac{a}{2}+\frac{bA^2}{4})\int_{\R^N}|\nabla v|^2-\frac{1}{2^*}(a+bA^2)\int_{\R^N}|\nabla v|^2\\
&=&\frac{1}{N}a\int_{\R^N}|\nabla v|^2+(\frac{1}{4}-\frac{1}{2^*})bA^2\int_{\R^N}|\nabla v|^2\\
&\ge &\frac{1}{N}a\int_{\R^N}|\nabla v|^2+(\frac{1}{4}-\frac{1}{2^*})b(\int_{\R^N}|\nabla v|^2)^2\\
&=&J_\infty(v).
\end{array}
$$
Since $J_\infty(v_\infty)\ge 0$ and $J_\infty(v^{(j)})\ge m_\infty$ for $j=1,2,\cdots, k,$ we conclude that
$J^A_\infty(v_\infty)\ge 0$ and $J^A_\infty(v^{(j)})\ge m_\infty$ for all $j=1,2,\cdots, k.$

If $N=4$ or $3$ then by \eqref{e412} and Fatou's lemma we have 
$$
\|v_\infty\|^2_2\le \liminf_{\lambda\to \infty}\|v_\lambda\|_2^2=\frac{2(2^*-q)}{q(2^*-2)}\int_{\R^N}|v_\infty|^q<\infty.
$$
Note that $v_\infty\notin L^2(\R^N)$ whenever $v_\infty\not=0$,  therefore,  $v_\infty=0$ and hence $k=1$. Thus, we obtain  
$J_\infty(v^{(1)})=m_\infty$ and hence $v^{(1)}=v_\rho$ for some $\rho\in (0,+\infty)$. Therefore, we conclude that 
$$v_\lambda-\xi_\lambda^{-\frac{N-2}{2}}v_1(\xi_\lambda^{-1}\cdot )\to 0
$$ in $D^{1,2}(\mathbb R^N)$ as $\lambda\to 0$, where $v_1$ is given by \eqref{e47} and 
$\xi_\lambda:=\rho\zeta_\lambda^{(1)}\in (0,+\infty)$ satisfying $\xi_\lambda\to 0$ as $\lambda\to 0$.
Moreover, $A^2=\lim_{\lambda\to 0}\int_{\R^N}|\nabla v_\lambda|^2=\int_{\R^N}|\nabla v^{(1)}|^2$, we conclude 
that $v^{(1)}$ is a solution of the equation
$-(a+b\int_{\R^N}|\nabla v|^2)\Delta v=v^{2^*-1}$. 
\end{proof}

We perform an additional rescaling
\begin{equation}\label{e424}
w(x)=\xi_\lambda^{\frac{N-2}{2}} v(\xi_\lambda x), 
\end{equation}
where $\xi_\lambda\in (0,+\infty)$ is given in Lemma \ref{l34}. This rescaling transforms $(Q_\lambda)$ into an equivalent equation
$$
-(a+b\int_{\mathbb R^N}|\nabla w|^2)\Delta w+\lambda^{-\sigma}\xi_\lambda^2w=\lambda^{-\sigma}\xi_\lambda^{\frac{2N-q(N-2)}{2}}w^{q-1}+w^{2^*-1},
\quad\text{in} \ \R^N,
\eqno(\tilde R_\lambda)
$$

The  corresponding energy functional is given by
\begin{equation}\label{e425}
\begin{array}{lcl}
\tilde J_\lambda(w):&=&\frac{1}{2}\int_{\mathbb R^N}a|\nabla w|^2+\lambda^{-\sigma}\xi_\lambda^2|w|^2+\frac{b}{4}(\int_{\mathbb R^N}|\nabla w|^2)^2\\
&\mbox{} &\quad -\frac{1}{q}\lambda^{-\sigma}\xi_\lambda^{\frac{2N-q(N-2)}{2}}\int_{\mathbb R^N}|w|^q-\frac{1}{2^*}\int_{\mathbb R^N}|w|^{2^*}.
\end{array}
\end{equation}

It is straightforward to verify the following.

\begin{lemma}\label{l36}
Let $\lambda>0$, $v$ is the rescaling of $u\in H^1(\mathbb R^N)$ and $w$ is the rescaling  of $v$ given in  \eqref{e41} and \eqref{e424}, respectively. Then:

\begin{enumerate}
	\item[$(a)$] $\|\nabla w\|_2^2= \|\nabla v\|_{2}^{2}=\|\nabla u\|_{2}^{2}$, $\|w\|^{2^*}_{2^*}=\|v\|_{2^*}^{2^*}=\|u\|_{2^*}^{2^*}$, \smallskip

\item[$(b)$]  $\xi_\lambda^{2}\|w\|^2_2=\|v\|_2^2=\lambda^{1+\sigma}\| u\|_2^2$,    $\xi_\lambda^{\frac{2N-q(N-2)}{2}}\|w\|^q_q=\|v\|_q^q=\lambda^{\sigma} \|u\|_q^q$, \smallskip

\item[$(c)$]  $\tilde J_\lambda(w)=J_\lambda(v)=I_\lambda(u)$. 
\end{enumerate}
\end{lemma}

Let $w_\lambda(x)=\xi_\lambda^{\frac{N-2}{2}} v_\lambda(\xi_\lambda x)$ where the $v_\lambda$ is a ground-state of $(R_\lambda)$. Then   
 it follows from Lemma 4.7(c) that  $w_\lambda$ is a ground state of $(\tilde R_\lambda)$. 
By Lemma \ref{l34} we conclude that 
\begin{equation}\label{e426}
\|\nabla(w_\lambda-v_1)\|_2\to 0, \qquad \|w_\lambda-v_1\|_{2^*}\to 0  \qquad\text{as} \ \lambda\to \infty.
\end{equation}

Note that the corresponding Nehari and Poho\v zaev identities read as follows
$$
\begin{array}{cl}
&a\int_{\mathbb R^N}|\nabla w_\lambda|^2+\lambda^{-\sigma}\xi_\lambda^{2}\int_{\mathbb R^N}|w_\lambda|^2+b\left(\int_{\R^N}|\nabla w_\lambda|^2\right)^2\\
&\quad =\int_{\mathbb R^N}|w_\lambda|^{2^*}+\lambda^{-\sigma}\xi_\lambda^{\frac{2N-q(N-2)}{2}}\int_{\mathbb R^N}|w_\lambda|^q,
\end{array}
$$
and 
$$
\begin{array}{cl}
&\frac{a}{2^*}\int_{\mathbb R^N}|\nabla w_\lambda|^2+\frac{1}{2}
\lambda^{-\sigma}\xi_\lambda^{2}\int_{\mathbb R^N}|w_\lambda|^2+\frac{b}{2^*}\left(\int_{\R^N}|\nabla w_\lambda|^2\right)^2\\
&\quad =\frac{1}{2^*}\int_{\mathbb R^N}|w_\lambda|^{2^*}+\frac{1}{q}\lambda^{-\sigma}\xi_\lambda^{\frac{2N-q(N-2)}{2}}\int_{\mathbb R^N}|w_\lambda|^q.
\end{array}
$$
We conclude that 
$$
\left(\frac{1}{2}-\frac{1}{2^*}\right)\lambda^{-\sigma}\xi_\lambda^{2}\int_{\mathbb R^N}|w_\lambda|^2
 =\left(\frac{1}{q}-\frac{1}{2^*}\right)\lambda^{-\sigma}\xi_\lambda^{\frac{2N-q(N-2)}{2}}\int_{\mathbb R^N}|w_\lambda|^q.
$$
Thus, we obtain
\begin{equation}\label{e427}
\xi_\lambda^{\frac{(N-2)(q-2)}{2}}\int_{\mathbb R^N}|w_\lambda|^2=\frac{2(2^*-q)}{q(2^*-2)}\int_{\mathbb R^N}|w_\lambda|^q.
\end{equation}
To control the norm $\|w_\lambda\|_2$ from below, we give the following  estimate:
\begin{lemma}\label{l37}
There exists a constant $C>0$ such that 
\begin{equation}\label{e428}
w_\lambda(x)\ge C\varpi_\lambda^{\frac{N-2}{2}}|x|^{-(N-2)}\exp(-\varpi_\lambda^{-\frac{1}{2}}\lambda^{-\frac{\sigma}{2}}\xi_\lambda |x|), \quad |x|\ge 1,
\end{equation}
where $\varpi_\lambda=a+b\int_{\mathbb R^N}|\nabla w_\lambda|^2$.
\end{lemma}

\begin{proof} 
It is easy to see that  $\tilde w_\lambda(x)=w_\lambda(\sqrt {\varpi_\lambda} x)$ satisfies the following 
$$
-\Delta \tilde w_\lambda +\lambda^{-\sigma}\xi_\lambda^2\tilde w_\lambda=\lambda^{-\sigma}\xi_\lambda^{\frac{2N-q(N-2)}{2}}\tilde w_\lambda^{q-1}+\tilde w_\lambda^{2^*-1}>0.
$$
Arguing as in the proof in \cite[Lemma 4.8]{Moroz-1}, we show that
$$
\tilde w _\lambda(x)\ge C|x|^{-(N-2)}\exp(-\lambda^{-\frac{\sigma}{2}}\xi_\lambda |x|), \quad |x|\ge 1.
$$
Therefore, we obtain
$$
w_\lambda(x)=\tilde w_\lambda(\frac{x}{\sqrt{\varpi_\lambda}})\ge C\varpi_\lambda^{\frac{N-2}{2}}|x|^{-(N-2)}\exp(-\varpi_\lambda^{-\frac{1}{2}}\lambda^{-\frac{\sigma}{2}}\xi_\lambda |x|),
\quad |x|\ge 1.
$$
The proof is complete. 
\end{proof}

Since $0<C_1<\varpi_\lambda<C_2<+\infty$ for some constants $C_1, C_2$ which are independent of $\lambda>0$, as consequences of the above lemma, we have the following two lemmas.

\begin{lemma}\label{l38}
If $N=3$, then $\|w_\lambda\|_2^2\gtrsim \lambda^{\frac{\sigma}{2}}\xi_\lambda^{-1}$.
\end{lemma}

\begin{lemma}\label{l39}
If $N=4$, then $\|w_\lambda\|_2^2\gtrsim  - \ln(\lambda^{-\sigma}\xi_\lambda^{2})$.
\end{lemma}

To prove our main result, the key point is to show the boundedness of $\|w_\lambda\|_q$. 

\begin{lemma}\label{l311}
If $N=3,4$ and  $\frac{N}{N-2}<s<2^*$, then $\|w_\lambda\|_s^s\sim 1$ as $\lambda\to \infty$. Furthermore, $w_\lambda\to v_1$ in $L^s(\mathbb R^N)$ as $\lambda\to \infty$. 
\end{lemma}

\begin{proof}
By \eqref{e426},  we have $w_\lambda\to v_1$ in $L^{2^*}(\mathbb R^N)$. Then, as in \cite[Lemma 4.6]{Moroz-1}, using the embeddings $L^{2^*}(B_1)\hookrightarrow L^s(B_1)$ we prove that
$\liminf_{\lambda\to \infty}\|w_\lambda\|_s^s>0$.  
 
 On the other hand,  
arguing as in \cite[Propositon 3.1]{Akahori-2}, we show that there exists a constant $C>0$ such that for all large $\lambda>0$, 
\begin{equation}\label{e429}
w_\lambda(x)\le \frac{C}{(1+|x|)^{N-2}}, \qquad \forall x\in \mathbb R^N,
\end{equation}
which together with the fact that $s>\frac{N}{N-2}$ implies that $w_\lambda$ is bounded in $L^s(\mathbb R^N)$  uniformly for large $\lambda>0$,  and 
 by the dominated convergence theorem $w_\lambda\to v_1$ in $L^s(\mathbb R^N)$ as $\lambda\to \infty$. 
\end{proof}

\begin{proof}[Proof of Theorem 2.1 for large $\lambda$]   
We first note that for a result similar to Lemma \ref{l24-3} holds for $w_\lambda$ and $\tilde J_\lambda$.  
By \eqref{e411}, we obtain
\begin{equation}\label{e430}
\begin{array}{cl}
m_\infty&\le \sup_{t\ge 0}J_\infty((w_\lambda)_t)=J_\infty((w_\lambda)_{t_\lambda})\\
&\le \sup_{t\ge 0} \tilde J_\lambda((w_\lambda)_t)+\lambda^{-\sigma}t_\lambda^{N}\left\{\frac{1}{q}\xi_\lambda^{\frac{2N-q(N-2)}{2}}\int_{\mathbb R^N}|w_\lambda|^q-\frac{1}{2}\xi_\lambda^{2}\int_{\mathbb R^N}|w_\lambda|^2\right\}\\
&\le m_\lambda+\lambda^{-\sigma}t_\lambda^{N}\frac{1}{q}\xi_\lambda^{\frac{2N-q(N-2)}{2}}\int_{\R^N}|w_\lambda|^q,
\end{array}
\end{equation}  
where $t_\lambda>0$ is such that $(w_\lambda)_{t_\lambda}\in \mathcal{N}_\infty,$ and more precisely, 
$$
t^2_\lambda=\left\{\begin{array}{lcl}
\frac{b(\int_{\R^N}|\nabla w_\lambda|^2)^2+\sqrt{b^2(\int_{\R^N}|\nabla w_\lambda|^2)^4+4a\int_{\R^N}|w_\lambda|^{2^*}\int_{\R^N}|\nabla w_\lambda|^2}}
{2\int_{\R^N}|w_\lambda|^{2^*}}, \ &{\rm if} &  N=3,\\
\frac{a\int_{\R^N}|\nabla w_\lambda|^2}{\int_{\R^N}|w_\lambda|^{2^*}-b(\int_{\R^N}|\nabla w_\lambda|^2)^2}, \ &{\rm if} & \  N=4,
\end{array}\right.
$$
and hence by \eqref{e48} and \eqref{e426}, as $\lambda\to \infty$, we have 
$$
t^2_\lambda\to \left\{\begin{array}{lcl}
\frac{b(\int_{\R^N}|\nabla v_1|^2)^2+\sqrt{b^2(\int_{\R^N}|\nabla v_1|^2)^4+4a\int_{\R^N}|v_1|^{2^*}\int_{\R^N}|\nabla v_1|^2}}
{2\int_{\R^N}|v_1|^{2^*}}, \ &{\rm if} &  N=3,\\
\frac{a\int_{\R^N}|\nabla v_1|^2}{\int_{\R^N}|v_1|^{2^*}-b(\int_{\R^N}|\nabla v_1|^2)^2}=\frac{a\int_{\R^N}|\nabla v_1|^2}{(1-bS^2)\int_{\R^N}|v_1|^{2^*}}, \ &{\rm if} & \  N=4.
\end{array}\right.
$$
The above inequality implies that
$$
\xi_\lambda^{\frac{2N-q(N-2)}{2}}t_\lambda^N\int_{\mathbb R^N}|w_\lambda|^q\ge \lambda^{\sigma}(m_\infty-m_\lambda).
$$
Hence, by Corollary \ref{c29}, we obtain
\begin{equation}\label{e431}
\xi_\lambda^{\frac{2N-q(N-2)}{2}}\int_{\mathbb R^N}|w_\lambda|^q\gtrsim 
\left\{\begin{array}{ll} 
 (\ln\lambda)^{-\frac{4-q}{q-2}}  &\text{if}\ N=4,\smallskip\\
 \lambda^{-\frac{(6-q)^2}{2(q-2)(q-4)}}   &\text{if}\  N=3.
 \end{array}\right.
\end{equation}  
Therefore, by Lemma \ref{l311}, we have 
\begin{equation}\label{e432}
\xi_\lambda\gtrsim  \left\{\begin{array}{ll} 
 (\ln\lambda)^{-\frac{1}{q-2}}&\text{if}\ N=4,\\
 \lambda^{-\frac{6-q}{(q-2)(q-4)}}&\text{if}\ N=3.
 \end{array}\right.
\end{equation}
On the other hand, if  $ N=3$, then by  \eqref{e427} and Lemma \ref{l38} and Lemma \ref{l311}, we have 
$$
\xi_\lambda^{\frac{q-2}{2}}\lesssim \frac{1}{\|w_\lambda\|_2^2}\lesssim\lambda^{-\frac{\sigma}{2}}\xi_\lambda.
$$
Then, observing that $\sigma=\frac{2^*-2}{q-2}=\frac{4}{q-2}$,  for $q\in (4,6)$ we obtain 
\begin{equation}\label{e433}
\xi_\lambda\lesssim \lambda^{-\frac{\sigma}{q-4}}= \lambda^{-\frac{6-q}{(q-2)(q-4)}}.
\end{equation}
If $N=4$, then by \eqref{e427} and Lemma \ref{l39} and Lemma \ref{l311},  we have 
$$
\xi_\lambda^{q-2}\lesssim \frac{1}{\|w_\lambda\|_2^2}\lesssim \frac{1}{-\ln(\lambda^{-\sigma}\xi_\lambda^{2})}\lesssim (\ln\lambda)^{-1}, 
$$
here we have used the fact that $\xi_\lambda\to 0$ as $\lambda\to +\infty$ and for large $\lambda>0$, 
$$
-\ln(\lambda^{-\sigma}\xi_\lambda^{2})=\sigma\ln{\lambda}-2\ln\xi_\lambda\ge \sigma\ln\lambda.
$$
Thus,  we obtain 
\begin{equation}\label{e434}
\xi_\lambda\lesssim   (\ln{\lambda})^{-\frac{1}{q-2}}.
\end{equation}
Thus, it follows from \eqref{e430}, \eqref{e433}, \eqref{e434} and Lemma \ref{l311} that 
$$
m_\infty-m_\lambda\lesssim \lambda^{-\sigma}\xi_\lambda^{\frac{2N-q(N-2)}{2}}\lesssim 
\left\{\begin{array}{ll} (\lambda\ln\lambda)^{-\frac{4-q}{q-2}}&\text{if}\ N=4,\\
 \lambda^{-\frac{6-q}{2(q-4)}}&\text{if}\ N=3,
 \end{array}\right.
$$
and hence  Corollary \ref{c29} yields that
\begin{equation}\label{e435}
m_\infty-m_\lambda\sim \lambda^{-\sigma}\xi_\lambda^{\frac{2N-q(N-2)}{2}}\sim 
\left\{\begin{array}{ll} (\lambda\ln\lambda)^{-\frac{4-q}{q-2}}&\text{if}\ N=4,\\
 \lambda^{-\frac{6-q}{2(q-4)}}&\text{if}\ N=3.
 \end{array}\right.
\end{equation}
Since
$$
\begin{array}{lcl}
m_\lambda&=&\frac{a}{2}\int_{\R^N}|\nabla w_\lambda|^2+\frac{1}{2}\lambda^{-\sigma}\xi_\lambda^2\int_{\R^N}|w_\lambda|^2+\frac{b}{4}\left(\int_{\R^N}|\nabla w_\lambda|^2\right)^2\\
&\mbox{}& \quad -\frac{1}{q}\lambda^{-\sigma}\xi_\lambda^{\frac{2N-q(N-2)}{2}}\int_{\R^N}|w_\lambda|^q-\frac{1}{2^*}\int_{\R^N}|w_\lambda|^{2^*}.
\end{array}
$$
$$
\begin{array}{cl}
&\frac{a}{2^*}\int_{\R^N}|\nabla w_\lambda|^2+\frac{1}{2}\lambda^{-\sigma}\xi_\lambda^2\int_{\R^N}|w_\lambda|^2+\frac{b}{2^*}\left(\int_{\R^N}|\nabla w_\lambda|^2\right)^2\\
&\quad =\frac{1}{q}\lambda^{-\sigma}\xi_\lambda^{\frac{2N-q(N-2)}{2}}\int_{\R^N}|w_\lambda|^q+\frac{1}{2^*}\int_{\R^N}|w_\lambda|^{2^*}
\end{array}
$$
we get
$$
m_\lambda=a(\frac{1}{2}-\frac{1}{2^*})\int_{\R^N}|\nabla w_\lambda|^2+b(\frac{1}{4}-\frac{1}{2^*})\left(\int_{\R^N}|\nabla w_\lambda|^2\right)^2.
$$
Similarly, we have 
$$
m_\infty=a(\frac{1}{2}-\frac{1}{2^*})\int_{\R^N}|\nabla w_\infty|^2+b(\frac{1}{4}-\frac{1}{2^*})\left(\int_{\R^N}|\nabla w_\infty|^2\right)^2.
$$
Therefore, we obtain
$$
m_\infty-m_\lambda=\left[\frac{a}{N}+b\frac{2^*-4}{2^*4}\left(\int_{\R^N}|\nabla w_\infty|^2+\int_{\R^N}|\nabla w_\lambda|^2\right)\right]\left(\int_{\R^N}|\nabla w_\infty|^2-\int_{\R^N}|\nabla w_\lambda|^2\right).
$$
which together with \eqref{e435} implies that 
\begin{equation}\label{e436}
\|\nabla w_\infty\|_2^2-\|\nabla w_\lambda\|_2^2\sim m_\infty-m_\lambda \sim 
\left\{\begin{array}{ll} (\lambda\ln\lambda)^{-\frac{4-q}{q-2}}  \  \ &\text{if}\ N=4,\\
 \lambda^{-\frac{6-q}{2(q-4)}}&\text{if}\  N=3.
 \end{array}\right.
 \end{equation}
Since
$$
\begin{array}{cl}
&\frac{1}{2^*}\left(\int_{\R^N}|w_\infty|^{2^*}-\int_{\R^N}|w_\lambda|^{2^*}\right)\\
&=-(m_\infty-m_\lambda)\\
&\mbox{}+\left[\frac{a}{2}+\frac{b}{4}\left(\int_{\R^N}|\nabla w_\infty|^2+\int_{\R^N}|\nabla w_\lambda|^2\right)\right]\left(\int_{\R^N}|\nabla w_\infty|^2-\int_{\R^N}|\nabla w_\lambda|^2\right)\\
&\mbox{} +\frac{(q-2)(N-2)}{4q}\lambda^{-\sigma}\xi_\lambda^{\frac{2N-q(N-2)}{2}}\int_{\R^N}|w_\lambda|^q.
\end{array}
$$
It follows from \eqref{e435} and  \eqref{e436} that 
$$
\|w_\infty\|^{2^*}_{2^*}-\|w_\lambda\|^{2^*}_{2^*}=\left\{\begin{array}{ll} O((\lambda\ln\lambda)^{-\frac{4-q}{q-2}})  \  \ &\text{if}\ N=4,\\
 O(\lambda^{-\frac{6-q}{2(q-4)}})&\text{if}\  N=3.
 \end{array}\right.
$$
Finally, by \eqref{e427}, \eqref{e435} and Lemma \ref{l311}, we obtain
$$
\|w_\lambda\|_2^2\sim 
\left\{\begin{array}{ll}
\ln\lambda  &\text{if}\ N=4,\\
\lambda^{\frac{6-q}{2(q-4)}}&\text{if}\  N=3.
\end{array}\right.
$$
Statements on $u_\lambda$ follow from the corresponding results on $v_\lambda$ and $w_\lambda$. This completes the proof of Theorem 2.1 for large $\lambda$.
\end{proof}

\section{Proof of Theorem \ref{t2}}\label{s5}

In this section, we aways assume $N=3$. We divide this section into three subsections, in which we consider the cases (i) $p=q$, (ii) $q<p$, $\lambda>0$  large, and (iii) $q<p$, $\lambda>0$ small, respectively.
If $2<q\le p<2^*$, then the arguments used in Section 3 and Section 4 do not work any more, so we try to transform the nonlocal equation $(P_\lambda)$ into the local equation \eqref{e13} by using a suitable rescaling which is dependent of   
the ground state solution. Unfortunately,  if $b\not=0$, such a rescaling does not  neceessarily transforms a ground state solution into  a ground state solution of  a new equation, which prevents us from deriving a precise energy estimate of the ground state.

\subsection{The case $p=q$}

Let $W\in H^1(\mathbb R^3)$ is the unique positive solution of the equation
$$
-\Delta W +W=W^{p-1},
$$
and $S_p=\|W\|_p^{p-2}$.  Let $u_\lambda$ be a ground state solution of $(P_\lambda)$, then for any $\lambda>0$, there holds
\begin{equation}\label{e51}
u_\lambda(x)=\left(\lambda/2\right)^{\frac{1}{p-2}}W(\lambda^{\frac{1}{2}}\varpi_\lambda^{-\frac{1}{2}}x), 
\end{equation}
where $\varpi_\lambda=a$ if $b=0$, and if $b\not=0$, then
$$
\begin{array}{lcl}
\sqrt{\varpi_\lambda}&=&\frac{1}{2}\left\{\frac{3b(p-2)}{2p}\lambda^{\frac{6-p}{2(p-2)}}(S_p/2)^{\frac{p}{p-2}}+\sqrt{\frac{9b^2(p-2)^2}{4p^2}\lambda^{\frac{6-p}{p-2}}(S_p/2)^{\frac{2p}{p-2}}+4a}   \right\}\\
&=&\left\{\begin{array}{lcl}\lambda^{\frac{6-p}{2(p-2)}}\left(\frac{3b(p-2)}{2p}(S_p/2)^{\frac{p}{p-2}}+\Theta(\lambda^{-\frac{6-p}{p-2}})\right),\   &{\rm as}& \lambda\to \infty,\\
a^{\frac{1}{2}}+\Theta(\lambda^{\frac{6-p}{2(p-2)}}), \  &{\rm as}&    \lambda\to 0,\end{array}\right.
\end{array}
$$
Clearly, we have 
$$
u_\lambda(0)= \left(\lambda/2\right)^{\frac{1}{p-2}}W(0),
$$
and a  direct computation shows that for $b\not=0$,
$$
\begin{array}{lcl}
\|\nabla u_\lambda\|^2_2&=&\lambda^{\frac{6-p}{2(p-2)}}\frac{3(p-2)}{p}\sqrt{\varpi_\lambda}\left(S_p/2\right)^{\frac{p}{p-2}}\\
&=&\left\{\begin{array}{lcl}\lambda^{\frac{6-p}{p-2}}\left(\frac{9b(p-2)^2}{2p^2}(S_p/2)^{\frac{2p}{p-2}}+\Theta(\lambda^{-\frac{6-p}{p-2}})\right),\   &{\rm as}& \lambda\to \infty,\\
\lambda^{\frac{6-p}{2(p-2)}}\left(\frac{3(p-2)}{p}a^{\frac{1}{2}}(S_p/2)^{\frac{p}{p-2}}+\Theta(\lambda^{\frac{6-p}{2(p-2)}})\right), \  &{\rm as}&    \lambda\to 0,\end{array}\right.
\end{array}
$$
$$
\begin{array}{lcl}
\| u_\lambda\|^2_2&=&\lambda^{\frac{10-3p}{2(p-2)}}\frac{6-p}{p}(\sqrt{\varpi_\lambda})^3\left(S_p/2\right)^{\frac{p}{p-2}}\\
&=&\left\{\begin{array}{lcl}\lambda^{\frac{14-3p}{p-2}}\left(\frac{27b^3(p-2)^3(6-p)}{8p^4}(S_p/2)^{\frac{4p}{p-2}}+\Theta(\lambda^{-\frac{6-p}{p-2}})\right),\   &{\rm as}& \lambda\to \infty,\\
\lambda^{\frac{10-3p}{2(p-2)}}\left(\frac{6-p}{p}a^{\frac{3}{2}}(S_p/2)^{\frac{p}{p-2}}+\Theta(\lambda^{\frac{6-p}{2(p-2)}})\right), \  &{\rm as}&    \lambda\to 0,\end{array}\right.
\end{array}
$$
$$
\begin{array}{lcl}
\|u_\lambda\|_p^p&=&\lambda^{\frac{6-p}{2(p-2)}}(\sqrt{\varpi_\lambda})^3\left(S_p/2\right)^{\frac{p}{p-2}}\\
&=&\left\{\begin{array}{lcl}\lambda^{\frac{2(6-p)}{p-2}}\left(\frac{27b^3(p-2)^3}{8p^3}(S_p/2)^{\frac{4p}{p-2}}+\Theta(\lambda^{-\frac{6-p}{p-2}})\right),\   &{\rm as}& \lambda\to \infty,\\
\lambda^{\frac{6-p}{2(p-2)}}\left(a^{\frac{3}{2}}(S_p/2)^{\frac{p}{p-2}}+\Theta(\lambda^{\frac{6-p}{2(p-2)}})\right), \  &{\rm as}&    \lambda\to 0. \end{array}\right.
\end{array}
$$
For $b=0$, we have 
$$
\|\nabla u_\lambda\|^2_2=\lambda^{\frac{6-p}{2(p-2)}}\frac{3(p-2)}{p}\sqrt{\varpi_\lambda}\left(S_p/2\right)^{\frac{p}{p-2}}
=\lambda^{\frac{6-p}{2(p-2)}}\frac{3(p-2)}{p}a^{\frac{1}{2}}\left(S_p/2\right)^{\frac{p}{p-2}},
$$
$$
\| u_\lambda\|^2_2=\lambda^{\frac{10-3p}{2(p-2)}}\frac{6-p}{p}(\sqrt{\varpi_\lambda})^3\left(S_p/2\right)^{\frac{p}{p-2}}
=\lambda^{\frac{10-3p}{2(p-2)}}\frac{6-p}{p}a^{\frac{3}{2}}\left(S_p/2\right)^{\frac{p}{p-2}},
$$
$$
\|u_\lambda\|_p^p=\lambda^{\frac{6-p}{2(p-2)}}(\sqrt{\varpi_\lambda})^3\left(S_p/2\right)^{\frac{p}{p-2}}
=\lambda^{\frac{6-p}{2(p-2)}}a^{\frac{3}{2}}\left(S_p/2\right)^{\frac{p}{p-2}}.
$$

\subsection{The case $q<p$ and $\lambda>0$ is sufficiently large.}
 Let  $u_\lambda$ be a ground state solution of $(P_\lambda)$, and 
\begin{equation}\label{e52}
w_\lambda(x)=\lambda^{-\frac{1}{p-2}}u_\lambda(\lambda^{-\frac{1}{2}}\sqrt{\varpi_\lambda}x),
\  \  \varpi_\lambda=a+b\int_{\R^N}|\nabla u_\lambda|^2.
\end{equation}
Then $w=w_\lambda$ satisies
\begin{equation}\label{e53}
-\Delta w+w=\lambda^{-\frac{p-q}{p-2}}w^{q-1}+w^{p-1},  \   \  {\rm in}  \ \R^N.
\end{equation}
The  corresponding functional is given by
$$
J_\lambda(w)=\frac{1}{2}\int_{\R^N}|\nabla w|^2+|w|^2-\frac{1}{q}\lambda^{-\frac{p-q}{p-2}}\int_{\R^N}|w|^q+\frac{1}{p}\int_{\R^N}|w|^p.
$$
Observe that 
\begin{equation}\label{e54}
\varpi_\lambda=\left\{\begin{array}{lcl}
a+b\lambda^{\frac{6-p}{2(p-2)}}\varpi_\lambda^{\frac{1}{2}}\int_{\R^N}|\nabla w_\lambda|^2,\quad &{\rm if}& \ b\not=0,\\
a, \quad &{\rm if}& \ b=0,
\end{array}\right.
\end{equation}
it follows that 
\begin{equation}\label{e55}
\begin{array}{lcl}
I_\lambda(u)&=&\frac{a}{2}\|\nabla u\|_2^2+\frac{\lambda}{2}\|u\|_2^2+\frac{b}{4}\|\nabla u\|_2^4-\frac{1}{q}\|u\|_q^q-\frac{1}{p}\|u\|_p^p\\
&=&\left\{\begin{array}{lcl}
\lambda^{\frac{6-p}{2(p-2)}}(\sqrt{\varpi_\lambda})^3[J_\lambda(w)+K_\lambda(w)],  \quad &{\rm if}& \  b\not=0,\\
\lambda^{\frac{6-p}{2(p-2)}}a^{\frac{3}{2}}J_\lambda(w), \quad &{\rm if}& \   b=0.
\end{array}\right.
\end{array}
\end{equation}
where 
$$
K_\lambda(w)=\frac{a}{2}\varpi_\lambda^{-1}\|\nabla w\|_2^2+\frac{b}{4}\lambda^{\frac{6-p}{2(p-2)}}\varpi_\lambda^{-\frac{1}{2}}\|\nabla w\|_2^4-\frac{1}{2}\|\nabla w\|_2^2.
$$
Clearly, for any $\varphi\in H^1(\R^N)$, we have 
$$
K'_\lambda(w_\lambda)\varphi=(a\varpi_\lambda^{-1}+b\lambda^{\frac{6-p}{2(p-2)}}\varpi_{\lambda}^{-\frac{1}{2}}\|\nabla w_\lambda\|_2^2-1)\int_{\R^N}\nabla w_\lambda\nabla \varphi
=0,
$$
therefore, $w=w_\lambda$ is a critical point of $K_\lambda(w)$. Therefore, if $u_\lambda$ is a critical point of $I_\lambda$, then 
$w_\lambda$ is a critical point of $J_\lambda$.

On the other hand, assume $w_\lambda$ is a critical point of $J_\lambda$ and $\varpi_\lambda$ is given by \eqref{e54},
then 
\begin{equation}\label{e56}
u_\lambda(x)=\lambda^{\frac{1}{p-2}}w_\lambda(\lambda^{\frac{1}{2}}(\sqrt{\varpi_\lambda})^{-1}x)
\end{equation}
is a critical point of $I_\lambda$.  These observation reduces the problem of finding critical point of $I_\lambda$ to the corresponding problem of finding critical point of $J_\lambda$.

For general $\lambda>0$, the ground state solutions of \eqref{e53} should not be unique. But for large $\lambda>0$, this is not the case.  Arguing as in \cite[Theorem 5.1]{Jeanjean-5} (see also \cite{Akahori-3}), by the implicit function theorem, we can show that  for large $\lambda>0$, \eqref{e53} admits a unique  positive ground state solution, therefore, $(P_\lambda)$ has only one ground state solution for large $\lambda>0$. This also yields that $w_\lambda$ is a ground state solution of \eqref{e53}.

 Let $m_\lambda$ be the least energy of nontrivial solutions of \eqref{e53}. Put
 \begin{equation}\label{e57}
A_\lambda=\int_{\R^N}|\nabla w_\lambda|^2, \  B_\lambda=\int_{\R^N}|w_\lambda|^2, \  C_\lambda=\int_{\R^N}|w_\lambda|^q, \ D_\lambda=\int_{\R^N}|w_\lambda|^p,
\end{equation}
Then the Nehari and Poho\v zaev identities hold true:
\begin{equation}\label{e58}
A_\lambda+B_\lambda=\lambda^{-\frac{p-q}{p-2}}C_\lambda+D_\lambda,
\end{equation}
\begin{equation}\label{e59}
\frac{1}{2^*}A_\lambda+\frac{1}{2}B_\lambda=\frac{1}{q}\lambda^{-\frac{p-q}{p-2}}C_\lambda+\frac{1}{p}D_\lambda,
\end{equation}
As a consequence, it follows that
$$
A_\lambda=\frac{N(q-2)}{2q}\lambda^{-\frac{p-q}{p-2}}C_\lambda+\frac{N(p-2)}{2p}D_\lambda
$$
$$
B_\lambda=\frac{N(2^*-q)}{2^*q}\lambda^{-\frac{p-q}{p-2}}C_\lambda+\frac{N(2^*-p)}{2^*p}D_\lambda
$$
Hence, we get
$$
\begin{array}{rcl}
m_\lambda&=&\frac{1}{2}A_\lambda+\frac{1}{2}B_\lambda-\frac{1}{q}\lambda^{-\frac{p-q}{p-2}}C_\lambda-\frac{1}{p}D_\lambda\\
&=&\frac{N(q-2)}{4q}\lambda^{-\frac{p-q}{p-2}}C_\lambda+\frac{N(p-2)}{4p}D_\lambda
+\frac{N(2^*-q)}{22^*q}\lambda^{-\frac{p-q}{p-2}}C_\lambda+\frac{N(2^*-p)}{22^*p}D_\lambda
-\frac{1}{q}\lambda^{-\frac{p-q}{p-2}}C_\lambda-\frac{1}{p}D_\lambda\\
&=&\frac{q-2}{2}\lambda^{-\frac{p-q}{p-2}}C_\lambda+\frac{p-2}{2}D_\lambda.
\end{array}
$$
In a similar way, we show that
$$
m_\infty=\frac{p-2}{2}D_\infty=\frac{p-2}{2}S_p^{\frac{p}{p-2}}.
$$
Thus, we get
\begin{equation}\label{e510}
m_\lambda-m_\infty=\frac{p-2}{2}(D_\lambda-D_\infty)+\frac{q-2}{2}\lambda^{-\frac{p-q}{p-2}}C_\lambda.
\end{equation}
Arguing as in \cite{MM-1}, it is shown that
\begin{equation}\label{e511}
m_0-m_\lambda\sim \lambda^{-\frac{p-q}{p-2}},  \quad {\rm as} \  \lambda\to \infty.
\end{equation}
Therefore, from \eqref{e511},  we obtain
\begin{equation}\label{e512}
\int_{\R^N}|w_\lambda|^p=D_\infty-\frac{2}{p-2}(m_0-m_\lambda)-\frac{q-2}{p-2}\lambda^{-\frac{p-q}{p-2}}C_\lambda
=S_p^{\frac{p}{p-2}}-\Theta(\lambda^{-\frac{p-q}{p-2}}),
\end{equation}
\begin{equation}\label{e513}
\begin{array}{lcl}
\int_{\R^N}|\nabla w_\lambda|^2&=&\frac{N(q-2)}{2q}\lambda^{-\frac{p-q}{p-2}}C_\lambda+\frac{N(p-2)}{2p}(D_\infty-\frac{2}{p-2}(m_0-m_\lambda)-\frac{q-2}{p-2}\lambda^{-\frac{p-q}{p-2}}C_\lambda)\\
&=&\frac{N(p-2)}{2p}D_\infty-\frac{N}{p}(m_\infty-m_\lambda)-\frac{N(q-2)}{2}(\frac{1}{p}-\frac{1}{q})\lambda^{-\frac{p-q}{p-2}}C_\lambda\\
&=& \frac{N(p-2)}{2p}S_p^{\frac{p}{p-2}}+O(\lambda^{-\frac{p-q}{p-2}}),
\end{array}
\end{equation}
\begin{equation}\label{e514}
\begin{array}{rcl}
\int_{\R^N}|w_\lambda|^2&=&\frac{N(2^*-q)}{2^*q}\lambda^{-\frac{p-q}{p-2}}C_\lambda+\frac{N(2^*-p)}{2^*p}(D_\infty-\frac{2}{p-2}(m_\infty-m_\lambda)-\frac{q-2}{p-2}\lambda^{-\frac{p-q}{p-2}}C_\lambda)\\
&=&\frac{N(2^*-p)}{2^*p}S_p^{\frac{p}{p-2}}-\frac{2N(2^*-p)}{2^*p(p-2)}(m_\infty-m_\lambda)-N(\frac{(2^*-p)(q-2)}{2^*p(p-2)}-\frac{2^*-q}{2^*q})\lambda^{-\frac{p-q}{p-2}}C_\lambda\\
&=&\frac{N(2^*-p)}{2^*p}S_p^{\frac{p}{p-2}}+O(\lambda^{-\frac{p-q}{p-2}}).
\end{array}
\end{equation}
From \eqref{e54} and \eqref{e514}, it follows that for $b\not=0$, 
$$
\begin{array}{lcl}
\sqrt{\varpi_\lambda}&=&\frac{1}{2}\left(b\lambda^{\frac{6-p}{2(p-2)}}\int_{\R^N}|\nabla w_\lambda|^2+\sqrt{b^2\lambda^{\frac{6-p}{p-2}}(\int_{\R^N}|\nabla w_\lambda|^2)^2+4a}\right)\\
&=&\lambda^{\frac{6-p}{2(p-2)}}\int_{\R^N}|\nabla w_\lambda|^2\cdot \frac{1}{2}\left(b+\sqrt{b^2+4a\lambda^{-\frac{6-p}{p-2}}(\int_{\R^N}|\nabla w_\lambda|^2)^{-2}}\right)\\
&=&\lambda^{\frac{6-q}{2(p-2)}}\int_{\R^N}|\nabla w_\lambda|^2\left(b+O(\lambda^{-\frac{6-p}{p-2}})\right)\\
&=&\lambda^{\frac{6-p}{2(p-2)}}\left(\frac{3(p-2)}{2p}S_p^{\frac{p}{p-2}}+O(\lambda^{-\frac{p-q}{p-2}})\right)\left(b+O(\lambda^{-\frac{6-p}{p-2}})\right)\\
&=&\left\{\begin{array}{lcl}
\lambda^{\frac{6-p}{2(p-2)}}\left(\frac{3b(p-2)}{2p}S_p^{\frac{p}{p-2}}+O(\lambda^{-\frac{p-q}{p-2}})\right),   \   \ &{\rm if}&  \  q>2p-6,\\
\lambda^{\frac{6-p}{2(p-2)}}\left(\frac{3b(p-2)}{2p}S_p^{\frac{p}{p-2}}+O(\lambda^{-\frac{6-p}{p-2}})\right),   \   \ &{\rm if}&  \  q\le 2p-6. \end{array} \right.
\end{array}
$$
Thus, by \eqref{e56}, \eqref{e512}, \eqref{e513} and \eqref{e514},  for $b\not=0$, we obtain
$$
\begin{array}{lcl}
\|\nabla u_\lambda\|_2^2&=&\lambda^{\frac{6-p}{2(p-2)}}\sqrt{\varpi_\lambda}\int_{\R^N}|\nabla w_\lambda|^2\\
&=&\lambda^{\frac{6-p}{p-2}}(b+O(\lambda^{-\frac{6-q}{p-2}}))(\int_{\R^N}|\nabla w_\lambda|^2)^2\\
&=&\lambda^{\frac{6-p}{p-2}}(b+O(\lambda^{-\frac{6-q}{p-2}}))(\frac{3(p-2)}{2p}S_p^{\frac{p}{p-2}}+O(\lambda^{-\frac{p-q}{p-2}}))^2\\
&=&\left\{\begin{array}{lcl}
\lambda^{\frac{6-p}{p-2}}\left(\frac{9b(p-2)^2}{4p^2}S_p^{\frac{2p}{p-2}}+O(\lambda^{-\frac{p-q}{p-2}})\right),   \   \ &{\rm if}&  \  q>2p-6,\\
\lambda^{\frac{6-p}{p-2}}\left(\frac{9b(p-2)^2}{4p^2}S_p^{\frac{2p}{p-2}}+O(\lambda^{-\frac{6-p}{p-2}})\right),   \   \ &{\rm if}&  \  q\le 2p-6,  \end{array} \right.
\end{array}
$$
$$
\begin{array}{lcl}
\| u_\lambda\|^2_2&=&\lambda^{\frac{10-3p}{2(p-2)}}(\sqrt{\varpi_\lambda})^3\int_{\R^N}|w_\lambda|^2\\
&=&\lambda^{\frac{10-3p}{2(p-2)}}\cdot \lambda^{\frac{3(6-q)}{2(p-2)}}(b+O(\lambda^{-\frac{6-q}{p-2}}))^3(\int_{\R^N}|\nabla w_\lambda|^2)^3\int_{\R^N}|w_\lambda|^2\\
&=&\lambda^{\frac{14-3p}{p-2}}(b^3+O(\lambda^{-\frac{6-q}{p-2}}))(\frac{3(p-2)}{2p}S_p^{\frac{p}{p-2}}+O(\lambda^{-\frac{p-q}{p-2}}))^3(\frac{6-p}{2p}S_p^{\frac{p}{p-2}}+O(\lambda^{-\frac{p-q}{p-2}}))\\
&=&\left\{\begin{array}{lcl}
\lambda^{\frac{14-3p}{p-2}}\left(\frac{27b^3(p-2)^3(6-p)}{16p^4}S_p^{\frac{4p}{p-2}}+O(\lambda^{-\frac{p-q}{p-2}})\right),   \   \ &{\rm if}&  \  q>2p-6,\\
\lambda^{\frac{14-3p}{p-2}}\left(\frac{27b^3(p-2)^3(6-p)}{16p^4}S_p^{\frac{4p}{p-2}}+O(\lambda^{-\frac{6-p}{p-2}})\right),   \   \ &{\rm if}&  \  q\le 2p-6, \end{array} \right.
\end{array}
$$
$$
\begin{array}{lcl}
\| u_\lambda\|^p_p&=&\lambda^{\frac{6-p}{2(p-2)}}(\sqrt{\varpi_\lambda})^3\int_{\R^N}|w_\lambda|^p\\
&=&\lambda^{\frac{6-p}{2(p-2)}}\cdot \lambda^{\frac{3(6-q)}{2(p-2)}}(b+O(\lambda^{-\frac{6-q}{p-2}}))^3(\int_{\R^N}|\nabla w_\lambda|^2)^3\int_{\R^N}|w_\lambda|^p\\
&=&\lambda^{\frac{2(6-p)}{p-2}}(b^3+O(\lambda^{-\frac{6-q}{p-2}}))(\frac{3(p-2)}{2p}S_p^{\frac{p}{p-2}}+O(\lambda^{-\frac{p-q}{p-2}}))^3(S_p^{\frac{p}{p-2}}-\Theta(\lambda^{-\frac{p-q}{p-2}}))\\
&=&\left\{\begin{array}{lcl}
\lambda^{\frac{2(6-p)}{p-2}}\left(\frac{27b^3(p-2)^3}{8p^3}S_p^{\frac{4p}{p-2}}+O(\lambda^{-\frac{p-q}{p-2}})\right),   \   \ &{\rm if}&  \  q>2p-6,\\
\lambda^{\frac{2(6-p)}{p-2}}\left(\frac{27b^3(p-2)^3}{8p^3}S_p^{\frac{4p}{p-2}}+O(\lambda^{-\frac{6-p}{p-2}})\right),   \   \ &{\rm if}&  \  q\le 2p-6. \end{array} \right.
\end{array}
$$
For $b=0$, noting that $\varpi_\lambda=a$, by \eqref{e56}, \eqref{e512}, \eqref{e513} and \eqref{e514}, we have
$$
\|\nabla u_\lambda\|_2^2=\lambda^{\frac{6-p}{2(p-2)}}\sqrt{\varpi_\lambda}\int_{\R^N}|\nabla w_\lambda|^2
=\lambda^{\frac{6-p}{2(p-2)}}\left(\frac{3(p-2)}{2p}a^{\frac{1}{2}} S_p^{\frac{p}{p-2}}+O(\lambda^{-\frac{p-q}{p-2}})\right),
$$
$$
\| u_\lambda\|^2_2=\lambda^{\frac{10-3p}{2(p-2)}}(\sqrt{\varpi_\lambda})^3\int_{\R^N}|w_\lambda|^2
=  \lambda^{\frac{10-3p}{2(p-2)}}\left( \frac{6-p}{2p}a^{\frac{3}{2}}S_p^{\frac{p}{p-2}}+O(\lambda^{-\frac{p-q}{p-2}})\right),
$$
$$
\| u_\lambda\|^p_p=\lambda^{\frac{6-p}{2(p-2)}}(\sqrt{\varpi_\lambda})^3\int_{\R^N}|w_\lambda|^p
=\lambda^{\frac{6-p}{2(p-2)}}\left(a^{\frac{3}{2}}S_p^{\frac{p}{p-2}}-\Theta(\lambda^{-\frac{p-q}{p-2}})\right).
$$

\subsection{The case $q<p$ and $\lambda>0$ is sufficiently small.}

Let $u_\lambda$ be a ground state solution of $(P_\lambda)$, and 
\begin{equation}\label{e515}
w_\lambda(x)=\lambda^{-\frac{1}{q-2}}u_\lambda(\lambda^{-\frac{1}{2}}\sqrt{\varpi_\lambda}x),
\  \  \varpi_\lambda=a+b\int_{\R^N}|\nabla u_\lambda|^2.
\end{equation}
Then $w=w_\lambda$ satisies
\begin{equation}\label{e516}
-\Delta w+w=w^{q-1}+\lambda^{\frac{p-q}{q-2}}w^{p-1},  \   \  {\rm in}  \ \R^N.
\end{equation}
As before, we show that  $w_\lambda$ is the unique positive solution of \eqref{e516} for small $\lambda>0$. Moreover, we have 
\begin{equation}\label{e517}
\int_{\R^N}|\nabla w_\lambda|^2+\int_{\R^N}|w_\lambda|^2=\int_{\R^N}|w_\lambda|^q+\lambda^{\frac{p-q}{q-2}}\int_{\R^N}|w_\lambda|^p,
\end{equation}
\begin{equation}\label{e518}
\frac{1}{2^*}\int_{\R^N}|\nabla w_\lambda|^2+\frac{1}{2}\int_{\R^N}|w_\lambda|^2=\frac{1}{q}\int_{\R^N}|w_\lambda|^q+\frac{1}{p}\lambda^{\frac{p-q}{q-2}}\int_{\R^N}|w_\lambda|^p,
\end{equation}
Put 
\begin{equation}\label{e519}
A_\lambda=\int_{\R^N}|\nabla w_\lambda|^2, \  B_\lambda=\int_{\R^N}|w_\lambda|^2, \  C_\lambda=\int_{\R^N}|w_\lambda|^q, \ D_\lambda=\int_{\R^N}|w_\lambda|^p,
\end{equation}
then
$$
A_\lambda=\frac{N(q-2)}{2q}C_\lambda+\frac{N(p-2)}{2p}\lambda^{\frac{p-q}{q-2}}D_\lambda,
$$
$$
B_\lambda=\frac{N(2^*-q)}{2^*q}C_\lambda+\frac{N(2^*-p)}{2^*p}\lambda^{\frac{p-q}{q-2}}D_\lambda,
$$
and hence
\begin{equation}\label{e520}
\begin{array}{lcl}
m_\lambda&=&\frac{1}{2}\int_{\R^N}|\nabla w_\lambda|^2+\frac{1}{2}\int_{\R^N}|w_\lambda|^2-\frac{1}{q}\int_{\R^N}|w_\lambda|^q-\frac{1}{p}\lambda^{\frac{p-q}{q-2}}\int_{\R^N}|w_\lambda|^p\\
&=&\frac{q-2}{2}C_\lambda+\frac{p-2}{2}\lambda^{\frac{p-q}{q-2}}D_\lambda.
\end{array}
\end{equation}
In a similar way, we show that
$$
m_0=\frac{q-2}{2}C_0=\frac{q-2}{2}S_q^{\frac{q}{q-2}}, \   \  C_0=S_q^{\frac{q}{q-2}}.
$$
Thus, we obtain
\begin{equation}\label{e521}
m_\lambda-m_0=\frac{q-2}{2}(C_\lambda-C_0)+\frac{p-2}{2}\lambda^{\frac{p-q}{q-2}}D_\lambda.
\end{equation}
On the other hand, as before, it is easy to show that 
\begin{equation}\label{e522}
m_0-m_\lambda\sim \lambda^{\frac{p-q}{q-2}}, \quad {\rm as}  \   \lambda\to 0.
\end{equation}
Therefore, we obtain 
$$
C_0-C_\lambda=\frac{2}{q-2}(m_0-m_\lambda)+\frac{p-2}{q-2}\lambda^{\frac{p-q}{q-2}}D_\lambda\sim \lambda^{\frac{p-q}{q-2}},
$$
that is, 
\begin{equation}\label{e523}
\int_{\R^N}|w_\lambda|^q=C_\lambda=C_0-\Theta(\lambda^{\frac{p-q}{q-2}})=S_q^{\frac{q}{q-2}}-\Theta(\lambda^{\frac{p-q}{q-2}}).
\end{equation}
Therefore, we have 
\begin{equation}\label{e524}
\begin{array}{lcl}
\int_{\R^N}|\nabla w_\lambda|^2&=&\frac{N(q-2)}{2q}(C_0-\frac{2}{q-2}(m_0-m_\lambda)-\frac{p-2}{q-2}\lambda^{\frac{p-q}{q-2}}D_\lambda)+\frac{N(p-2)}{2p}\lambda^{\frac{p-q}{q-2}}D_\lambda\\
&=&\frac{N(q-2)}{2q}C_0-\frac{N}{q}(m_0-m_\lambda)-\frac{N(p-2)}{2}(\frac{1}{q}-\frac{1}{p})\lambda^{\frac{p-q}{q-2}}D_\lambda\\
&=&\frac{N(q-2)}{2q}S_q^{\frac{q}{q-2}}-\Theta(\lambda^{\frac{p-q}{q-2}}).
\end{array}
\end{equation}
$$
\begin{array}{lcl}
\int_{\R^N}|w_\lambda|^2&=&\frac{N(2^*-q)}{2^*q}(C_0-\frac{2}{q-2}(m_0-m_\lambda)-\frac{p-2}{q-2}\lambda^{\frac{p-q}{q-2}}D_\lambda)+\frac{N(2^*-p)}{2^*p}\lambda^{\frac{p-q}{q-2}}D_\lambda\\
&=&\frac{N(2^*-q)}{2^*q}S_q^{\frac{q}{q-2}}-\frac{2N(2^*-q)}{2^*q(q-2)}(m_0-m_\lambda)-\left[\frac{N(2^*-q)(p-2)}{2^*q(q-2)}-\frac{N(2^*-p)}{2^*p}\right]\lambda^{\frac{p-q}{q-2}}D_\lambda.
\end{array}
$$
Notice that 
$$
\frac{(2^*-q)(p-2)}{2^*q(q-2)}-\frac{(2^*-p)}{2^*p}=\frac{p-2}{q-2}\left(\frac{1}{q}-\frac{1}{2^*}\right)-\left(\frac{1}{p}-\frac{1}{2^*}\right)\ge \frac{p-q}{q-2}\left(\frac{1}{p}-\frac{1}{2^*}\right)>0,
$$
we conclude that
\begin{equation}\label{e525}
\int_{\R^N}|w_\lambda|^2=\frac{N(2^*-q)}{2^*q}S_q^{\frac{q}{q-2}}-\Theta(\lambda^{\frac{p-q}{q-2}}).
\end{equation}
Since
$$
\varpi_\lambda=a+b\int_{\R^N}|\nabla u_\lambda|^2
=\left\{\begin{array}{lcl}
a+b\lambda^{\frac{6-q}{2(q-2)}}\sqrt{\varpi_\lambda}\int_{\R^N}|\nabla w_\lambda|^2,\quad &{\rm if}& \ b\not=0,\\
a, \quad &{\rm if}& \  b=0,
\end{array}\right.
$$
by \eqref{e524},  for $b\not=0$, we have 
$$
\begin{array}{rcl}
\sqrt{\varpi_\lambda}&=&
\frac{1}{2}\left(b\lambda^{\frac{6-q}{2(q-2)}}\int_{\R^N}|\nabla w_\lambda|^2+\sqrt{b^2\lambda^{\frac{6-q}{q-2}}(\int_{\R^N}|\nabla w_\lambda|^2)^2+4a}\right)\\
&=&\frac{1}{2}\left(b\lambda^{\frac{6-q}{2(q-2)}}
+\sqrt{b^2\lambda^{\frac{6-q}{q-2}}
+4a(\int_{\R^N}|\nabla w_\lambda|^2)^{-2}}\right)\int_{\R^N}|\nabla w_\lambda|^2\\
&=&a^{\frac{1}{2}}+O(\lambda^{\frac{6-q}{2(q-2)}}).
\end{array}
$$
Thus, it follows from \eqref{e515}, \eqref{e523}, \eqref{e524} and \eqref{e525}  that  for $b\not=0$,
$$
\begin{array}{lcl}
\|\nabla u_\lambda\|_2^2&=&\lambda^{\frac{2N-q(N-2)}{2(q-2)}}(\sqrt{\varpi_\lambda})^{N-2}\int_{\R^N}|\nabla w_\lambda|^2\\
&=&\lambda^{\frac{6-q}{2(q-2)}}\sqrt{\varpi_\lambda}\int_{\R^N}|\nabla w_\lambda|^2\\
&=&
\left\{\begin{array}{lcl}
\lambda^{\frac{6-q}{2(q-2)}}\left(\frac{3(q-2)}{2q}a^{\frac{1}{2}}S_q^{\frac{q}{q-2}}-\Theta(\lambda^{\frac{p-q}{q-2}})\right),   \   \ &{\rm if}&  \  q>2p-6,\\
\lambda^{\frac{6-q}{2(q-2)}}\left(\frac{3(q-2)}{2q}a^{\frac{1}{2}}S_q^{\frac{q}{q-2}}+O(\lambda^{\frac{6-q}{2(q-2)}})\right),   \   \ &{\rm if}&  \  q\le 2p-6, \end{array} \right.
\end{array}
$$
$$
\begin{array}{rcl}
\|u_\lambda\|_2^2&=&\lambda^{\frac{4-N(q-2)}{2(q-2)}}(\sqrt{\varpi_\lambda})^N\int_{\R^N}|w_\lambda|^2\\
&=&\left\{\begin{array}{lcl}
\lambda^{\frac{10-3q}{2(q-2)}}\left(\frac{6-q}{2q}a^{\frac{3}{2}}S_q^{\frac{q}{q-2}}-\Theta(\lambda^{\frac{p-q}{q-2}})\right),   \   \ &{\rm if}&  \  q>2p-6,\\
\lambda^{\frac{10-3q}{2(q-2)}}\left(\frac{6-q}{2q}a^{\frac{3}{2}}S_q^{\frac{q}{q-2}}+O(\lambda^{\frac{6-q}{2(q-2)}})\right),   \   \ &{\rm if}&  \  q\le 2p-6, \end{array} \right.
\end{array}
$$
$$
\begin{array}{rcl}
\|u_\lambda\|_q^q&=&\lambda^{\frac{2N-q(N-2)}{2(q-2)}}(\sqrt{\varpi_\lambda})^{N}\int_{\R^N}|w_\lambda|^q\\
&=&\left\{\begin{array}{lcl}
\lambda^{\frac{6-q}{2(q-2)}}\left(a^{\frac{3}{2}}S_q^{\frac{q}{q-2}}-\Theta(\lambda^{\frac{p-q}{q-2}})\right),   \   \ &{\rm if}&  \  q>2p-6,\\
\lambda^{\frac{6-q}{2(q-2)}}\left(a^{\frac{3}{2}}S_q^{\frac{q}{q-2}}+O(\lambda^{\frac{6-q}{2(q-2)}})\right),   \   \ &{\rm if}&  \  q\le 2p-6. \end{array} \right.
\end{array}
$$
For $b=0$, by \eqref{e523}, \eqref{e524} and \eqref{e525}, we have 
$$
\|\nabla u_\lambda\|_2^2=\lambda^{\frac{6-q}{2(q-2)}}\sqrt{\varpi_\lambda}\int_{\R^N}|\nabla w_\lambda|^2
=\lambda^{\frac{6-q}{2(q-2)}}\left(\frac{3(q-2)}{2q}a^{\frac{1}{2}}S_q^{\frac{q}{q-2}}-\Theta(\lambda^{\frac{p-q}{q-2}})\right),
$$
$$
\|u_\lambda\|_2^2=\lambda^{\frac{10-3q}{2(q-2)}}(\sqrt{\varpi_\lambda})^3\int_{\R^N}|w_\lambda|^2
=\lambda^{\frac{10-3q}{2(q-2)}}\left(\frac{6-q}{2q}a^{\frac{3}{2}}S_q^{\frac{q}{q-2}}-\Theta(\lambda^{\frac{p-q}{q-2}})\right),
$$
$$
\|u_\lambda\|_q^q=\lambda^{\frac{6-q}{2(q-2)}}(\sqrt{\varpi_\lambda})^{3}\int_{\R^N}|w_\lambda|^q
=\lambda^{\frac{6-q}{2(q-2)}}\left(a^{\frac{3}{2}}S_q^{\frac{q}{q-2}}-\Theta(\lambda^{\frac{p-q}{q-2}})\right).
$$
The proof of Theorem 2.2 is complete.

\section{A connection with the mass constrained problem}

It is clear that if $u_\lambda\in H^1(\mathbb R^N)$ is a ground state of $(P_\lambda)$, and for some $c>0$ there holds
\begin{equation}\label{e221}
M(\lambda)=\|u_\lambda\|_2^2=c^2,
\end{equation}
then $u_\lambda$ is a positive normalized solution of \eqref{e15} with $\mu=1$.  We denote this normalized solution by  a pair $(u_{\lambda_c},\lambda_c)$. 
In what follows, we always assume $\mu=1$ in \eqref{e15}.  As a consequence of Theorem 2.1, we have the following

{\bf Proposition 6.1.}  {\em Let $p=2^*$, $2<q<2^*$ and $bS^2<1$ if $N=4$. Then the following statements hold true:

If $N=4$ and $q\in (3,4)$, or $N=3$ and $q\in (4,6)$,  then for any  $c>0$ the problem \eqref{e15} has at least one positive normalized solution $(u_{\lambda_c},\lambda_c)$ with $\lambda_c>0$, $\lim_{c\to 0}\lambda_c=\infty$ and $\lim_{c\to \infty}\lambda_c=0$. If $N=3$ and $q\in (10/3,4]$, then there exists a constant $c_1>0$ such that for any $c>c_1$,  the problem \eqref{e15} has at least one positive normalized solution $(u_{\lambda_c},\lambda_c)$ with $\lambda_c>0$ and  $\lim_{c\to \infty}\lambda_c=0$. 

If $N=4$ and $q\in (2,3)$,  then there exists $c_1>0$ such that for any  $c\in (0, c_1)$ the problem \eqref{e15} has at least two positive normalized solutions $(u_{\lambda_c^{(i)}},\lambda_c^{(i)})$ with $\lambda_c^{(i)}>0, i=1,2,$ $\lim_{c\to 0}\lambda_c^{(1)}=0$ and $\lim_{c\to 0}\lambda_c^{(2)}=+\infty$. 
If $N=3$ and $q\in (2,10/3)$, then there exists $c_1>0$ such that for any  $c\in (0,c_1)$ the problem \eqref{e15} has at least one  positive  normalized solution 
 $(u_{\lambda_c},\lambda_c)$ with $\lambda_c>0$ and $\lim_{c\to 0}\lambda_c=0$.

Furthermore,  if $q\not=2+\frac{4}{N}$ and $\lim \lambda_c=0$, then $\lim c^{\frac{1}{4-N(q-2)}}=0$ and  
$$
\lambda_c\simeq \left(\frac{2q}{2N-q(N-2)}a^{-\frac{N}{2}}S_q^{-\frac{q}{q-2}}c^2\right)^{\frac{2(q-2)}{4-N(q-2)}},
$$
$$
u_{\lambda_c}(0)\simeq V_0(0)\left(\frac{2q}{2N-q(N-2)}a^{-\frac{N}{2}}S_q^{-\frac{q}{q-2}}c^2\right)^{\frac{2}{4-N(q-2)}},
$$
$$
\|\nabla u_{\lambda_c}\|_2^2\simeq \frac{N(q-2)}{2q}a^{\frac{N-2}{2}}S_q^{\frac{q}{q-2}}\left(\frac{2q}{2N-q(N-2)}a^{-\frac{N}{2}}S_q^{-\frac{q}{q-2}}c^2\right)^{\frac{2N-q(N-2)}{4-N(q-2)}},
$$
$$
\|u_{\lambda_c}\|_{2^*}^{2^*}\simeq a^{\frac{N}{2}}\|V_0\|_{2^*}^{2^*}\left(\frac{2q}{2N-q(N-2)}a^{-\frac{N}{2}}S_q^{-\frac{q}{q-2}}c^2\right)^{\frac{N(2N-2q(N-2))}{(N-2)(4-N(q-2))}},
$$
$$
\|u_{\lambda_c}\|_q^q\simeq a^{\frac{N}{2}}S_q^{\frac{q}{q-2}}\left(\frac{2q}{2N-q(N-2)}a^{-\frac{N}{2}}S_q^{-\frac{q}{q-2}}c^2\right)^{\frac{2N-q(N-2)}{4-N(q-2)}}.
$$
If $N=4$, $q\in (2,4)$ and $\lim\lambda_c=\infty$, then $\lim c=0$,  and 
$$
\lambda_c(\ln \lambda_c)^{\frac{4-q}{2}}\sim \frac{1}{c^{q-2}},
\quad 
u_{\lambda_c}(0)\sim \frac{1}{c^2}\ln \lambda_c,
$$
$$
\|\nabla u_{\lambda_c}\|_2^2=\frac{aS^2}{1-bS^2}-\Theta(c^{4-q}(\ln \lambda_c)^{-\frac{4-q}{2}}),
$$
$$
\|u_{\lambda_c}\|_{2^*}^{2^*}=\frac{a^2S^2}{(1-bS^2)^2}+O(c^{4-q}(\ln \lambda_c)^{-\frac{4-q}{2}}),
\quad
\|u_{\lambda_c}\|_q^q\sim c^{4-q}(\ln \lambda_c)^{-\frac{4-q}{2}},
$$
if $N=3$, $q\in (4,6)$ and $\lim \lambda_c=\infty$, then $\lim c=0$, and 
$$
\lambda_c\sim c^{-\frac{q-4}{q-2}},
\quad 
u_{\lambda_c}(0)\sim c^{-\frac{1}{2(q-2)}},
$$
$$
\|\nabla u_{\lambda_c}\|_2^2=\frac{bS^3+S^{\frac{3}{2}}\sqrt{b^2S^3+4a}}{2}-\Theta(c^{\frac{6-q}{2(q-2)}}),
$$
$$ 
\|u_{\lambda_c}\|_{2^*}^{2^*}= \frac{1}{8}(bS^2+S^{\frac{1}{2}}\sqrt{b^2S^3+4a})^3+O(c^{\frac{6-q}{2(q-2)}}),
\quad
\|u_{\lambda_c}\|_q^q\sim c^{\frac{6-q}{2(q-2)}}.
$$
The same conclusions hold true for $(u_{\lambda_c^{(i)}},\lambda_c^{(i)}), i=1,2$.}

In what follows, we consider the problem \eqref{e15} with subcritical nonlinearity. Let 
$$
g(w)=w^{p-1}+w^{q-1}, \quad w\ge 0, 
$$
then  for $2<q\le p<6$, $g(u)$ satisfies all the assumptions (G1)-(G3) in \cite{Jeanjean-5}. For $0<\Lambda_1<\Lambda_2<+\infty$, it follows from \cite[Corollary 3.2]{Jeanjean-5} that 
$$
\mathcal{W}_{\Lambda_1}^{\Lambda_2}=\left\{w\in H^1_{rad}(\mathbb R^3):\ \text{$w$ is a nonnegetive solution  of $(P_\lambda)$ with $a=1$, $b=0$, $\lambda\in [\Lambda_1,\Lambda_2]$}\right\}
$$ 
is compact in $H^1(\mathbb R^3)$, and hence is compact in $L^2(\mathbb R^3)$ and $D^{1,2}(\mathbb R^3)$. Set
$$
\mathcal{S}_0^\lambda=\{w\in H^1_{rad}(\mathbb R^3):  \ w  \  {\rm solves} \ (P_\lambda) \ {\rm with} \ a=1, b=0 \ {\rm and} \ w>0\}.
$$
Then the map
$\rho: \cup_{\lambda\in [\Lambda_1,\Lambda_2]}\mathcal S_0^\lambda \to (0,+\infty)$ defined by $\rho(w)=\|w\|^2_2$ is compact.
Therefore, there exist positive constants $C_i=C_i(\Lambda_1,\Lambda_2), i=1,2$ and $D_i=D_i(\Lambda_1,\Lambda_2), i=1,2$ such that
$$
0<C_1\le \|w\|_2^2\le C_2<+\infty, \quad 0<D_1\le \|\nabla w\|_2^2\le D_2<+\infty, \quad \forall w\in \cup_{\lambda\in [\Lambda_1,\Lambda_2]}\mathcal S_0^\lambda.
$$
Set 
$$
\mathcal{S}_b^\lambda=\{u\in H^1_{rad}(\mathbb R^3):  \ u  \  {\rm solves} \ (P_\lambda)  \ {\rm and} \ u>0\}.
$$
Then there exists an one-to-one correspondence through the rescaling 
\begin{equation}\label{e225+}
w(x)=u(\sqrt{\varpi}x), \quad 
\varpi=a+b\int_{\mathbb R^3}|\nabla u|^2=a+b\varpi^{\frac{1}{2}}\int_{\mathbb R^3}|\nabla w|^2
\end{equation}
between $\mathcal S_0^\lambda$ and $\mathcal S_b^\lambda$.  Clearly, by \eqref{e225+},  we have $\|w\|_2^2=\varpi^{-\frac{3}{2}}\|u\|_2^2$ and 
$$
\varpi^{\frac{1}{2}}=\frac{b+\sqrt{b^2\|\nabla w\|_2^4+4a}}{2}.
$$
Therefore, for any $u\in \cup_{\lambda\in [\Lambda_1,\Lambda_2]}\mathcal S_b^\lambda $, we have
$$
\frac{(b+\sqrt{b^2D_2^2+4a})^3C_1}{8}\le \|u\|_2^2\le \frac{(b+\sqrt{b^2D_2^2+4a})^3C_2}{8}.
$$
Thus arguing as in \cite{Jeanjean-5},  Theorem 2.2 implies the following result concerning the existence, non-existence and exact number of normalized solutions of \eqref{e15}, and their  precise asymptotic behavior as the parameter $c$ varies.

\smallskip
\noindent
{\bf Proposition 6.2.} 
{\it  Let $2<q<p<6$, $b>0$ and 
\begin{equation}\label{e222}
m_1:=\sqrt{\frac{6-q}{2q}}a^{\frac{3}{4}}S_q^{\frac{q}{2(q-2)}}, \quad m_2:=\frac{\sqrt{27b^3(p-2)^3(6-p)}}{4p^2}S_p^{\frac{2p}{p-2}}.
\end{equation}
If $q<10/3$ and $p<14/3$, then for any  $c>0$ the problem \eqref{e15} has at least one positive normalized solution $(u_{\lambda_c},\lambda_c)$ with $\lambda_c>0$, $\lim_{c\to 0}\lambda_c=0$ and $\lim_{c\to \infty}\lambda_c=+\infty$. Moreover, for sufficiently small $c>0$  and for sufficiently large $c>0$, the problem \eqref{e15} has exactly one positive normalized solution.

If $10/3<q<p<14/3$, then there exists $c_1>0$ such that for any  $c>c_1$ the problem \eqref{e15} has two positive  normalized solutions $(u_{\lambda_c^{(i)}},\lambda_c^{(i)})$ with $\lambda_c^{(i)}>0, i=1,2,$ $\lim_{c\to \infty}\lambda_c^{(1)}=0$ and $\lim_{c\to \infty}\lambda_c^{(2)}=+\infty$. 
 Moreover, if  $c>0$ is sufficiently large, the problem \eqref{e15} has exactly two positive normalized solutions,  and  if  $c>0$ is sufficiently small,   the problem \eqref{e15} has no normalized solution.

If $q<10/3$ and $p>14/3$, then there exists $c_1>0$ such that for any  $c\in (0, c_1)$ the problem \eqref{e15} has two positive normalized solutions $(u_{\lambda_c^{(i)}},\lambda_c^{(i)})$ with $\lambda_c^{(i)}>0, i=1,2,$ $\lim_{c\to 0}\lambda_c^{(1)}=0$ and $\lim_{c\to 0}\lambda_c^{(2)}=+\infty$. 
Moreover, if  $c>0$ is sufficiently small,  the problem \eqref{e15} has exactly two positive normalized solutions, and   if   $c>0$ is sufficiently large,  the problem \eqref{e15} has no normalized solution.

If $q>10/3$ and  $p>14/3$,  then for any  $c>0$ the problem \eqref{e15} has at least one positive normalized solution  $(u_{\lambda_c},\lambda_c)$ with $\lambda_c>0$, $\lim_{c\to 0}\lambda_c=+\infty$ and $\lim_{c\to \infty}\lambda_c=0$.  Moreover, for  sufficiently small $c>0$  and for sufficiently  large $c>0$, the problem \eqref{e15} has exactly one positive normalized solution.

If $q=10/3$ and  $p<14/3$,  then there exists positive number $c_1\in (0,m_1)$ such that  for 
any $c\in (c_1,m_1)$, the problem \eqref{e15} has at least  two normalized solutions $(u_{\lambda_c^{(i)}},\lambda_c^{(i)})$ with $\lambda_c^{(i)}>0,$ 
$\lim_{c\to m_1}\lambda_c^{(1)}=0$ and $\lim_{c\to  m_1}\lambda_c^{(2)}>0$,  and for any $c\ge m_1$, the problem \eqref{e15}  has at least one positive  normalized solution  $(u_{\lambda_c},\lambda_c)$ with $\lambda_c>0$ and $\lim_{c\to \infty}\lambda_c=+\infty$.   Moreover, if $c>0$ is sufficiently large, the problem \eqref{e15} has exactly one positive normalized solution,  and if  $c>0$ is sufficiently small,  the problem \eqref{e15}   has  no normalized solution.

If $q=10/3$ and  $p>14/3$,  then there exists a positive number $c_1\ge m_1$ such that  for 
any $c\in (0,c_1)$, the problem \eqref{e15} has at least  one positive normalized solution $(u_{\lambda_c},\lambda_c)$ with $\lambda_c>0$, $\lim_{c\to m_1}\lambda_c=0$ and $\lim_{c\to 0}\lambda_c=+\infty$. 
Moreover, if  $c>0$ is sufficiently small, the problem \eqref{e15} has exactly one positive normalized solution, and  if  $c>0$ is sufficiently large, the problem \eqref{e15} has  no normalized solution.

If $q<10/3$ and  $p=14/3$,  then there exists   $c_1\ge m_2$ such that  for 
any $c\in (0,c_1)$, the problem \eqref{e15} has at least  one positive  normalized solution $(u_{\lambda_c},\lambda_c)$ with $\lambda_c>0,$ $\lim_{c\to 0 }\lambda_c=0$ and $\lim_{c\to m_2}\lambda_c=+\infty$.  
Moreover, if  $c>0$ is sufficiently small, the problem \eqref{e15} has exactly one positive normalized solution, and  if $c>0$  is sufficiently large, the problem \eqref{e15}  has no normalized solution.

If $q>10/3$ and  $p=14/3$,  then there exists a positive number $c_1\le m_2$ such that  for 
any $c>c_1$, the problem \eqref{e15} has  at least one positive normalized solution  $(u_{\lambda_c},\lambda_c)$ with $\lambda_c>0$, $\lim_{c\to m_2}\lambda_c=+\infty$ and $\lim_{c\to\infty}\lambda_c=0$.
  Moreover, if $c>0$ is sufficiently large, the problem \eqref{e15} has exactly one positive normalized solution, and  if  $c>0$ is sufficiently small, the problem \eqref{e15}  has  no normalized solution.  
  
 Furthermore,
 if $q\not=\frac{10}{3}$ and  $\lim\lambda_c=0$, then $\lim c^{\frac{1}{10-3q}}=0$, and
$$
\lambda_c\simeq
\left(\frac{2q}{6-q}\right)^{\frac{2(q-2)}{10-3q}}a^{-\frac{3(q-2)}{10-3q}}S_q^{-\frac{2q}{10-3q}}c^{\frac{4(q-2)}{10-3q}},
$$
$$
u_{\lambda_c}(0)\sim c^{\frac{2}{10-3q}},
$$
$$
\|\nabla u_{\lambda_c}\|_2^2
\simeq\frac{3(q-2)}{2q}\left(\frac{2q}{6-q}\right)^{\frac{6-q}{10-3q}}a^{-\frac{4}{10-3q}}S_q^{-\frac{2q}{10-3q}}c^{\frac{2(6-q)}{10-3q}},
$$
$$
\|u_{\lambda_c}\|_q^q\simeq \left(\frac{2q}{6-q}\right)^{\frac{6-q}{10-3q}}a^{-\frac{3(q-2)}{10-3q}}S_q^{-\frac{2q}{10-3q}}c^{\frac{2(6-q)}{10-3q}}.
$$
If $p\not=\frac{14}{3}$ and $\lim\lambda_c=+\infty$, then $\lim  c^{\frac{1}{14-3p}}=+\infty$, and 
$$
\lambda_c\simeq
\left(\frac{16p^4}{27b^3(p-2)^3(6-p)}\right)^{\frac{p-2}{14-3p}}S_p^{-\frac{4p}{14-3p}}c^{\frac{2(p-2)}{14-3p}},
$$
$$
u_{\lambda_c}(0)\sim c^{\frac{2}{14-3p}},
$$
$$
\|\nabla u_{\lambda_c}\|_2^2\simeq \frac{9b(p-2)^2}{4p^2}\left(\frac{16p^4}{27b^3(p-2)^3(6-p)}\right)^{\frac{6-p}{14-3p}}S_p^{-\frac{2p}{14-3p}}c^{\frac{2(6-p)}{14-3p}},
$$
$$
\| u_{\lambda_c}\|^p_p
\simeq 
\frac{27b^3(p-2)^3}{8p^3}\left(\frac{16p^4}{27b^3(p-2)^3(6-p)}\right)^{\frac{2(6-p)}{14-3p}}S_p^{-\frac{4p}{14-3p}}c^{\frac{4(6-p)}{14-3p}}.
$$
If  $q=\frac{10}{3}$ and $\lim\lambda_c=0$, then 
$$\lim c=m_1=\sqrt{\frac{6-q}{2q}}a^{\frac{3}{4}}S_q^{\frac{q}{2(q-2)}},
$$
and 
$$
\lim u_{\lambda_c}(0)=\lim \|\nabla u_{\lambda_c}\|_2^2=
\lim \|u_{\lambda_c}\|_q^q=0.
$$
If $p=\frac{14}{3}$ and $\lim\lambda_c=+\infty$, then 
$$
\lim c
=m_2=\frac{\sqrt{27b^3(p-2)^3(6-p)}}{4p^2}S_p^{\frac{2p}{p-2}},
$$
and 
$$
\lim u_{\lambda_c}(0)=
\lim \|\nabla u_{\lambda_c}\|_2^2=\lim \| u_{\lambda_c}\|^p_p
=+\infty.$$
The same conclusions hold true for $(u_{\lambda_c^{(i)}},\lambda_c^{(i)}), i=1,2$.}

\smallskip
\noindent
{\bf Remark 6.1.}  In the case that $ b=0$ and $2<q\le p<6$, Jeanjean, Zhang and Zhong \cite{Jeanjean-5} obtain the existence, non-existence and multiplicity of positive normalized solutions to $(P_\lambda)$. The authors in \cite{Jeanjean-5} also obtain some asymptotic behavior of normalized solutions as the Lagrange multiplier $\lambda\to 0$ or $\lambda\to +\infty$. In fact,  by the discussion in Section 5, a direct computation shows that if $b=0$,  $p=q\not=\frac{10}{3}$, then we have 
$$
\lambda_c=\left(\frac{p}{6-p}\right)^{\frac{2(p-2)}{10-3p}}a^{-\frac{3(p-2)}{10-3p}}(S_p/2)^{-\frac{2p}{10-3p}}c^{\frac{4(p-2)}{10-3p}},
$$
$$
\|\nabla u_{\lambda_c}\|^2_2=\frac{3(p-2)}{p}(\frac{p}{6-p})^{\frac{6-p}{10-3p}}a^{-\frac{4}{10-3p}}\left(S_p/2\right)^{-\frac{2p}{10-3p}}c^{\frac{2(6-p)}{10-3p}},
$$
$$
\|u_{\lambda_c}\|_p^p
=(\frac{p}{6-p})^{\frac{6-p}{10-3p}}a^{-\frac{3(p-2)}{10-3p}}\left(S_p/2\right)^{-\frac{2p}{10-3p}}c^{\frac{2(6-p)}{10-3p}}.
$$
For $b=0$, $q<p\not=\frac{10}{3}$ and $\lim\lambda_c=+\infty$,  then we have $\lim c^{\frac{1}{10-3p}}=+\infty$ and
$$
\lambda_c\simeq (\frac{2p}{6-p})^{\frac{2(p-2)}{10-3p}}a^{-\frac{3(p-2)}{10-3p}}S_p^{-\frac{2p}{10-3p}}c^{\frac{4(p-2)}{10-3p}},
$$
$$
\|\nabla u_{\lambda_c}\|_2^2
\simeq \frac{3(p-2)}{2p}(\frac{2p}{6-p})^{\frac{6-p}{10-3p}}a^{-\frac{4}{10-3p}}S_p^{-\frac{2p}{10-3p}}c^{\frac{2(6-q)}{10-3p}},
$$
$$
\| u_{\lambda_c}\|^p_p
\simeq (\frac{2p}{6-p})^{\frac{6-p}{10-3p}}a^{-\frac{3(p-2)}{10-3p}}S_p^{-\frac{2p}{10-3p}}c^{\frac{2(6-p)}{10-3p}}.
$$
For $b=0$, $\frac{10}{3}\not=q<p$ and $\lim\lambda_c=0$, then we have $\lim c^{\frac{1}{10-3q}}=0$ and 
$$
\lambda_c\simeq (\frac{2q}{6-q})^{\frac{2(q-2)}{10-3q}}a^{-\frac{3(q-2)}{10-3q}}S_q^{-\frac{2q}{10-3q}}c^{\frac{4(q-2)}{10-3q}},
$$
$$
\|\nabla u_{\lambda_c}\|_2^2
\simeq \frac{3(q-2)}{2q}(\frac{2q}{6-q})^{\frac{6-q}{10-3q}}a^{-\frac{4}{10-3q}}S_q^{-\frac{2q}{10-3q}}c^{\frac{2(6-q)}{10-3q}},
$$
$$
\|u_{\lambda_c}\|_q^q
\simeq (\frac{2q}{6-q})^{\frac{6-q}{10-3q}}a^{-\frac{3(q-2)}{10-3q}}S_q^{-\frac{2q}{10-3q}}c^{\frac{2(6-q)}{10-3q}}.
$$ 
We mention that Zeng et al. \cite{Zeng-2} extend the results in \cite{Jeanjean-5} to a Kirchhoff equation with general subcritical nonlinearity and obtain some results concerning the existence, non-existence and multiplicity of normalized solutions,  but the exact number  and the precise asymptotic expression of normalized solutions are not addressed there.

\smallskip
\noindent
{\bf Remark 6.2.}   
By Proposition  6.2 and Remark 6.1, we see that in the space dimension $N=3$,  there is a  striking difference between the cases $b=0$ and $b\not=0$ ( See also Figure 1 and Figure 2 ).  More precisely, if $b=0$ then $p_0=\frac{10}{3}$ plays a key role in the existence, non-existence, multiplicity and asymptotic behavior of normalized solutions of \eqref{e15}. 
However, if $b\not=0$, then both $p_0=\frac{10}{3}$ and $p_b=\frac{14}{3}$ play a role in the existence, non-existence, multiplicity and asymptotic behavior of normalized solutions of \eqref{e15}, 
which are completely different from those
for the corresponding nonlinear Schr\"odinger equation and which reveal the special influence of the nonlocal term. We mention that the difference between the Kirchhoff equations with pure power nonlinearity and nonlinear Schr\"odinger equations has also been observed by  Qi and Zou \cite{Qi-1}. But the difference between the Kirchhoff equations with combined  powers nonlinearity and nonlinear Schr\"odinger equations have not been addressed there.

\smallskip
\noindent
{\bf Remark 6.3.} 
Asymptotic behavior of $M(\lambda)$ similar to the cases depicted in Figure 3 and Figure 4 have been observed in nonlinear Sch\"odinger equations with a power nonlinearity, the cases depicted in Figure 3, Figure 4 and Figure 6 have been observed in nonlinear Schr\"odinger equations with general nonlinearity \cite{Jeanjean-5},  while the cases depicted in Figure 3, Figure 4 and Figure 5 have been observed in Kirchhoff equations with a pure power nonlinearity \cite{Qi-1}. If $q=10/3$ and  $p<14/3$,  then there exists positive number $c_1\in (0,m_1)$ such that  for 
any $c\in (c_1,m_1)$, the problem \eqref{e15} has at least  two normalized solutions,  if $c>0$ is sufficiently large, the problem \eqref{e15} has exactly one positive normalized solution,  and if  $c>0$ is sufficiently small,  the problem \eqref{e15}   has  no normalized solution.  This is new  phenomenon, which is not observed before in the literature   and which  does not shared by nonlinear Schr\"odinger equations and Kirchhoff equations with pure power nonlinearity.  Some new phenomenon is also observed in the case  $\frac{10}{3}<q<p=\frac{14}{3}$. See the diagrams of $M(\lambda)$  given below in Figure 3-Figure 8, where $m_1$ and $m_2$ are given in \eqref{e222}. 

\begin{figure}[h]
	\centering
	\includegraphics[width=0.85\linewidth]{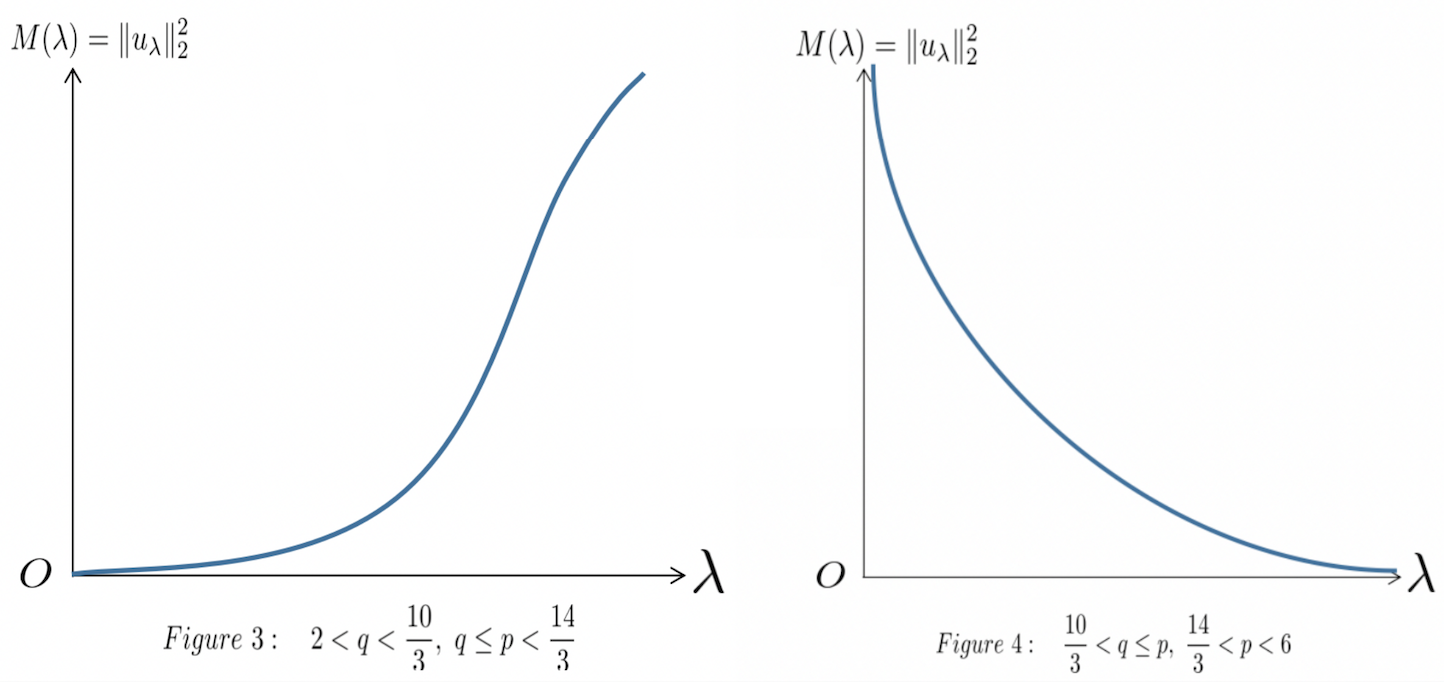}\\
	\includegraphics[width=0.85\linewidth]{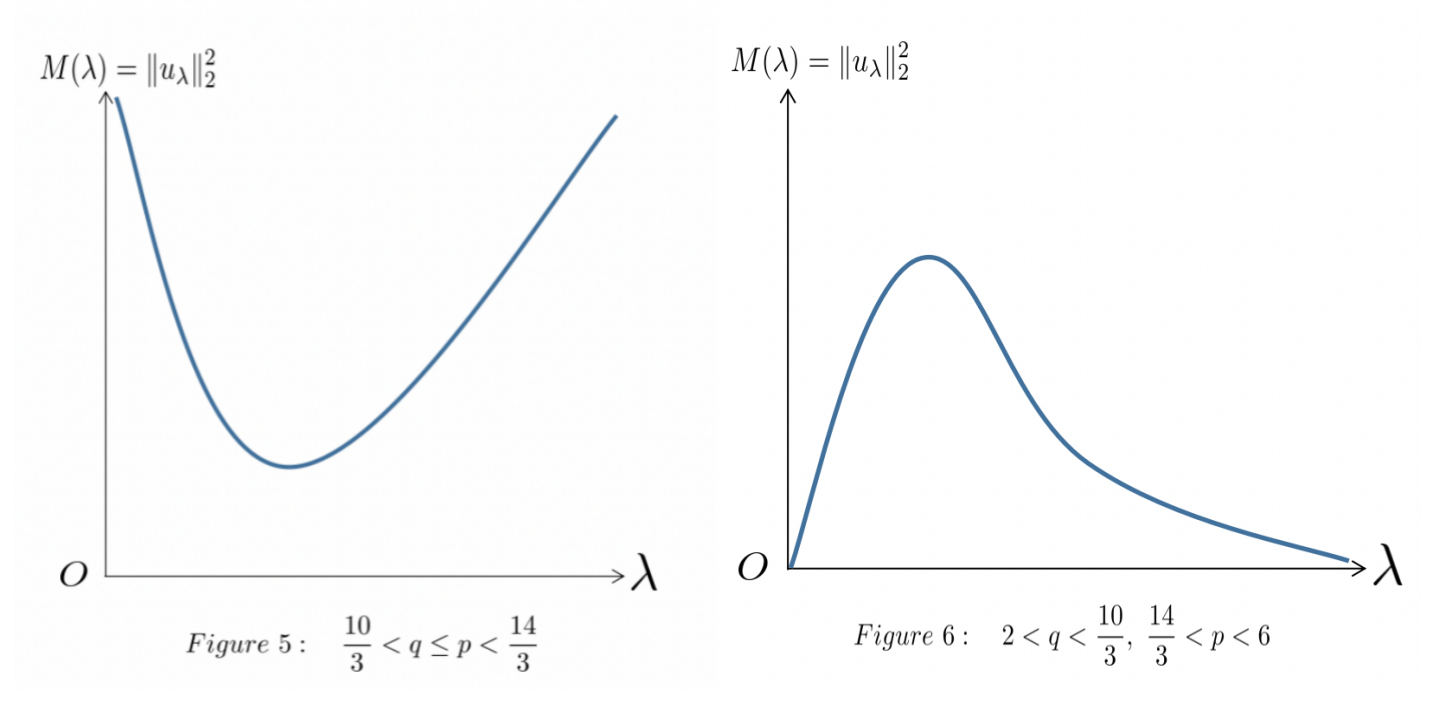}\\
	\includegraphics[width=0.85\linewidth]{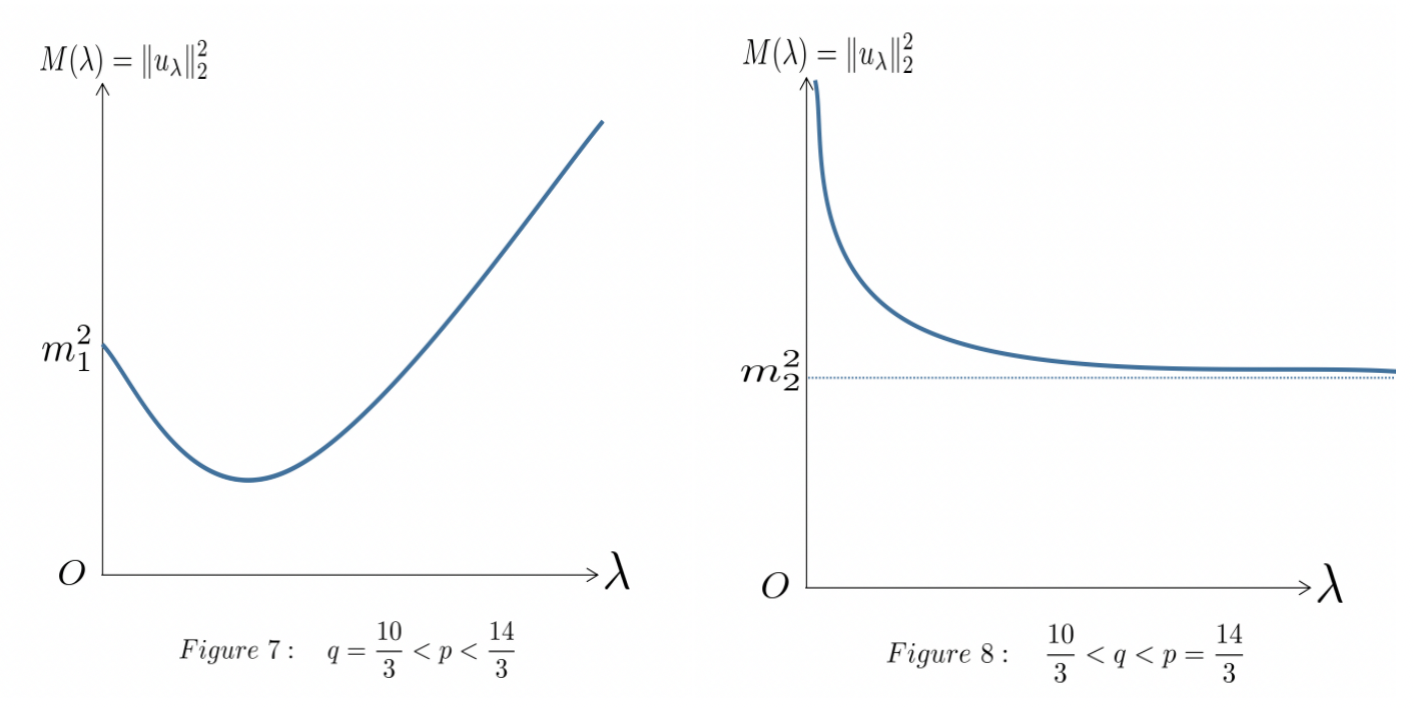}
\end{figure}

We mention that when $2<q<\frac{10}{3}, \frac{14}{3}<p<6$, a nontrivial  variation of $M(\lambda)$ can also be observed in Figure 6, which affects the existence, non-existence and multiplicity of normalized solutions of \eqref{e15}. This special behavior of $M(\lambda)$   is mainly caused by the combined nonlinearity and have been observed in nonlinear Schr\"odinger equations and Kirchhoff equations with general subcritical nonlinearity \cite{Jeanjean-5,  He-1}. We also mention that this type of behavior of $M(\lambda)$ does not appear in the Kirchhoff equations with a pure power nonlinearity \cite{Qi-1}. Typically, the asymptotic behavior of $M(\lambda)$ depicted in Figure 5 is mainly caused by the appearance 
of the nonlocal term $b\int_{\mathbb R^N}|\nabla u|^2$,  which has been reported by Qi and Zou \cite{Qi-1} as a new phenomenon for Kirchhoff equation with a pure power nonlinearity. 
While the asymptotic behaviors of $M(\lambda)$ depicted in Figure 7 and Figure 8 are mainly caused by the combined effect of the nonlocal term and the combined nonlinearity, which have not been reported before in the literature.

  Besides, the value of $b>0$ has also an effect on the existence, non-existence and the number of normalized solutions of \eqref{e15}, which can be seen from Figure 8 and the following diagrams.
  
  \begin{figure}[h]
  	\centering
  	\includegraphics[width=1.0\linewidth]{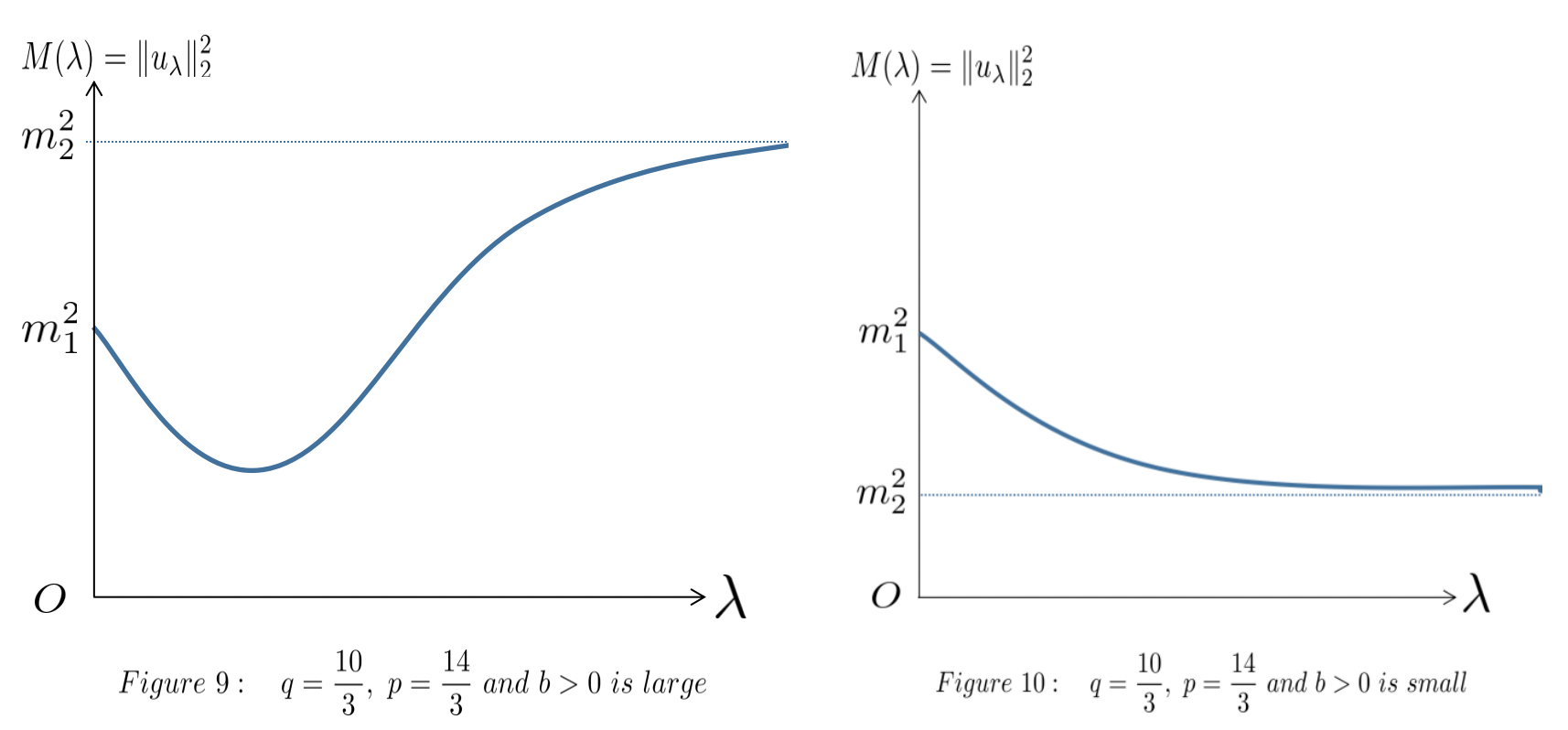}
    \end{figure}

\smallskip
{\small
\noindent {\bf Acknowledgements.} 
Part of this research was carried out while S.M. was visiting Swansea University. S.M. thanks the Department of Mathematics for its hospitality.
S.M. was supported by National Natural Science Foundation of China
(Grant Nos.11571187, 11771182)

\end{document}